\documentclass[11pt,twoside]{article}
\usepackage{amsfonts,epsfig,graphicx,subfigure}
\usepackage{amsmath,amssymb,amsthm}
\usepackage{fullpage}
\usepackage{epic,eepic}

\usepackage{epsf}
\usepackage{fancyhdr}
\usepackage{graphics}
\usepackage{amsfonts}
\usepackage{amsmath}
\usepackage{amssymb}
\usepackage{psfrag}
\usepackage{enumerate}

\RequirePackage[dvips]{hyperref}

\newcommand{\var}{\operatorname{var}}

\newcommand{\sign}{\operatorname{sign}}

\providecommand{\abs}[1]{|#1|}
\providecommand{\norm}[1]{ \| #1 \|}

\DeclareMathOperator{\diag}{diag}
\theoremstyle{plain}
\newtheorem{theorem}{Theorem}
\newtheorem{corollary}{Corollary}
\newtheorem{lemma}{Lemma}

\theoremstyle{definition}

\theoremstyle{remark}

\providecommand{\betamin}{B^*_{\operatorname{min}}}

\providecommand{\relaxn}{\lambda_n} 
\providecommand{\proj}{\Pi}

\providecommand{\mata}{X}

\providecommand{\matb}{\widebar{X}}

\providecommand{\noiseb}{\widebar{w}}

\providecommand{\duala}{\dwit{1}}
\providecommand{\dualb}{\dwit{2}}

\providecommand{\noisea}{w}
\providecommand{\overlap}{\alpha}

\providecommand{\define}{\mathrel{\mathop:}=}

\newcommand{\orpar}{\ensuremath{\theta_{1,\infty}}}
\newcommand{\numobs}{\ensuremath{n}}
\newcommand{\pdim}{\ensuremath{p}}
\newcommand{\spindex}{\ensuremath{s}}

\newcommand{\real}{\ensuremath{\mathbb{R}}}
\newcommand{\Sset}{\ensuremath{S}}

\newcommand{\orlas}{\ensuremath{\theta_{\operatorname{Las}}}}

\newcommand{\defn}{\ensuremath{: \, =}}

\newcommand{\SignSup}{\ensuremath{\mathbb{\Sset}_{\pm}}}

\newcommand{\SignPri}{\ensuremath{\mathbb{\Sset}_{\operatorname{pri}}}}
\newcommand{\SignDua}{\ensuremath{\mathbb{\Sset}_{\operatorname{dua}}}}

\newcommand{\regpar}{\ensuremath{\lambda_n}}
\newcommand{\estim}[1]{\ensuremath{\widehat{#1}}}

\newcommand{\mytermone}{T}

\newcommand{\adjusta}{\bvecstar{1}_\Both}
\newcommand{\adjustb}{\bvecstar{2}_\Both}

\newcommand{\are}{r}

\newlength{\widebarargwidth}
\newlength{\widebarargheight}
\newlength{\widebarargdepth}
\DeclareRobustCommand{\widebar}[1]{%
  \settowidth{\widebarargwidth}{\ensuremath{#1}}%
  \settoheight{\widebarargheight}{\ensuremath{#1}}%
  \settodepth{\widebarargdepth}{\ensuremath{#1}}%
  \addtolength{\widebarargwidth}{-0.3\widebarargheight}%
  \addtolength{\widebarargwidth}{-0.3\widebarargdepth}%
  \makebox[0pt][l]{\hspace{0.3\widebarargheight}%
    \hspace{0.3\widebarargdepth}%
    \addtolength{\widebarargheight}{0.3ex}%
    \rule[\widebarargheight]{0.95\widebarargwidth}{0.1ex}}%
          {#1}}

\newcommand{\matsnorm}[2]{|\!|\!| #1 | \! | \!|_{{#2}}}

\newenvironment{carlist} {\begin{list}{} {\setlength{\topsep}{0in}
               \setlength{\partopsep}{0in} \setlength{\parsep}{0in}
               \setlength{\itemsep}{\parskip}
               \setlength{\leftmargin}{0.12in}
               \setlength{\rightmargin}{0.08in}
               \setlength{\listparindent}{0.05in}
               \setlength{\labelwidth}{0.08in}
               \setlength{\labelsep}{0.1in}
               \setlength{\itemindent}{.1in}}} {\end{list}}
               
\newcommand{\bcar}{\begin{carlist}}

\newcommand{\ecar}{\end{carlist}}

\newcommand{\widgraph}[2]{\includegraphics[keepaspectratio,width=#1]{#2}}

\newcommand{\Prob}{\ensuremath{\mathbb{P}}}
\newcommand{\Exs}{\ensuremath{\mathbb{E}}}

\newcommand{\Joint}{\ensuremath{U}}
\newcommand{\Jointcom}{\ensuremath{{U^{c}}}}

\newcommand{\Both}{\ensuremath{B}}
\newcommand{\Bothcom}{\ensuremath{{B^{c}}}}

\newcommand{\Sing}{\ensuremath{\Bothcom}}

\newcommand{\Uvar}{\ensuremath{\Delta}}

\newcommand{\Vvar}{\ensuremath{V}}

\newcommand{\Event}{\mathcal{E}}
\newcommand{\Tail}{\mathcal{T}}

\newcommand{\lammin}{\ensuremath{\lambda_{\operatorname{min}}}}
\newcommand{\lammax}{\ensuremath{\lambda_{\operatorname{max}}}}

\newcommand{\myvar}{\ensuremath{\operatorname{var}}}

\newcommand{\mprob}{\Prob}

\newcommand{\sahand}{\ensuremath{f}}
\newcommand{\sahandtil}{\ensuremath{\widetilde{\sahand}}}

\newcommand{\Smat}{\ensuremath{M}}
\newcommand{\Smattil}{\ensuremath{\widetilde{\Smat}}}

\newcommand{\Signmat}{\ensuremath{S}}
\newcommand{\Signmattil}{\ensuremath{\widetilde{\Signmat}}}

\newcommand{\Proj}[1]{\ensuremath{\Pi_{#1}}}
\newcommand{\myproj}[1]{\Proj{#1}}

\newcommand{\order}{\ensuremath{\mathcal{O}}}

\newcommand{\myorder}{\ensuremath{\order(\sqrt{\spindex/\numobs})}}

\newcommand{\edist}{\ensuremath{\stackrel{d}{=}}}

\newcommand{\randvar}{\ensuremath{\sigma^2}}
\newcommand{\randvartil}{\ensuremath{\widetilde{\sigma}^2}}

\newcommand{\Uterma}{\ensuremath{T_a}}
\newcommand{\Utermb}{\ensuremath{T_b}}

\newcommand{\Smatd}{\ensuremath{D}}
\newcommand{\Quni}{\ensuremath{Q}}

\newcommand{\Specmat}{\ensuremath{R}}
\newcommand{\Amat}{\ensuremath{A}}

\newcommand{\vecone}{\ensuremath{\vec{1}}}

\newcommand{\mindex}{\ensuremath{m}}
\newcommand{\martin}{\ensuremath{v}}

\newcommand{\qvec}{\ensuremath{q}}

\newcommand{\mineig}{\ensuremath{C}}
\newcommand{\Covmat}{\Sigma}

\newcommand{\Sail}{\ensuremath{\mathcal{S}}}

\newcommand{\Aevent}{\ensuremath{\mathbb{A}}}

\newcommand{\incopar}{\ensuremath{\gamma}}

\newcommand{\jind}{\ensuremath{k}} 
\newcommand{\kind}{\ensuremath{j}}

\newcommand{\keypar}{\ensuremath{\xi}}

\newcommand{\Cmin}{\ensuremath{C_{\min}}}

\newcommand{\Dmax}{\ensuremath{D_{\max}}}

\newcommand{\Mterm}{\ensuremath{M}}

\newcommand{\Kball}{\ensuremath{\mathbb{K}}}

\newcommand{\inprod}[2]{\ensuremath{\langle #1 , \, #2 \rangle}}
\newcommand{\binprod}[2]{\ensuremath{\big\langle #1 , \, #2 \big\rangle}}

\newcommand{\phiprob}{\ensuremath{\phi_1(\keypar, \pdim, \spindex)}}
\newcommand{\phiprobtwo}{\ensuremath{\phi_2(\tautwo, \keypar, \numobs,
\pdim, \spindex)}}

\newcommand{\Bou}{\ensuremath{b_1(\keypar, \relaxn, \numobs, \spindex)}}
\newcommand{\Boutwo}{\ensuremath{b_2(\keypar, \relaxn, \numobs, \spindex)}}

\theoremstyle{definition}
\newtheorem{asss}{Assumption}
\newcommand{\bass}{\begin{asss}}
\newcommand{\eass}{\end{asss}}

\newcommand{\Xmat}[1]{\ensuremath{X^{#1}}}
\newcommand{\Xmatf}[2]{\ensuremath{X^{#1}_{#2}}}
\newcommand{\Xmatt}[1]{\ensuremath{(X^{#1})^T}}
\newcommand{\Xmatft}[2]{\ensuremath{(X^{#1}_{#2})^T}}

\newcommand{\yobs}[1]{\ensuremath{y^{#1}}}
\newcommand{\wnoise}[1]{\ensuremath{w^{#1}}}

\newcommand{\Best}{\ensuremath{\widehat{B}}}
\newcommand{\Bstar}{\ensuremath{\widebar{B}}}

\newcommand{\Bwit}{\ensuremath{\widetilde{B}}}
\newcommand{\bwit}[1]{\ensuremath{\widetilde{\beta}^{#1}}}

\newcommand{\bvec}[1]{\ensuremath{\beta^{#1}}}
\newcommand{\bvecstar}[1]{\ensuremath{\widebar{\beta}^{#1}}}
\newcommand{\bstar}[1]{\ensuremath{\widebar{\beta}^{\, #1}}}
\newcommand{\best}[1]{\ensuremath{\widehat{\beta}^{\,#1}}}

\newcommand{\numreg}{\ensuremath{r}}
\newcommand{\pind}{\ensuremath{i}}
\newcommand{\pindtwo}{\ensuremath{j}}

\newcommand{\Dwit}{\ensuremath{\widetilde{Z}}}
\newcommand{\Dwitf}[1]{\ensuremath{\widetilde{Z}_{#1}}}
\newcommand{\Bwitf}[1]{\ensuremath{\widetilde{B}_{#1}}}
\newcommand{\dwit}[1]{\ensuremath{\widetilde{z}^{\,#1}}}
\newcommand{\dwitf}[2]{\ensuremath{\widetilde{z}^{\,#1}_{#2}}}

\newcommand{\bnorm}[3]{\ensuremath{\| #1\|_{\ell_{#2}/\ell_{#3}}}}
\newcommand{\bnorminf}[1]{\ensuremath{\bnorm{#1}{1}{\infty}}}
\newcommand{\bnormq}[2]{\ensuremath{\bnorm{#1}{1}{#2}}}

\newcommand{\Nuvar}[1]{\ensuremath{\Delta^{#1}}}

\newcommand{\Uvarf}[2]{\ensuremath{\Uvar^{#1}_{#2}}}
\newcommand{\Covmati}[1]{\ensuremath{\Sigma^{#1}}}
\newcommand{\Covmatif}[2]{\ensuremath{\Sigma^{#1}_{#2}}}
\newcommand{\Wnew}[2]{\ensuremath{W^{#1}_{#2}}}

\newcommand{\Yif}[2]{\ensuremath{Y^{#1}_{#2}}}

\newcommand{\mysign}{\ensuremath{b}}

\newcommand{\tautwo}{\ensuremath{\kappa}}

\newcommand{\qpar}{\ensuremath{q}}

\newcommand{\Mset}{\ensuremath{\mathbb{M}}}

\newcommand{\delpar}{\ensuremath{\delta}}

\newcommand{\mybetamin}{\ensuremath{\Bstar_{\operatorname{min}}}}
\newcommand{\mybetagap}{\ensuremath{\Bstar_{\operatorname{gap}}}}

\newcommand{\bdiff}{\ensuremath{\Bstar_{\operatorname{diff}}}}

\newcommand{\mynumregspind}{\ensuremath{|\Joint|}}

\makeatletter
\long\def\@makecaption#1#2{
  \vskip 0.8ex
  \setbox\@tempboxa\hbox{\small {\bf #1:} #2}
  \parindent 1.5em  %% How can we use the global value of this???
  \dimen0=\hsize
  \advance\dimen0 by -3em
  \ifdim \wd\@tempboxa >\dimen0
  \hbox to \hsize{
    \parindent 0em
    \hfil
    \parbox{\dimen0}{\def\baselinestretch{0.96}\small
      {\bf #1.} #2
    }
    \hfil}
  \else \hbox to \hsize{\hfil \box\@tempboxa \hfil}
  \fi
}
\makeatother

\long\def\comment#1{}

\makeatletter
\def\@cite#1#2{[\if@tempswa #2 \fi #1]}
\makeatother

\begin{document}

%%%%%%% TITLE PAGE %%%%%%%%%%%%%%%%%%%%%%%%%%%%%%%%%%%%%%%%%%%%%%%%%%%

\begin{center}

{\bf{\LARGE{Simultaneous support recovery in high dimensions: Benefits
               and perils of block
               $\ell_1/\ell_\infty$-regularization}}}

\vspace*{.2in}

{\large{
\begin{tabular}{ccc}
  Sahand Negahban$^{\star}$ & &  Martin J. Wainwright$^{\dagger,\star}$ \\
\end{tabular}
}}

\vspace*{.2in}

\begin{tabular}{c}
Department of Statistics$^\dagger$, and \\
Department of Electrical Engineering and Computer Sciences$^\star$ \\
UC Berkeley,  Berkeley, CA  94720
\end{tabular}

\vspace*{.2in}

%\today
May 5, 2009
\vspace*{.2in}

\begin{tabular}{c}
Technical Report, \\
Department of Statistics,  UC Berkeley
\end{tabular}

\end{center}

%%%%%%%%%%%%%%%%%%%%%%%%%%%%%%%%%%%%%%%%%%%%%%%%%%%%%%%%%%%%%%%%%%%%%%%%%

\begin{abstract}

Given a collection of $r \geq 2$ linear regression problems in $\pdim$
dimensions, suppose that the regression coefficients share partially
common supports.  This set-up suggests the use of
$\ell_{1}/\ell_{\infty}$-regularized regression for joint estimation
of the $\pdim \times \numreg$ matrix of regression coefficients.  We
analyze the high-dimensional scaling of
$\ell_1/\ell_\infty$-regularized quadratic programming, considering
both consistency rates in $\ell_\infty$-norm, and also how the minimal
sample size $n$ required for performing variable selection grows as a
function of the model dimension, sparsity, and overlap between the
supports.  We begin by establishing bounds on the $\ell_\infty$-error
as well sufficient conditions for exact variable selection for fixed
design matrices, as well as designs drawn randomly from general
Gaussian matrices.  Our second set of results applies to $\numreg = 2$
linear regression problems with standard Gaussian designs whose
supports overlap in a fraction $\alpha \in [0,1]$ of their entries:
for this problem class, we prove that the
$\ell_{1}/\ell_{\infty}$-regularized method undergoes a phase
transition---that is, a sharp change from failure to
success---characterized by the rescaled sample size
$\theta_{1,\infty}(n, p, s, \alpha) = n/\{(4 - 3 \alpha) s \log(p-(2-
\alpha) \, s)\}$.  More precisely, given sequences of problems
specified by $(n, p, s, \alpha)$, for any $\delta > 0$, the
probability of successfully recovering both supports converges to $1$
if $\theta_{1, \infty}(n, p, s, \alpha) > 1+\delta$, and converges to
$0$ for problem sequences for which $\theta_{1,\infty}(n,p,s, \alpha)
< 1 - \delta$.  An implication of this threshold is that use of
$\ell_1 / \ell_{\infty}$-regularization yields improved statistical
efficiency if the overlap parameter is large enough ($\alpha > 2/3$),
but has \emph{worse} statistical efficiency than a naive Lasso-based
approach for moderate to small overlap ($\alpha < 2/3$).  Empirical
simulations illustrate the close agreement between these theoretical
predictions, and the actual behavior in practice.  These results
indicate that some caution needs to be exercised in the application of
$\ell_1/\ell_\infty$ block regularization: if the data does not match
its structure closely enough, it can impair statistical performance
relative to computationally less expensive schemes.\footnote{This work
was presented in part at the NIPS 2008 conference in Vancouver,
Canada, December 2008.  Supported in part by NSF grants DMS-0528488,
DMS-0605165, and CCF-0545862.}
\end{abstract}

\section{Introduction}
\label{SecIntro}

The area of high-dimensional statistical inference is concerned with
the behavior of models and algorithms in which the dimension $\pdim$
is comparable to, or possibly even larger than the sample size
$\numobs$.  In the absence of additional structure, it is well-known
that many standard procedures---among them linear regression and
principal component analysis---are not consistent unless the ratio
$\pdim/\numobs$ converges to zero.  Since this scaling precludes
having $\pdim$ comparable to or larger than $\numobs$, an active line
of research is based on imposing structural conditions on the data
(e.g., sparsity, manifold constraints, or graphical model structure),
and studying the high-dimensional consistency (or inconsistency) of
various types of estimators.

This paper deals with high-dimensional scaling in the context of
solving multiple regression problems, where the regression vectors are
assumed to have shared sparse structure.  More specifically, suppose
that we are given a collection of $\numreg$ different linear
regression models in $\pdim$ dimensions, with regression vectors
$\bstar{\pind} \in \real^\pdim$, for $\pind =1, \ldots, \numreg$.  We
let $\Sset(\bstar{\pind}) = \{j \, \mid \, \bstar{\pind}_j \neq 0 \}$
denote the support set of $\bstar{\pind}$.  In many
applications---among them sparse approximation, graphical model
selection, and image reconstruction---it is natural to impose a
sparsity constraint, corresponding to restricting the cardinality
$|\Sset(\bstar{\pind})|$ of each support set.  Moreover, one might
expect some amount of overlap between the sets $\Sset(\bstar{\pind})$
and $\Sset(\bstar{j})$ for indices $\pind \neq j$ since they
correspond to the sets of active regression coefficients in each
problem.  Let us consider some examples to illustrate:
\begin{itemize}
\item Consider the problem of image denoising or compression, say
using a wavelet transform or some other type of multiresolution
basis~\cite{Mallat}. It is well known that natural images tend to have
sparse representations in such bases~\cite{Simoncelli98e}. Moreover,
similar images---say the same scene taken from multiple
cameras---would be expected to share a similar subset of active
features in the reconstruction.  Consequently, one might expect that
using a block-regularizer that enforces such joint sparsity could lead
to improved image denoising or compression.
\item Consider the problem of identifying the structure of a Markov
network or graphical model~\cite{Jordan_book1} based on a collection
of samples (e.g., such as observations of a social network).  For
networks with a single parameter per edge (e.g., Gaussian
models~\cite{Meinshausen06}, Ising models~\cite{RavWaiLaf08}), a line
of recent work has shown that $\ell_1$-based methods can be successful
in recovering the network structure.  However, many graphical models
have multiple parameters per edge (e.g., for discrete models with
non-binary state spaces), and it is natural that the subset of
parameters associated with a given edge are zero (or non-zero) in a
grouped manner.  Thus, any method for recovering the graph structure
should impose a block-structured regularization that groups together
the subset of parameters associated with a single edge.
\item Finally, consider a standard problem in genetic analysis: given
a set of gene expression arrays, where each array corresponds to a
different patient but the same underlying tissue type (e.g., tumor),
the goal is to discover the subset of features relevant for tumorous
growths.  This problem can be expressed as a joint regression problem,
again with a shared sparsity constraint coupling together the
different patients.  In this context, the recent work of Liu et
al.~\cite{LiuLafWas08} shows that imposing additional structural
constraints can be beneficial (e.g., they are able to greatly reduce
the number of expressed genes while maintaining the same prediction
performance).
\end{itemize}

\noindent Given these structural conditions of shared sparsity in these and
other applications, it is reasonable to consider how this common
structure can be exploited so as to increase the statistical
efficiency of estimation procedures.

There is now a substantial and relatively mature body of work on
$\ell_1$-regularization for estimation of sparse models, dating back
to the introduction of the Lasso and basis
pursuit~\cite{Tibshirani96,Chen98}. With contributions from various
researchers
(e.g.,~\cite{DonTan06,Meinshausen06,Tropp06,Zhao06,BiRiTsy08}), there
is now a fairly complete theory of the behavior of the Lasso for
high-dimensional sparse estimation.  A more recent line of work
(e.g.,~\cite{Turlach05,Kim06,Obo07,Tro06,ZhaRoc06}), motivated by
applications in which block or hierarchical structure arises, has
proposed the use of block $\ell_{a,b}$ norms for various $a,b \in [1,
\infty]$.  Of particular relevance to this paper is the block
$\ell_1/\ell_\infty$ norm, proposed initially by Turlach et
al.~\cite{Turlach05} and Tropp et al.~\cite{Tro06}.  This form of
block regularization is a special case of the more general family of
composite or hierarchical penalties, as studied by Zhao et
al.~\cite{ZhaRoc06}.

Various authors have empirically demonstrated that block
regularization schemes can yield better performance for different data
sets~\cite{ZhaRoc06,Obo07,LiuLafWas08}.  Some recent work by
Bach~\cite{Bach08} has provided consistency results for
$\ell_1/\ell_2$ block-regularization schemes under classical scaling,
meaning that $\numobs \rightarrow +\infty$ with $\pdim$ fixed. Meier
et al.~\cite{Mei07} has established high-dimensional consistency for
the predictive risk of $\ell_1/\ell_2$ block-regularized logistic
regression.  The papers~\cite{Liu08,NarRin08,RavLiuLafWas08} have
provided high-dimensional consistency results for $\ell_1/\ell_q$
block regularization for support recovery using fixed design matrices,
but the rates do not provide sharp differences between the case $q =
1$ and $q > 1$.  

To date, there has a relatively limited amount of theoretical work
characterizing if and when the use of block regularization schemes
actually leads to gains in statistical efficiency.  As we elaborate
below, this question is significant due to the greater computational
cost involved in solving block-regularized convex programs.  In the
case of $\ell_1/\ell_2$ regularization, concurrent work by Obozinski
et al~\cite{OboWaiJor08} (involving a subset of the current authors)
has shown that that the $\ell_1/\ell_2$ method can yield statistical
gains up to a factor of $r$, the number of separate regression
problems; more recent concurrent work~\cite{HuaZha09,Lou09} has
provided related high-dimensional consistency results for
$\ell_1/\ell_2$ regularization, emphasizing the gains when the number
of tasks $r$ is much larger than $\log \pdim$.

  This paper considers this issue in the context of variable selection
using block $\ell_1/\ell_\infty$ regularization.  Our main
contribution is to obtain some precise---and arguably
surprising---insights into the benefits and dangers of using block
$\ell_1/\ell_\infty$ regularization, as compared to simpler
$\ell_1$-regularization (separate Lasso for each regression
problem). We begin by providing a general set of sufficient conditions
for consistent support recovery for both fixed design matrices, and
random Gaussian design matrices.  In addition to these basic
consistency results, we then seek to characterize rates, for the
particular case of standard Gaussian designs, in a manner precise
enough to address the following questions:
\begin{enumerate}
\item[(a)] First, under what structural assumptions on the data does
the use of $\ell_1/\ell_\infty$ block-regularization provide a
quantifiable reduction in the scaling of the sample size $\numobs$, as
a function of the problem dimension $\pdim$ and other structural
parameters, required for consistency?
\item[(b)] Second, are there any settings in which
$\ell_1/\ell_\infty$ block-regularization can be harmful relative to
computationally less expensive procedures?
\end{enumerate}
Answers to these questions yield useful insight into the
\emph{tradeoff between computational and statistical efficiency} in
high-dimensional inference.  Indeed, the convex programs that arise
from using block-regularization typically require a greater
computational cost to solve.  Accordingly, it is important to
understand under what conditions this increased computational cost
guarantees that fewer samples are required for achieving a fixed level
of statistical accuracy.  

The analysis of this paper gives conditions on the designs and
regression matrix $\Bstar$ for which $\ell_1/\ell_\infty$ yields
improvements (question (a)), and also shows that if there is
sufficient mismatch between the regression matrix $\Bstar$ and the
$\ell_1/\ell_\infty$ norm, then use of this regularizer actually
impairs statistical efficiency relative to a naive $\ell_1$-approach.
As a representative instance of our theory, consider the special case
of standard Gaussian design matrices and two regression problems
($\numreg = 2$), with the supports $\Sset(\bstar{1})$ and
$\Sset(\bstar{2})$ each of size $\spindex$ and overlapping in a
fraction $\overlap \in [0,1]$ of their entries.  For this problem, we
prove that block $\ell_1/\ell_\infty$ regularization undergoes a phase
transition---meaning a \emph{sharp threshold between success and
recovery}---that is specified by the rescaled sample size
\begin{eqnarray}
\label{EqnDefnOrder}
\orpar(\numobs, \pdim, \spindex, \overlap) & \defn & \frac{\numobs}{(4
  - 3 \overlap) \spindex \log(p-(2-\overlap) \spindex)}.
\end{eqnarray}
In words, for any $\delpar > 0$ and for scalings of the quadruple
$(\numobs, \pdim, \spindex, \overlap)$ such that $\orpar \geq 1 +
\delpar$, the probability of successfully recovering both
$\Sset(\bstar{1})$ and $\Sset(\bstar{2})$ converges to one, whereas
for scalings such that $\orpar \leq 1 - \delpar$, the probability of
success converges to zero.

Figure~\ref{FigSims} illustrates how the theoretical
threshold~\eqref{EqnDefnOrder} agrees with the behavior observed in
practice.  This figure plots the probability of successful recovery
using the block $\ell_1/\ell_\infty$ approach versus the rescaled
sample size $\numobs/ \{2 \spindex \log[\pdim - (2-\overlap) \spindex
\}$; the results shown here are for $r=2$ regression parameters.  The
plots show twelve curves, corresponding to three different problem
sizes $\pdim \in \{128, 256, 512 \}$ and four different values of the
overlap parameter $\overlap \in \{0.1, 0.3, 0.7, 1 \}$.  First, let us
focus on the set of curves labeled with $\overlap = 1$, corresponding
to case of complete overlap between the regression vectors.  Notice
how the curves for all three problem sizes $\pdim$, when plotted
versus the rescaled sample size, line up with one another; this
``stacking effect'' shows that the rescaled sample size captures the
phase transition behavior.  Similarly, for other choices of the
overlap, the sets of three curves (over problem size $\pdim$) exhibit
the same stacking behavior.  Secondly, note that the results are
consistent with the theoretical prediction~\eqref{EqnDefnOrder}:
\begin{figure}[h]
  \begin{center}
    \begin{tabular}{c}
      \widgraph{.7\textwidth}{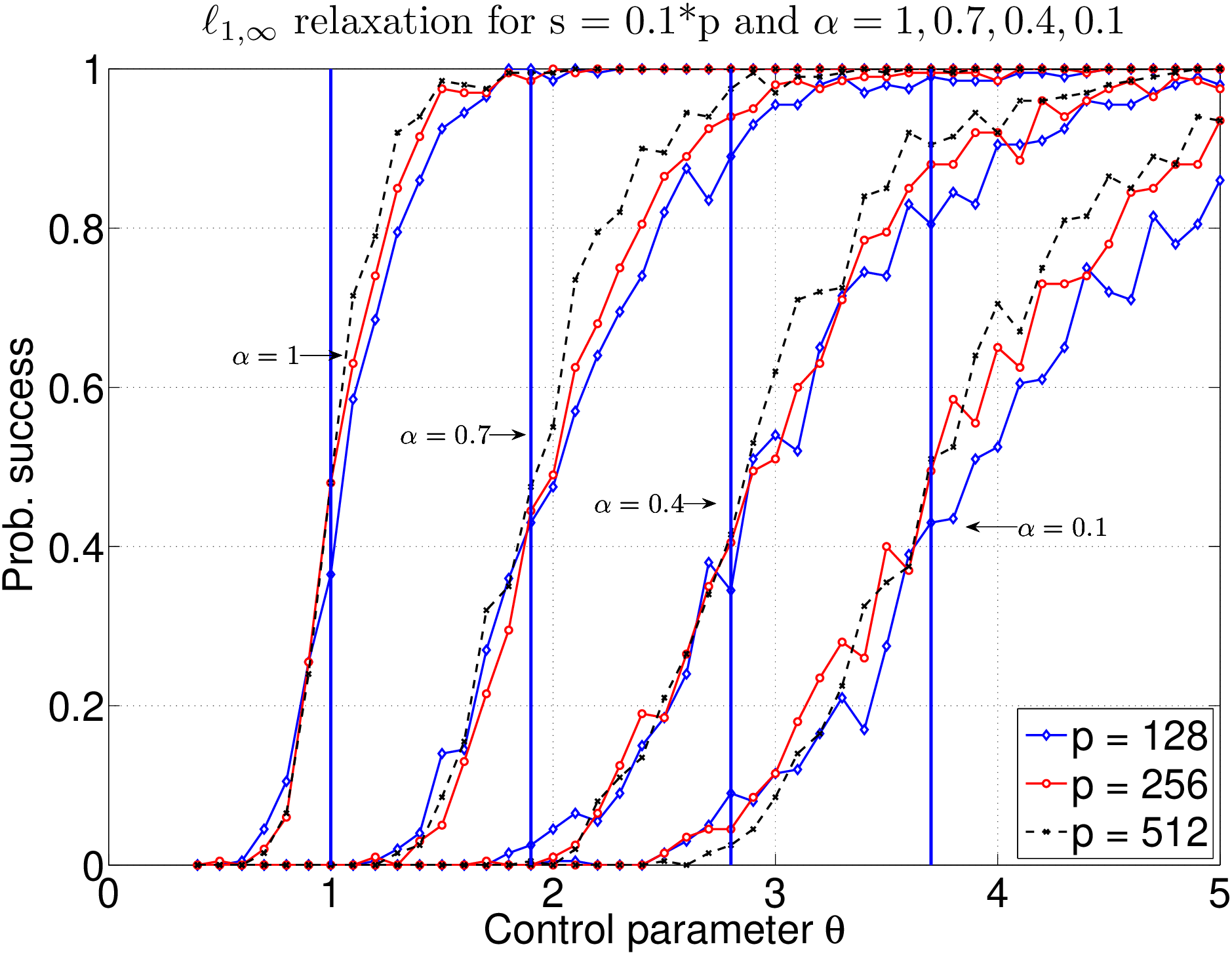}
    \end{tabular}
  \end{center}
\caption{Probability of success in recovering the joint signed
supports plotted against the rescaled sample size $\orlas \defn
\numobs/[2 \spindex \log(\pdim - (2-\overlap)\spindex))]$ for linear
sparsity $\spindex = 0.1 \pdim$.  Each stack of graphs corresponds to
a fixed overlap $\overlap$, as labeled on the figure.  The three
curves within each stack correspond to problem sizes $\pdim \in \{128,
256, 512 \}$; note how they all align with each other and exhibit
step-like behavior, consistent with Theorem~\ref{ThmPhase}.  The
vertical lines correspond to the thresholds $\orpar^*(\overlap)$
predicted by Theorem~\ref{ThmPhase}; note the close agreement between
theory and simulation.}
\label{FigSims}
\end{figure}
the stacks of curves shift to the right as the overlap parameter
$\overlap$ decreases from $1$ towards $0$, showing that problems with
less overlap require a larger rescaled sample size.  More interesting
is the sharpness of agreement in \emph{quantitative terms:} the
vertical lines in the center of each stack show the point at which our
theory~\eqref{EqnDefnOrder} predicts that the method should transition
from failure to success.

 By comparison to previous theory on the behavior of the Lasso
(ordinary $\ell_1$-regularized quadratic programming), the
scaling~\eqref{EqnDefnOrder} has two interesting implications.  For
the $\spindex$-sparse regression problem with standard Gaussian
designs, the Lasso has been shown~\cite{Wainwright06a} to transition
from success to failure as a function of the rescaled sample size
\begin{eqnarray}
\label{EqnLassoThresh}
\orlas(\numobs, \pdim, \spindex) & \defn & 
 \frac{\numobs}{2 \spindex \log(\pdim - \spindex)}.
\end{eqnarray}
In particular, under the conditions imposed here, solving two separate
Lasso problems, one for each regression problem, would recover both
supports for problem sequences $(\numobs, \pdim, \spindex)$ such that
$\orlas > 1$.  Thus, one consequence of our analysis is to
characterize the \emph{relative statistical efficiency} of
$\ell_1/\ell_\infty$ regularization versus ordinary
$\ell_1$-regularization, as described by the ratio $R \defn
\frac{\orpar}{\orlas}$.

Our theory predicts that (disregarding some $o(1)$ factors) the
relative efficiency scales as $R(\overlap) \sim \frac{4 - 3
\alpha}{2}$, which (as we show later) shows excellent agreement with
empirical behavior in simulation.  Our characterization of
$R(\overlap)$ confirms that if the regression matrix $\Bstar$ is
well-aligned with the block $\ell_1/\ell_\infty$ regularizer---more
specifically for overlaps $\overlap \in [\frac{2}{3}, 1]$---then
block-regularization increases statistical efficiency.  On the other
hand, our analysis also conveys a \emph{cautionary message}: if the
overlap is too small---more precisely, if $\overlap < 2/3$---then
block $\ell_{1,\infty}$ is actually relative to the naive Lasso-based
approach.  This fact illustrates that some care is required in the
application of block regularization schemes.

In terms of proof techniques, the analysis of this paper is
considerably more delicate than the analogous arguments required to
show support consistency for the
Lasso~\cite{Meinshausen06,Wainwright06a,Zhao06}.  The major
difference---and one that presents substantial technical
challenges---is that the sub-differential\footnote{As we describe in
more detail in Section~\ref{SecSubdiff}, the sub-differential is the
appropriate generalization of gradient to convex functions that are
allowed to have ``corners'', like the $\ell_1$ and
$\ell_1/\ell_\infty$ norms; the standard
books~\cite{Rockafellar,Hiriart1} contain more background on
sub-differentials and their properties.}  of the block
$\ell_1/\ell_\infty$ is a much more subtle object than the
subdifferential of the ordinary $\ell_1$-norm.  In particular, the
$\ell_1$-norm has an ordinary derivative whenever the coefficient
vector is non-zero.  In contrast, even for non-zero rows of the
regression matrix, the block $\ell_1/\ell_\infty$ norm may be
non-differentiable, and these non-differentiable points play a key
role in our analysis.  (See Section~\ref{SecSubdiff} for more detail
on the sub-differential of this block norm.)  As we show, it is the
Frobenius norm of the sub-differential on the regression matrix
support that controls high-dimensional scaling.  For the ordinary
$\ell_1$-norm, this Frobenius norm is always equal to $\spindex$,
whereas for matrices with $\numreg=2$ columns and $\overlap$ fraction
overlap, this Frobenius norm can be as small as $\frac{(4 - 3\overlap)
\, \spindex}{2}$.  As our analysis reveals, it is precisely the
differing structures of these sub-differentials that leads to
different high-dimensional scaling for $\ell_1$ versus
$\ell_{1,\infty}$ regularization.

The remainder of this paper is organized as follows.  In
Section~\ref{SecProblem}, we provide a precise description of the
problem.  Section~\ref{SecMain} is devoted to the statement of our
main results, some discussion of their consequences, and illustration
by comparison to empirical simulations.  In Section~\ref{SecProof}, we
provide an outline of the proof, with the technical details of many
intermediate lemmas deferred to the appendices.

\noindent \paragraph{Notational conventions:} For the convenience of
the reader, we summarize here some notation to be used throughout the
paper. We reserve the index \mbox{$\pind \in \{1, \ldots, \numreg\}$}
as a superscript in indexing the different regression problems, or
equivalently the columns of the matrix $\Bstar \in \real^{\pdim \times
\numreg}$.  Given a design matrix $\Xmat{} \in \real^{\numobs \times
\pdim}$ and a subset $S \subseteq \{1, \ldots, \pdim\}$, we use
$\Xmatf{}{S}$ to denote the $\numobs \times |\Sset|$ sub-matrix
obtained by extracting those columns indexed by $S$.  For a pair of
matrices $A \in \real^{m \times \ell}$ and $B \in \real^{m \times n}$,
we use the notation $\binprod{A}{B} \defn A^T B$ for the resulting
$\ell \times n$ matrix.

We use the following standard asymptotic notation: for functions
$f,g$, the notation \mbox{$f(n) = \order(g(n))$} means that there
exists a fixed constant $0 < C < +\infty$ such that $f(n) \leq C
g(n)$; the notation $f(n) = \Omega(g(n))$ means that $f(n) \geq C
g(n)$, and $f(n) = \Theta(g(n))$ means that $f(n) = \order(g(n))$ and
$f(n) = \Omega(g(n))$.

%%%%%%%%%%%%%%%%%%%%%%%%%%%%%%%%%%%%%%%%%%%%%%%%%%%%%%%%%%%%%%%%%%%%%%%%%%

\section{Problem set-up}
\label{SecProblem}

We begin by setting up the problem to be studied in this paper,
including multivariate regression and family of block-regularized
programs for estimating sparse vectors.

\subsection{Multivariate regression and block regularization schemes}

In this paper, we consider the following form of multivariate
regression.  For each \mbox{$\pind = 1, \ldots, \numreg$,} let
$\bvecstar{\pind} \in \real^\pdim$ be a regression vector, and
consider the $\numreg$-variate linear regression problem
\begin{eqnarray}
\label{EqnMeasurement}
\yobs{\pind} & = & \Xmat{\pind} \bvecstar{\pind} + \wnoise{\pind},
\qquad \pind = 1, 2, \ldots, \numreg.
\end{eqnarray}
Here each $\Xmat{\pind} \in \real^{\numobs \times \pdim}$ is a design
matrix, possibly different for each vector $\bvecstar{\pind}$, and
$\wnoise{\pind} \in \real^\numobs$ is a noise vector.  We assume that
the noise vectors $\wnoise{\pind}$ and $\wnoise{\pindtwo}$ are
independent for different regression problems $\pind \neq \pindtwo$.
In this paper, we assume that each $\wnoise{\pind}$ has a multivariate
Gaussian $N(0, \sigma^2 I_{\numobs \times \numobs})$ distribution.
However, we note that qualitatively similar results will hold for any
noise distribution with sub-Gaussian tails (see the book~\cite{BulKoz}
for more background on sub-Gaussian variates).

For compactness in notation, we frequently use $\Bstar$ to denote the
$\pdim \times \numreg$ matrix with $\bstar{\pind} \in \real^\pdim$ as
the $\pind^{th}$ column.  Given a parameter $\qpar \in [1, \infty]$,
we define the $\ell_1/\ell_\qpar$ block-norm as follows:
\begin{eqnarray}
\label{EqnDefnBnormq}
\bnormq{\Bstar}{\qpar} & \defn & \sum_{\jind = 1}^\pdim
\|(\bstar{1}_\jind, \bstar{2}_\jind, \ldots,
\bstar{\numreg}_\jind)\|_\qpar,
\end{eqnarray}
corresponding to applying the $\ell_\qpar$ norm to each row of
$\Bstar$, and the $\ell_1$-norm across all of these blocks.  We note
that all of these block norms are special cases of the CAP family of
penalties~\cite{ZhaRoc06}.

This family of block-regularizers~\eqref{EqnDefnBnormq} suggests a
natural family of $M$-estimators for estimating $\Bstar$, based on
solving the block-$\ell_1/\ell_\qpar$-regularized quadratic program
\begin{eqnarray}
\label{EqnGenBlockReg}
\Best & \in & \arg \min_{B \in \real^{\pdim \times \numreg}} \big \{
\frac{1}{2 \numobs} \sum_{\pind = 1}^\numreg \|\yobs{\pind} -
\Xmat{\pind} \beta^\pind\|_2^2  + \regpar \bnormq{B}{\qpar} \big \},
\end{eqnarray}
where $\regpar > 0$ is a user-defined regularization parameter.  Note
that the data term is separable across the different regression
problems $\pind = 1, \ldots, \numreg$, due to our assumption of
independence on the noise vectors.  Any coupling between the different
regression problems is induced by the block-norm regularization.

In the special case of univariate regression ($\numreg = 1$), the
parameter $\qpar$ plays no role, and the block-regularized
scheme~\eqref{EqnBlockReg} reduces to the
Lasso~\cite{Tibshirani96,Chen98}.  If $\qpar = 1$ and $\numreg \geq
2$, the block-regularization function (like the data term) is
separable across the different regression problems $\pind = 1, \ldots,
\numreg$, and so the scheme~\eqref{EqnBlockReg} reduces to solving
$\numreg$ separate Lasso problems.  For $\numreg \geq 2$ and $\qpar =
2$, the program~\eqref{EqnBlockReg} is frequently referred to as the
group Lasso~\cite{Kim06,Obo07}.  Another important
case~\cite{Turlach05,Tro06} and the focus of this paper is the setting
$\qpar = \infty$ and $\numreg \geq 2$, which we refer to as block
$\ell_1/\ell_\infty$ regularization.

The motivation for using block $\ell_1/\ell_\infty$ regularization is
to encourage \emph{shared sparsity} among the columns of the
regression matrix $B$.  Geometrically, like the $\ell_1$ norm that
underlies the ordinary Lasso, the $\ell_1/\ell_\infty$ block norm has
a polyhedral unit ball.  However, the block norm captures potential
interactions between the columns $\beta^\pind$ in the matrix $B$.
Intuitively, taking the maximum encourages the elements
$(\beta^1_\jind, \beta^2_\jind \ldots, \beta^\numreg_\jind)$ in any
given row $\jind = 1, \ldots, \pdim$ to be zero simultaneously, or to
be non-zero simultaneously.  Indeed, if $\beta^\pind_\jind \neq 0$ for
at least one $\pind \in \{1, \ldots, \numreg \}$, then there is no
additional penalty to have $\beta^\pindtwo_\jind \neq 0$ as well, as
long as $|\beta^\pindtwo_\jind| \leq |\beta^\pind_\jind|$.

\subsection{Estimation in $\ell_\infty$ norm and support recovery}

For a given $\regpar > 0$, suppose that we solve the block
$\ell_1/\ell_\infty$ program, thereby obtaining an estimate
\begin{eqnarray}
\label{EqnBlockReg}
\Best & \in & \arg \min_{B \in \real^{\pdim \times \numreg}} \big \{
\frac{1}{2 \numobs} \sum_{\pind = 1}^\numreg \|\yobs{\pind} -
\Xmat{\pind} \beta^\pind\|_2^2 + \regpar \bnormq{B}{\infty} \big \},
\end{eqnarray}
We note that under high-dimensional scaling ($\pdim \gg \numobs$),
this convex program~\eqref{EqnBlockReg} is not necessarily strictly
convex, since the quadratic term is rank deficient and the block
$\ell_1/\ell_\infty$ norm is polyhedral, which implies that the
program is not strictly convex.  However, a consequence of our
analysis is that under appropriate conditions, the optimal solution
$\Best$ is in fact unique.

In this paper, we study the accuracy of the estimate $\Best$, as a
function of the sample size $\numobs$, regression dimensions $\pdim$
and $\numreg$, and the sparsity index $\spindex = \max_{\pind = 1,
\ldots, \numreg} |\Sset(\bstar{\pind})|$.  There are various metrics
with which to assess the ``closeness'' of the estimate $\Best$ to the
truth $\Bstar$, including predictive risk, various types of norm-based
bounds on the difference $\Best-\Bstar$, and variable selection
consistency.  In this paper, we prove results bounding the
$\ell_{\infty}/\ell_\infty$ difference
\begin{eqnarray*}
\bnorm{\Best-\Bstar}{\infty}{\infty} & \defn & \max_{\jind = 1,
\ldots, \pdim} \; \max_{\pind = 1, \ldots, \numreg} |\Best^\pind_\jind
- \Bstar^\pind_\jind|.
\end{eqnarray*}
In addition, we prove results on support recovery criteria.  Recall
that for each vector $\bstar{\pind} \in \real^\pdim$, we use
$\Sset(\bstar{\pind}) = \{\jind \, \mid \, \bstar{\pind}_\jind \neq 0
\}$ to denote its support set.  The problem of \emph{row support
recovery} corresponds to recovering the set
\begin{eqnarray}
\Joint & \defn &  \bigcup_{\pind=1}^\numreg \Sset(\bstar{\pind}),
\end{eqnarray}
corresponding to the subset $\Joint \subseteq \{1, \ldots, \pdim \}$
of indices that are active in at least one regression problem.  Note
that the cardinality of $|\Joint|$ is upper bounded by $\numreg
\spindex$, but can be substantially smaller (as small as $\spindex$)
if there is overlap among the different supports.

As discussed at more length in Appendix~\ref{SecIndividualSupports},
given an estimate of the row support of $\Bstar$, it is possible to
either use additional structure of the solution $\Best$ or perform
some additional computation to recover \emph{individual signed
supports} of the columns of $\Bstar$.  To be precise, define the sign
function
\begin{eqnarray}
\sign(t) & = & \begin{cases} +1 & \mbox{if $t > 0$} \\
                             0 & \mbox{if $t = 0$} \\
                             -1 & \mbox{if $t < 0$.} 
               \end{cases}
\end{eqnarray}
Then the recovery of individual signed supports means estimating the
signed vectors with entries $\sign(\bstar{\pind}_\jind)$, for each
$\pind = 1, 2, \ldots, \numreg$ and for all $\jind = 1, 2, \ldots,
\pdim$.  Interestingly, when using block $\ell_1/\ell_\infty$
regularization, there are multiple ways in which the support (or
signed support) can be estimated, depending on whether we use primal
or dual information from an optimal solution.  

The \emph{dual recovery method} involves the following steps.  First,
solve the block-regularized program~\eqref{EqnBlockReg}, thereby
obtaining an primal solution $\Best \in \real^{\pdim \times \numreg}$.
For each row $\jind = 1, \ldots, \pdim$, compute the set $\Mset_\jind
\defn \arg \max \limits_{\pind = 1, \ldots, \numreg}
|\best{\pind}_\jind|$.  Estimate the support union via $\estim{\Joint}
= \bigcup \limits_{\pind = 1, \ldots, \numreg} S(\best{\pind})$, and
estimate the signed support vectors
\begin{eqnarray}
\label{EqnSignDua}
[\SignDua(\best{\pind}_\jind)] & = & \begin{cases}
\sign(\best{\pind}_\jind) & \mbox{if $\pind \in \Mset_\jind$} \\ 0 &
\mbox{otherwise.}
  \end{cases}
\end{eqnarray}
As our development will clarify, this procedure~\eqref{EqnSignDua}
corresponds to estimating the signed support on the basis of a dual
optimal solution associated with the optimal primal solution.  We
discuss the primal-based recovery method and its differences with the
dual-based method at more length in
Appendix~\ref{SecIndividualSupports}.

%%%%%%%%%%%%%%%%%%%%%%%%%%%%%%%%%%%%%%%%%%%%%%%%%%%%%%%%%%%%%%%%%%%%%%%%%%%

\section{Main results and their consequences}
\label{SecMain}

In this section, we provide precise statements of the main results of
this paper.  Our first main result (Theorem~\ref{ThmDetDesign})
provides sufficient conditions for deterministic design matrices
$\Xmat{1}, \ldots, \Xmat{\numreg}$, whereas our second main result
(Theorem~\ref{ThmGauss}) provides sufficient conditions for design
matrices drawn randomly from sub-Gaussian ensembles.  Both of these
results allow for an arbitrary number $\numreg$ of regression
problems, and the random design case allows for random Gaussian
designs $\Xmat{k}$ with i.i.d. rows and covariance matrix $\Covmati{k}
\in \real^{\pdim \times \pdim}, k = 1, \ldots, \numreg$. Not
surprisingly, these results show that the high-dimensional scaling of
block $\ell_1/\ell_\infty$ is \emph{qualitatively similar} to that of
ordinary $\ell_1$-regularization: for instance, in the case of random
Gaussian designs and bounded $\numreg$, our sufficient conditions
ensure that $\numobs = \Omega(\spindex \log \pdim)$ samples are
sufficient to recover the union of supports correctly with high
probability, which matches known results on the
Lasso~\cite{Wainwright06a}, as well as known information-theoretic
results on the problem of support recovery~\cite{Wainwright06_info}.

As discussed in the introduction, we are also interested in the more
refined question: can we provide necessary and sufficient conditions
that are sharp enough to reveal \emph{quantitative differences}
between ordinary $\ell_1$-regularization and block regularization?
Addressing this question requires analysis that is sufficiently
precise to control the constants in front of the rescaled sample size
$\numobs/\spindex \log(\pdim - \spindex)$ that controls the
performance of both $\ell_1$ and block $\ell_1/\ell_\infty$ methods.
Accordingly, in order to provide precise answers to this question, our
final two results concern the special case of $\numreg = 2$ regression
problems, both with supports of size $\spindex$ that overlap in a
fraction $\overlap$ of their entries, and with design matrices drawn
randomly from the standard Gaussian ensemble.  In this setting, our
final result (Theorem~\ref{ThmPhase}) shows that block
$\ell_1/\ell_\infty$ regularization undergoes a phase
transition---that is, a rapid change from failure to
success---specified by the rescaled sample size $\orpar(\numobs,
\pdim, \spindex, \overlap)$ previously
defined~\eqref{EqnDefnOrder}. We then discuss some consequences of
these results, and illustrate their sharpness with some simulation
results.

\subsection{Sufficient conditions for general deterministic and random designs}

In addition to the sample size $\numobs$, problem dimensions $\pdim$
and $\numreg$, sparsity index $\spindex$ and overlap parameter
$\overlap$, our results involve certain quantities associated with the
design matrices $\Xmatf{\pind}{}$.  To begin, in the deterministic
case, we assume that the columns of each design matrix
\mbox{$\Xmat{\pind}, \pind = 1, \ldots, \numreg$} are
normalized\footnote{The choice of the factor $2$ in this bound is for
later technical convenience.}  so that
\begin{eqnarray}
\label{EqnColMax}
\|\Xmat{\pind}_\jind\|_2^2 & \leq & 2 \numobs \qquad \mbox{for all
$\jind = 1, 2, \ldots \pdim$.}
\end{eqnarray}
More significantly, we require that the following \emph{incoherence
condition} on the design matrix be satisfied:
\begin{eqnarray}
\label{EqnMutIncoDet}
\incopar(\Covmati{}) & \defn & 1 - \max_{\ell = 1, \ldots,
  |\Jointcom|} \sum_{\pind=1}^\numreg \|
\binprod{\Xmatf{\pind}{\ell}}{\Xmatf{\pind}{\Joint}
  (\inprod{\Xmatf{\pind}{\Joint}}{\Xmatf{\pind}{\Joint}})^{-1}} \|_1
\; > \; 0.
\end{eqnarray}
For the case of the ordinary Lasso, conditions of this type are
known~\cite{Meinshausen06,Zhao06,Wainwright06a} to be both necessary
and sufficient for successful support recovery.\footnote{Some
work~\cite{MeiYu08} has shown that multi-stage methods can allow some
relaxation of this incoherence condition; however, as our main
interest is in understanding the sample complexity of ordinary
$\ell_1$ versus $\ell_1/\ell_\infty$ relaxations, we do not pursue
such extensions here.}

In addition, the statement of our results involve certain quantities
associated with the $|\Joint| \times |\Joint|$ matrices
$\frac{1}{\numobs}
\inprod{\Xmatf{\pind}{\Joint}}{\Xmatf{\pind}{\Joint}}$; in particular,
we define a lower bound on the minimum eigenvalue
\begin{eqnarray}
\label{EqnCmin}
\Cmin(\Xmat{}) & \leq & \min_{\pind =1, \ldots, \numreg } \lammin
\big( \frac{1}{\numobs}
\inprod{\Xmatf{\pind}{\Joint}}{\Xmatf{\pind}{\Joint}} \big),
\end{eqnarray}
as well as an upper bound maximum $\ell_{\infty, \infty}$-operator
norm of the inverses
\begin{eqnarray}
\label{EqnDmax}
\Dmax(\Xmat{}) & \geq & \max_{\pind = 1, \ldots, \numreg} \matsnorm{
\big(\frac{1}{\numobs} \inprod{\Xmatf{\pind}{\Joint}}
{\Xmatf{\pind}{\Joint}} \big)^{-1}}{\infty}.
\end{eqnarray}
Remembering that our analysis applies to to sequences
$\{\Xmat{}_{\numobs, \pdim} \}$ of design matrices, in the simplest
scenario, both of the bounding quantities $\Cmin$ and $\Dmax$ do not
scale with $(\numobs, \pdim, \spindex)$.  To keep notation compact, we
write $\Cmin$ and $\Dmax$ in the analysis to follow.

We also define the \emph{support minimum value}
\begin{eqnarray}
\label{EqnDefnBetamin}
\mybetamin & = & \min_{\jind \in \Joint} \max_{\pind =
1,\hdots,\numreg} |\bstar{\pind}_\jind|,
\end{eqnarray}
corresponding to the minimum value of the $\ell_\infty$ norm of any
row $\jind \in \Joint$.

\begin{theorem}[Sufficient conditions for deterministic designs]
\label{ThmDetDesign}
Consider the observation model~\eqref{EqnMeasurement} with design
matrices $\Xmat{\pind}$ satisfying the column bound~\eqref{EqnColMax}
and incoherence condition~\eqref{EqnMutIncoDet}. Suppose that we solve
the block-regularized $\ell_1/\ell_\infty$ convex
program~\eqref{EqnBlockReg} with regularization parameter $\relaxn^2
\geq \frac{4 \keypar \sigma^2}{\incopar^2} \frac{\numreg^2 + \numreg
\log (\pdim)}{\numobs}$ for some $\keypar > 1$.  Then with probability
greater than
\begin{eqnarray}
\phiprob & \defn & 1 - 2 \exp(-(\keypar-1) [\numreg + \log \pdim])
-2\exp(-(\keypar^2-1) \log (\numreg \spindex)),
\end{eqnarray}
we are guaranteed that
\begin{enumerate}
\item[(a)] The block-regularized program has a unique solution $\Best$
such that $\bigcup_{\pind=1}^\numreg \Sset(\best{\pind}) \subseteq
\Joint$.
\item[(b)] Moreover, the solution satisfies the elementwise
$\ell_\infty$-bound
\begin{eqnarray}
\label{EqnEllinfDet}
\bnorm{\Best-\Bstar}{\infty}{\infty} & \leq & \underbrace{\keypar
\sqrt{\frac{4\sigma^2}{\Cmin} \; \frac{\log \mynumregspind}{\numobs}}
+ \Dmax \, \relaxn}. \\
& & \qquad \qquad \Bou \nonumber
\end{eqnarray}
Consequently, as long as $\mybetamin \geq \Bou$, then
$\bigcup_{\pind=1}^\numreg \Sset(\best{\pind}) = \Joint$, so that the
solution $\Best$ correctly specifies the union of supports $\Joint$.
\end{enumerate}
\end{theorem}

%%%%%%%%%%%%%%%%%%%%%%%%%%%%%%%%%%%%%%%%%%%%%%%%%%%%%%%%%%%%%%%%%%%%%%%%%%%%%%%

We now state an analogous result for random design matrices; in
particular, consider the observation model~\eqref{EqnMeasurement} with
design matrices $\Xmat{\pind}$ chosen with i.i.d. rows from covariance
matrices $\Covmati{\pind}$.  In analogy to definitions~\eqref{EqnCmin}
and~\eqref{EqnDmax} in the deterministic case, we define the lower
bound
\begin{eqnarray}
\label{EqnCminRand}
\Cmin(\Covmati{}) & \leq & \min_{\pind = 1, \ldots, \numreg}
\lammin \big( \Covmatif{\pind}{\Joint \Joint} \big),
\end{eqnarray}
as well as an analogous upper bound on $\ell_\infty$-operator norm of
the inverses
\begin{equation}
\label{EqnDmaxRand}
\Dmax(\Covmati{}) \; \geq \; \max_{\pind = 1, \ldots, \numreg}
\matsnorm{ (\Covmatif{\pind}{\Joint \Joint})^{-1}}{\infty} \leq \Dmax.
\end{equation}
Note that unlike the case of deterministic designs, these quantities
are \emph{not} functions of the design matrix $\Xmat{}$, which is now
a random variable.  Finally, our results involve an analogous
incoherence parameter of the covariance matrices $\Covmati{} = \{
\Covmati{\pind}, \pind = 1, \ldots, \numreg \}$, defined as
\begin{eqnarray}
\label{EqnMutIncoRand}
\incopar(\Covmati{}) & \defn & 1- \max_{\jind = 1, \ldots,
  |\Jointcom|} \sum_{\pind=1}^\numreg \big\| \Covmatif{\pind}{\jind \,
  \Joint} \; (\Covmatif{\pind}{\Joint \Joint})^{-1} \big \|_1 \; > \;
0.
\end{eqnarray}

With this notation, the following result provides an analog of
Theorem~\ref{ThmDetDesign} for random design matrices:
\begin{theorem}[Sufficient conditions for random Gaussian designs]
\label{ThmGauss}
Suppose that we are given $\numobs$ i.i.d. observations from the
model~\eqref{EqnMeasurement} with
\begin{eqnarray}
\label{EqnGaussSampSize}
\numobs & > & \frac{8 \, \tautwo \; \numreg }{\Cmin \incopar^2} \,
\spindex \big(\numreg + \log \pdim \big)
\end{eqnarray}
for some $\tautwo > 1$.  If we solve the convex
program~\eqref{EqnBlockReg} with regularization parameter satisfying
\mbox{$\relaxn \geq \frac{4 \keypar \sigma^2}{\incopar^2}
\big[\frac{\numreg^2 + \numreg \log (\pdim)}{\numobs}]$} for some
$\keypar > 1$, then with probability greater than
\begin{equation}
\phiprobtwo \defn 1- 2 \exp \big \{ -2 (\keypar^2 -1)\log (\numreg
\spindex) \big \} - 2 \exp \big \{ -\tautwo ( \numreg + \log \pdim)
\big \} \rightarrow 1,
\end{equation}
we are guaranteed that
\begin{enumerate}
\item[(a)] The block-regularized program~\eqref{EqnBlockReg} has a
unique solution $\Best$ such that $\bigcup_{\pind=1}^\numreg
\Sset(\best{\pind}) \subseteq \Joint$.
\item[(b)] The solution satisfies the elementwise $\ell_\infty$ bound
\begin{eqnarray}
\label{EqnEllinfGauss}
\bnorm{\Best-\Bstar}{\infty}{\infty} & \leq & \underbrace{\keypar\;
\sqrt{\frac{100 \sigma^2}{\Cmin} \, \frac{\log
\mynumregspind}{\numobs}} + \relaxn \, \big[\frac{4
\spindex}{\sqrt{\numobs}} + \Dmax \big],} \\
& & \qquad \qquad \qquad \Boutwo \nonumber
\end{eqnarray}
Consequently, if $\betamin \geq \Boutwo$, then
$\bigcup_{\pind=1}^\numreg \Sset(\best{\pind}) = \Joint$, so that the
solution $\Best$ correctly specifies the union of supports $\Joint$.
\end{enumerate}
\end{theorem}

\vspace*{.2in}

To clarify the interpretation of Theorems~\ref{ThmDetDesign} and
Theorem~\ref{ThmGauss}, part (a) of each claim guarantees that the
estimator has \emph{no false inclusions}, in that the row support of
the estimate $\Best$ is contained within the row support of the true
matrix $\Bstar$.  One consequence of part (b) is that as long as the
minimum signal parameter $\betamin$ decays slowly enough, then the
estimators have \emph{no false exclusions}, so that the true row
support is correctly recovered.  

In terms of consistency rates in block $\ell_\infty/\ell_\infty$ norm,
assuming that the design-related quantities $\Cmin$, $\Dmax$ and
$\incopar$ do not scale with $\pdim$, Theorem~\ref{ThmDetDesign}(a)
guarantees consistency in elementwise $\ell_\infty$-norm at the rate
\begin{eqnarray*}
\bnorm{\Best-\Bstar}{\infty}{\infty} & = & \order \big ( \sigma^2 \;
\sqrt{\frac{\numreg^2 +\numreg \log \pdim}{\numobs}} \big).
\end{eqnarray*}
Here we have used the fact that $\log |\Joint| \leq \log (\numreg
\spindex) = o(\numreg \log \pdim)$.  Similarly,
Theorem~\ref{ThmGauss}(b) guarantees consistency in elementwise
$\ell_\infty$-norm at the rate
\begin{eqnarray*}
\bnorm{\Best-\Bstar}{\infty}{\infty} & = & \order \big ( \sigma^2 \;
\sqrt{\max\{1, \frac{\spindex}{\numreg \log \pdim} \big \}
\; \frac{\numreg^2 +\numreg \log \pdim}{\numobs}} \big).
\end{eqnarray*}
In this expression, the extra term $\max \{1, \spindex/(\numreg \log
\pdim)\}$ arises in the analysis due to the need to control the norms
of the random design matrices.  For sufficiently sparse problems
(e.g., $\spindex = \order(\log \pdim)$), this factor is constant.

\vspace*{.2in}

At a high level, our results thus far show that for a fixed number
$\numreg$ of regression problems, the $\ell_1/\ell_\infty$ method
guarantees exact support recovery with $\numobs = \Omega(\spindex \log
\pdim)$ samples, and guarantees consistency in an elementwise sense at
rate $\order(\sqrt{\frac{\log \pdim}{\numobs}})$.  In qualitative
terms, these results match the known scaling~\cite{Wainwright06a} for
the Lasso ($\ell_1$-regularized QP), which is obtained as the special
case for univariate regression ($\numreg = 1$).  It should be noted
that this scaling is known to be optimal in an information-theoretic
sense: no algorithm can recover support correctly if the rescaled
sample size $\orlas = \frac{\numobs}{2 \spindex \log (\pdim -
\spindex)}$ is below a critical
threshold~\cite{Wainwright06_info,WanWaiRam08}.

\subsection{A phase transition for standard Gaussian ensembles}

In order to provide keener insight into the advantages and/or
disadvantages associated with using $\ell_1/\ell_\infty$ block
regularization, we need to obtain even sharper results, ones that are
capable of distinguishing constants in front of the rescaled sample
size $\orlas$.  With this aim in mind, the following results are
specialized to the case of $\numreg = 2$ regression problems, where
the corresponding design matrices $\Xmat{\pind}, \pind =1,2$ are
sampled from the standard Gaussian ensemble---i.e., with i.i.d. rows
$N(0, I_{\pdim \times \pdim})$.  By studying this simpler class of
problems, we can make \emph{quantitative comparisons} to the sample
complexity of the Lasso, which provide insight into the benefits and
dangers of block $\ell_1/\ell_\infty$ regularization.

The main result of this section asserts that there is a phase
transition in the performance of $\ell_1/\ell_\infty$ quadratic
programming for suppport recovery---by which we mean a sharp
transition from failure to success---and provide the exact location of
this transition point as a function of $(\numobs, \pdim, \spindex)$
and the overlap parameter $\overlap \in (0,1)$.  The phase transition
involves the \emph{support gap}
\begin{eqnarray}
\label{EqnDefnBetagap}
\mybetagap & = & \max_{i \in \Sset(\bstar{1}) \cap \Sset(\bstar{2})}
\big | \, |\bstar{1}_i| - |\bstar{2}_i| \, \big|.
\end{eqnarray}
This quantity measures how close the two regression vectors are in
absolute value on their \emph{shared support}.  Our main theorem
treats the case in which this gap vanishes (i.e., \mbox{$\mybetagap =
o(1)$}); note that block $\ell_1/\ell_\infty$ regularization is
best-suited to this type of structure.  A subsequent corollary
provides more general but technical conditions for the cases of
non-vanishing support gaps.  Our main result specifies a phase
transition in terms of the \emph{rescaled sample size}
\begin{eqnarray}
\label{EqnNewOrpar}
\orpar(\numobs, \pdim, \spindex, \overlap) & \defn & \frac{\numobs}{(4
 - 3 \overlap) \spindex \log(\pdim-(2-\overlap) \spindex)},
\end{eqnarray}
as stated in the theorem below.  

\begin{theorem}[Phase transition]
\label{ThmPhase}
Consider sequences of problems, indexed by $(\numobs, \pdim, \spindex,
\overlap)$ drawn from the observation model~\eqref{EqnMeasurement}
with random design $X$ drawn with i.i.d. standard Gaussian entries and 
with $\Cmin = 1 = \Dmax$.
\begin{enumerate}
\item[(a)] {\underline{Success:}} Suppose that the problem sequence
$(\numobs, \pdim, \spindex, \overlap)$ satisfies
\begin{eqnarray}
\orpar(\numobs, \pdim, \spindex, \overlap) & > & 1 + \delpar\qquad
\mbox{for some $\delpar > 0$.}
\end{eqnarray}
If we solve the block-regularized program~\eqref{EqnBlockReg} with
$\relaxn \geq \sqrt{\frac{\keypar \sigma^2 \log \pdim}{\numobs}}$ for
some $\keypar > 2$ and $\mybetagap = o(\relaxn)$, then with
probability greater than \mbox{$1- c_1\exp(-c_2 \log (\pdim -
(2-\overlap) \spindex))$,} the block
$\ell_{1,\infty}$-program~\eqref{EqnBlockReg} has a unique solution
$\Best$ such that $\Sset(\Best) \subseteq \Joint$, and moreover it
satisfies the elementwise bound~\eqref{EqnEllinfGauss} with $\Cmin = 1
= \Dmax$. In addition, if $\betamin > \Boutwo$, then the unique
solution recovers the correct signed support.

%%%%%%%%%%%%
\item[(b)] {\underline{Failure:}} For problem sequences $(\numobs,
\pdim, \spindex, \overlap)$ such that
\begin{eqnarray}
\orpar(\numobs, \pdim, \spindex, \overlap) & < & 1 - \delpar \qquad
\mbox{for some $\delpar > 0$}
\end{eqnarray}
and for any non-increasing regularization sequence $\relaxn > 0$, no
solution $\Best = (\best{1}, \best{2})$ to the block-regularized
program~\eqref{EqnBlockReg} has the correct signed support.
% union $\Sset(\best{1}) \cup \Sset(\best{2})$.
%
\end{enumerate}
\end{theorem}
In a nutshell, Theorem~\ref{ThmPhase} states that block
$\ell_1/\ell_\infty$ regularization recovers the correct support with
high probability for sequences $(\numobs, \pdim, \spindex, \overlap)$
such that $\orpar(\numobs, \pdim, \spindex, \overlap) > 1$, and
otherwise fails with high probability.

\newcommand{\Trun}{\ensuremath{T_{\regpar}}}
\newcommand{\Gaplim}{\ensuremath{\Delta}}
\newcommand{\myvmax}{\ensuremath{v^*}}
We now consider the case in which the support gap does not vanish, and
show that it only further degrades the performance of block
$\ell_1/\ell_\infty$ regularization.  To make the degree of this
degradation precise, we define the $\regpar$-truncated gap vector
$\Trun(\Best) \in \real^\pdim$, with elements
\begin{eqnarray*}
[\Trun(\Bstar)]_i & \defn & 
\begin{cases} \min \big \{ \regpar, \; \big | \, |\bstar{1}_i| -
  |\bstar{2}_i| \, \big| \big \} & \mbox{if $i \in \Sset(\bstar{1})
   \cap \Sset(\bstar{2})$} \\
0 & \mbox{otherwise}
\end{cases}
\end{eqnarray*}
Recall that support overlap $\Sset(\bstar{1}) \cap \Sset(\bstar{2})$
has cardinality $\overlap \spindex$ by assumption.  Therefore,
$\Trun(\Bstar)$ has at most $\overlap \spindex$ non-zero entries, and
moreover $\|\Trun(\Bstar)\|_2^2 \leq \regpar^2 \overlap \spindex$.  We
then define the rescaled gap limit
\begin{eqnarray}
\label{EqnRescaledGap}
\Gaplim(\Bstar, \relaxn) & \defn & \lim \sup_{(\numobs, \pdim,
  \spindex)} \frac{\|\Trun(\Bstar)\|^2_2}{\relaxn^2 \spindex}.
\end{eqnarray}
Note that $\Gaplim(\Bstar, \relaxn) \in [0,\overlap]$ by construction.
With these definitions, we have the following:
\begin{corollary}[Poorer performance with non-vanishing gap]
\label{gapcor}
If for any $\delpar > 0$, the sample size $\numobs$ is upper bounded
as
\begin{eqnarray}
\label{EqnGapCor}
  \numobs & < & (1-\delpar) \; \big [ (4 - 3 \overlap) +
\Gaplim(\Bstar, \relaxn) \big ] \spindex \log [\pdim-(2-\overlap)
\spindex],
\end{eqnarray}
then the dual recovery method~\eqref{EqnSignDua} fails to recover the
individual signed supports.
\end{corollary}
To understand the implications of this result, suppose that all
$\overlap \spindex$ of the gaps $\big | \, |\bstar{1}_i| -
|\bstar{2}_i| \, \big|$ were above the regularization level $\relaxn$.
Then by definition, we have $\Gaplim(\Bstar, \relaxn) = \overlap$, so
that condition~\eqref{EqnGapCor} implies that the method fails for all
$\numobs \, < \, (1-\delpar) \, [4 - 2 \overlap] \spindex
\log[\pdim-(2-\overlap) \spindex]$.  Since the factor $(4-2 \overlap)$
is strictly greater than $2$ for all $\overlap < 1$, this scaling is
\emph{always worse}\footnote{Here we are assuming that $\spindex/\pdim
= o(1)$, so that $\log(\pdim - \spindex) \asymp
\log[\pdim-(2-\overlap) \spindex]$.} than the Lasso scaling given by
$\numobs \asymp 2 \spindex \log(\pdim - \spindex)$ (see
equation~\eqref{EqnLassoThresh}), unless there is perfect overlap
($\overlap = 1$), in which case it yields no improvements.
Consequently, Corollary~\ref{gapcor} shows that the performance
$\ell_1/\ell_\infty$ regularization is also very sensitive to the
numerical amplitudes of the signal vectors.

\subsection{Illustrative simulations and some consequences} 

In this section, we provide some simulation results to illustrate the
phase transition predicted by Theorem~\ref{ThmPhase}.  Interestingly,
these results show that the theory provides an accurate description of
practice even for relatively small problem sizes (e.g., $\pdim =
128$).  As specified in Theorem~\ref{ThmPhase}, we simulate
multivariate regression problems with $\numreg = 2$ columns, with the
design matrices $\Xmat{\pind}$ drawn from the standard Gaussian
ensemble.  In all cases, we initially solved the $\ell_1/\ell_\infty$
program using MATLAB, and then verified that the behavior of the
solution agreed with the primal-dual optimality conditions specified
by our theory.  In subsequent simulations, we solved directly for
the dual variables, and then checked whether or not the dual
feasibility conditions are met.

We first illustrate the difference between unscaled and rescaled plots
of the empirical performance, which demonstrate that the rescaled
sample size $\numobs/[\spindex \log(\pdim - \spindex)]$ specifies the
high-dimensional scaling of block $\ell_1/\ell_\infty$ regularization.
\begin{figure}[h]
  \begin{center}
    \begin{tabular}{ccc}
      \widgraph{0.5\textwidth}{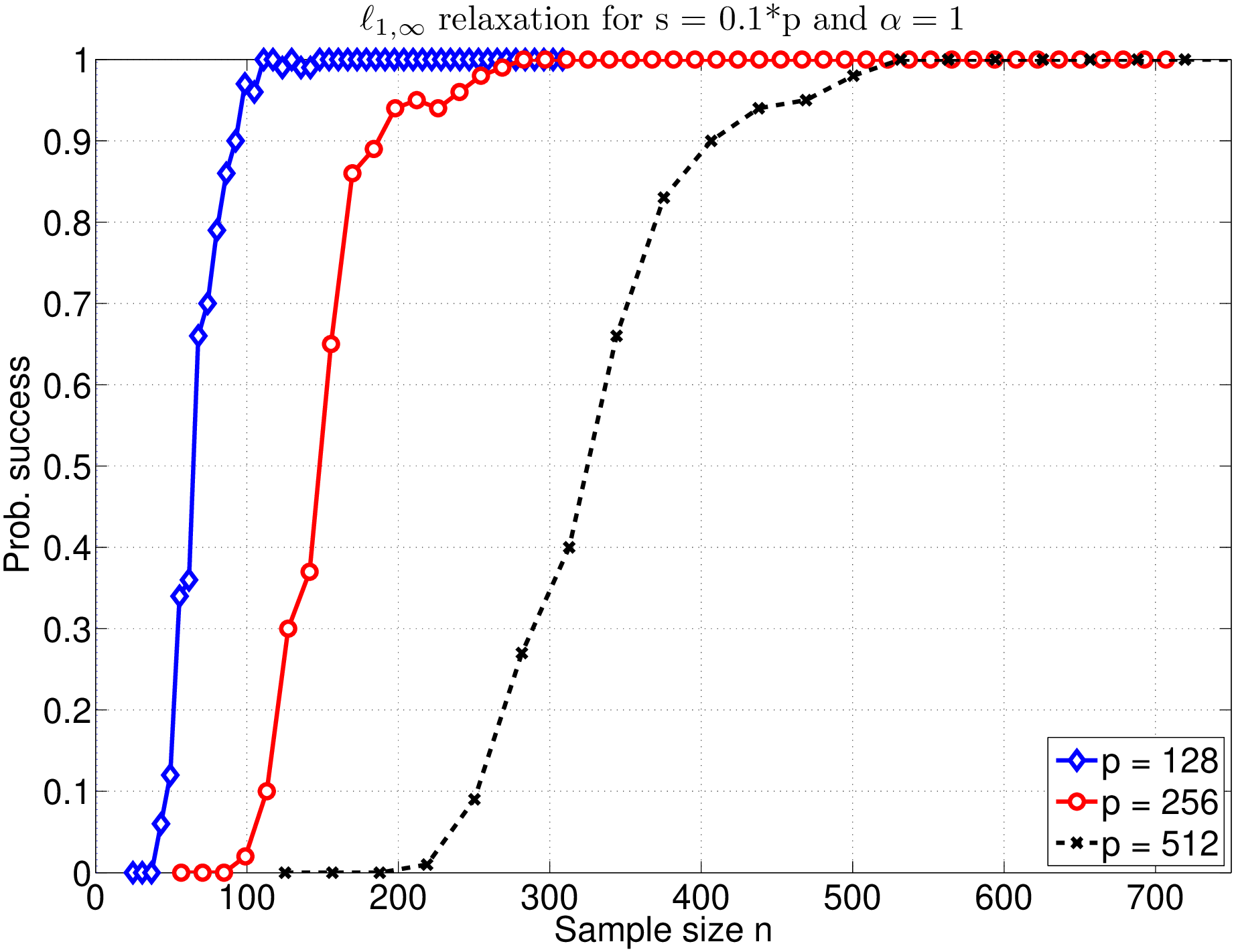} & &
      \widgraph{0.5\textwidth}{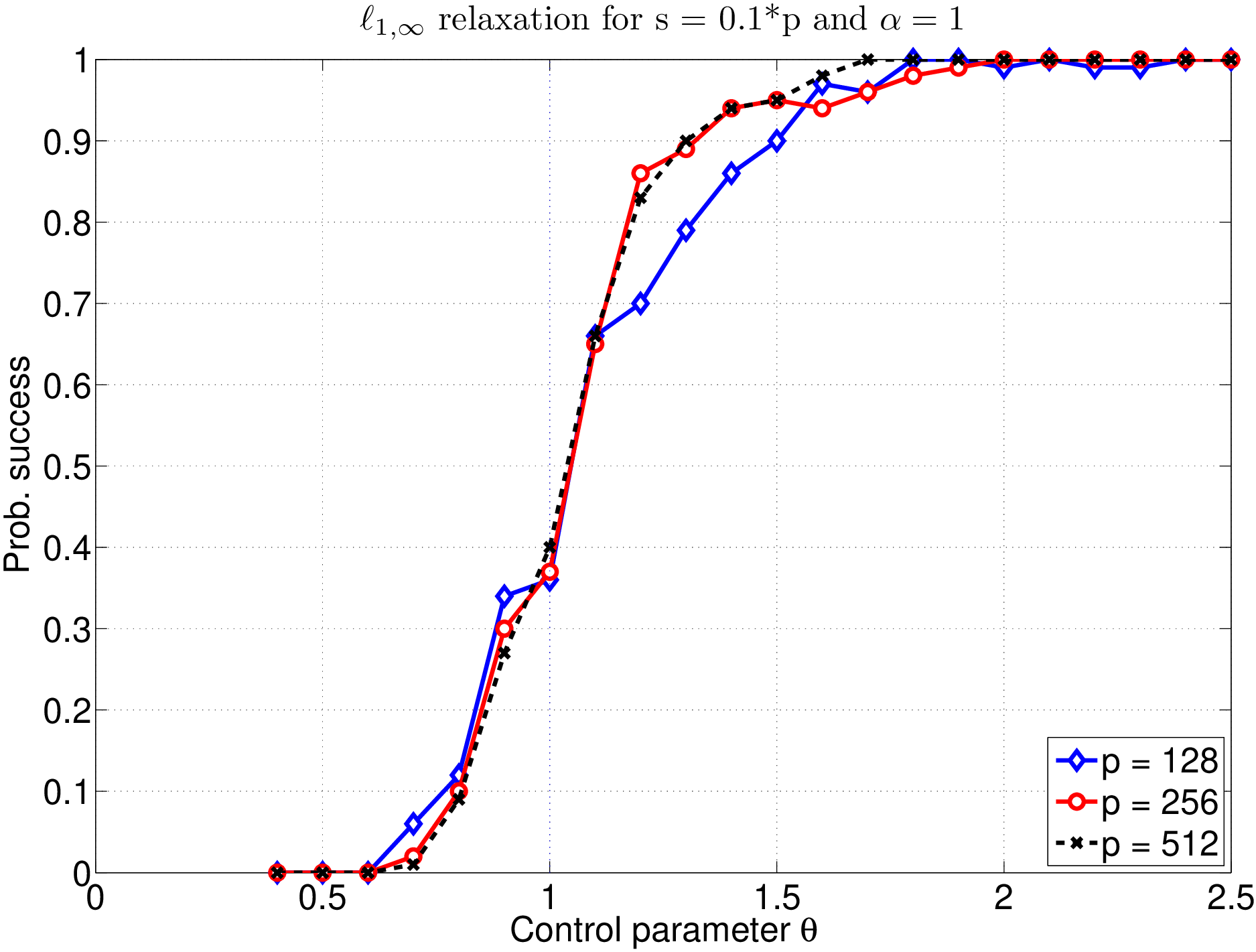} \\
      (a) & & (b)
    \end{tabular}
  \end{center}
  \caption{(a) Plots of the probability $\mprob[\estim{\Joint} =
\Joint]$ of successful joint support recovery versus the sample size
$n$.  Each curve corresponds to a different problem size $\pdim$;
notice how the curves shift to the right as $\pdim$ increases,
reflecting the difficulty of solving larger problems.  (b) Plots of
the same data versus the rescaled sample size $
\numobs/[2
\spindex\log(\pdim - \spindex)]$; note how all three curves now align
with one another, showing that this order parameter is the correct
scaling for assessing the method.}
  \label{FigUnScaledScaled}
\end{figure}
Figure~\ref{FigUnScaledScaled}(a) shows the empirical behavior of the
block $\ell_1/\ell_\infty$ method for joint support recovery.  For
these simulations, we applied the method to $\numreg = 2$ regression
problems with overlap $\overlap = 1$, and to three different problem
sizes $\pdim \in \{128, 256, 512 \}$, in all cases with the sparsity
index $\spindex = \lfloor 0.1 \pdim \rfloor$.  Each curve in panel (a)
shows the probability of correct support recovery
$\mprob[\estim{\Joint} = \Joint]$ versus the raw sample size
$\numobs$.  As would be expected, all the curves initially start at
$\mprob[\estim{\Joint} = \Joint] = 0$, but then transition to $1$ as
$n$ increases, with the transition taking place at larger and larger
sampler sizes as $\pdim$ is increased.  The purpose of the rescaling
is to determine exactly how this transition point depends on the
problem size $\pdim$ and other structural parameters ($\spindex$ and
$\overlap$).  Figure~\ref{FigUnScaledScaled}(b) shows the same
simulation results, now plotted versus the rescaled sample size
$\theta \defn \numobs/[2 \spindex\log(\pdim - \spindex)]$, which is
the appropriate rescaling predicted by our theory.  Notice how all
three curves now lie on top of another, and moreover transition
from failure to success at $\theta \approx 1$, consistent with our
theoretical predictions.

We now seek to explore the dependence of the sample size on the
overlap fraction $\overlap \in [0,1]$ of the two regression vectors.
For this purpose, we plot the probability of successful recovery
versus the rescaled sample size
\begin{eqnarray*}
\orpar(\numobs, \pdim, \spindex, \overlap) = \frac{\numobs}{(4 - 3
\overlap) \spindex \log(\pdim-(2-\overlap) \spindex)}.
\end{eqnarray*}
As shown by Figure~\ref{FigUnScaledScaled}(b), when plotted with this
rescaling, there is any longer size $\pdim$.  Moreover, if we choose
the sparsity index $\spindex$ to grow in a fixed way with $\pdim$
(i.e., $\spindex = f(\pdim)$ for some fixed function $f$), then the
only remaining free variable is the overlap parameter $\overlap$.
Note that the theory predicts that the required sample size should
decrease as $\overlap$ increases towards $1$.

As shown earlier in Section~\ref{SecIntro}, Figure~\ref{FigSims} plots
the probability of successful recovery of the joint supports versus
the rescaled samples size $\orpar(\numobs, \pdim, \spindex,
\overlap)$.  Notice that the plot shows four sets of `stacked''
curves, where each stack corresponds to a different choice of the
overlap parameter, ranging from $\overlap = 1$ (left-most stack), to
$\overlap = 0.1$ (right-most stack).  Each stack contains three
curves, corresponding to the problem sizes $\pdim \in \{128, 256, 512
\}$.  In all cases, we fixed the support size $\spindex = 0.1 \pdim$.
As with Figure~\ref{FigUnScaledScaled}(b), the ``stacking'' behavior
of these curves demonstrates that Theorem~\ref{ThmPhase} isolates the
correct dependence on $\pdim$.  Moreover, their step-like behavior is
consistent with the theoretical prediction of a phase transition.
Notice how the curves shift towards the left as the overlap parameter
$\overlap$ parameter increases towards one, reflecting that the
problems become easier as the amount of shared sparsity increases.  To
assess this shift in a qualitative manner for each choice of overlap
$\overlap \in \{0.1, 0.3. 0.7. 1 \}$, we plot a vertical line within
each group, which is obtained as the threshold value of $\orpar$
predicted by our theory.  Observe how the theoretical value shows
excellent agreement with the empirical behavior.

As noted previously in Section~\ref{SecIntro}, Theorem~\ref{ThmPhase}
has some interesting consequences, particularly in comparison to the
behavior of the ``naive'' Lasso-based individual decoding of signed
supports---that is, the method that simply applies the Lasso (ordinary
$\ell_1$-regularization) to each column $\pind = 1, 2$ separately.  By
known results~\cite{Wainwright06a} on the Lasso, the performance of
this naive approach is governed by the order parameter
$\orlas(\numobs, \pdim, \spindex) \, = \, \frac{\numobs}{2 \spindex
\log(\pdim - \spindex)}$, meaning that for any $\delpar > 0$, it
succeeds for sequences such that $\orlas > 1 + \delpar$, and
conversely fails for sequences such that $\orlas < 1 - \delpar$.  To
compare the two methods, we define the relative efficiency coefficient
$R(\orpar, \orlas) \defn \orlas(\numobs, \pdim, \spindex) /
\orpar(\numobs, \pdim, \spindex, \overlap)$.  A value of $R < 1$
implies that the block method is more efficient, while $R > 1$ implies
that the naive method is more efficient.  With this notation, we have
the following:
\begin{corollary}
\label{CorScaling}
The relative efficiency of the block $\ell_{1,\infty}$
program~\eqref{EqnBlockReg} compared to the Lasso is given by
$R(\orpar, \orlas) = \frac{4 - 3 \overlap}{2} \frac{\log(\pdim -
(2-\overlap) \spindex)}{\log(\pdim - \spindex)}$. Thus, for sublinear
sparsity $\spindex/\pdim \rightarrow 0$, the block scheme has greater
statistical efficiency for all overlaps $\overlap \in (2/3,1]$, but
lower statistical efficiency for overlaps $\overlap \in [0, 2/3)$.
\end{corollary}

\begin{figure}[h]
  \begin{center}
    \begin{tabular}{c}
      \widgraph{.65\textwidth}{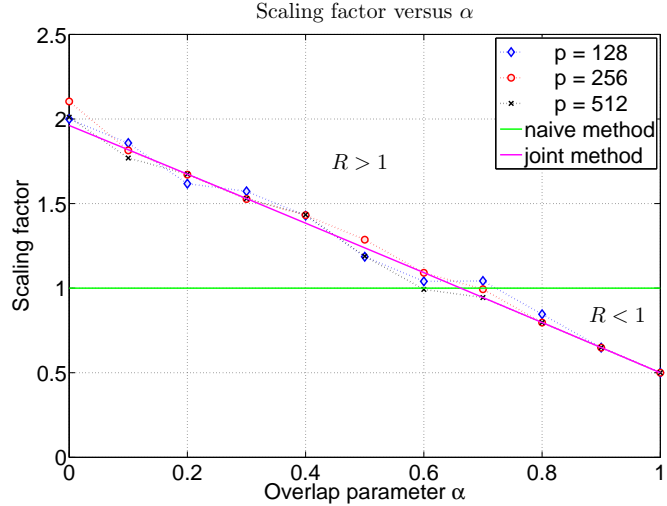}
    \end{tabular}
  \end{center}
\caption{Plots of the relative statistical efficiency $R(\overlap)$ of
a method based on block-$\ell_1/\ell_\infty$ regularization versus the
Lasso (ordinary $\ell_1$-regularization).  For each value of the
parameter $\overlap \in [0,1]$ that measures overlap between the
regression problems, the quantity $R(\alpha)$ is the ratio of sample
size required by an estimator based on block
$\ell_1/\ell_\infty$-regularization relative to the sample size
required by the Lasso (ordinary $\ell_1$-regularization).  The error
criterion here is recovery of the correct subset of active variables
in the regression.  Over a range of overlaps, the empirical thresholds
of the $\ell_1/\ell_\infty$ block regularization method closely align
with the theoretical prediction of $(4 - 3 \overlap)/2$.  The
block-based method begins to give benefits versus the ``naive''
Lasso-based method at the critical overlap $\overlap^* \approx 2/3$,
at which point the relative efficiency $R(\overlap)$ first drops below
$1$.  For overlaps $\overlap \in [0, 2/3)$, the joint method actually
requires more samples than the naive method. }
\label{FigRelEff}
\end{figure}

Figure~\ref{FigRelEff} provides an alternative perspective on the
data, where we have plotted how the sample size required by block
regression changes as a function of the overlap parameter $\overlap
\in [0,1]$.  Each set of data points plots a scaled form of the sample
size required to hit $50\%$ success, for a range of overlaps, and the
straight line $(4-3\overlap)/2$ that is predicted by
Theorem~\ref{ThmPhase} Note the excellent agreement between the
experimental results, for all three problem sizes for $\pdim \in
\{128, 256, 512 \}$, and the full range of overlaps.  The line
$(4-3\overlap)/2$ also characterizes the relative efficiency $R$ of
block regularization versus the naive Lasso-based method, as described
in Corollary~\ref{CorScaling}.  For overlaps $\overlap > 2/3$, this
parameter $R$ drops below $1$.  On the other hand, for overlaps
$\overlap < 1$, we have $R > 1$, so that applying the joint
optimization problem actually decreases statistical efficiency.
Intuitively, although there is still some fraction of overlap, the
regularization is misleading, in that it tries to enforce a higher
degree of shared sparsity than is actually present in the data.

%%%%%%%%%%%%%%%%%%%%%%%%%%%%%%%%%%%%%%%%%%%%%%%%%%%%%%%%%%%%%%%%%%%%%%%%%%%

\section{Proofs}
\label{SecProof}

This section contains the proofs of our three theorems.  Our proofs
are constructive in nature, based on a procedure that constructs pair
of matrices $\Bwit = (\bwit{1}, \ldots, \bwit{\numreg}) \in
\real^{\pdim \times \numreg}$ and $\Dwit = (\dwit{1}, \ldots,
\dwit{\numreg}) \in \real^{\pdim \times \numreg}$.  The goal of the
construction is to show that matrix $\Bwit$ is an optimal primal
solution to the convex program~\eqref{EqnBlockReg}, and that the
matrix $\Dwit$ is a corresponding dual-optimal solution, meaning that
it belongs to the sub-differential of the $\ell_{1,\infty}$-norm (see
Lemma~\ref{LemSubdiff}), evaluated at $\Bwit$.  If the construction
succeeds, then the pair $(\Bwit, \Dwit)$ acts as a witness for the
success of the convex program~\eqref{EqnBlockReg} in recovering the
correct signed support---in particular, success of the primal-dual
witness procedure implies that $\Bwit$ is the unique optimal solution
of the convex program~\eqref{EqnBlockReg}, with its row support
contained with $\Joint$.  To be clear, the procedure for constructing
this candidate primal-dual solution is \emph{not} a practical
algorithm (as it exploits knowledge of the true support sets), but
rather a proof technique for certifying the correctness of the
block-regularized program.

We begin by providing some background on the sub-differential of the
$\ell_1/\ell_\infty$ norm; we refer the reader to the
books~\cite{Rockafellar,Hiriart1} for more background on convex
analysis.

\subsection{Structure of $\ell_1/\ell_\infty$-norm sub-differential}
\label{SecSubdiff}

The sub-differential of a convex function $f:\real^d \rightarrow
\real$ at a point $x \in \real^d$ is the set of all vectors $y \in
\real^d$ such that $f(x') \geq f(x) + \inprod{y}{x'-x}$ for all $x'
\in \real^d$.  See the standard references~\cite{Rockafellar,Hiriart1}
for background on subdifferentials and their properties.

We state for future reference a characterization of the
sub-differential of the $\ell_1/\ell_\infty$ block norm:
\begin{lemma} 
\label{LemSubdiff}
The matrix $\Dwit \in \real^{\pdim \times \numreg}$ belongs to the
sub-differential $\partial \bnorminf{\Bwit}$ if and only if the
following conditions hold for each $\jind = 1, \ldots, \pdim$.

\begin{enumerate}
\item[(i)] If $\bwit{\pind}_\jind \neq 0$ for at least one index
$\pind \in \{1, \ldots, \numreg\}$, then 
\begin{eqnarray*}
\dwit{\pind}_\jind & = & \begin{cases} t_\pind
\sign(\bwit{\pind}_\jind) & \mbox{if $\pind \in \Mset_\jind$} \\
0 & \mbox{otherwise}.
                         \end{cases},
\end{eqnarray*}
where $\Mset_\jind \defn \arg \max \limits_{\pind = 1, \ldots, \numreg}
|\bwit{\pind}_\jind|$, for a set of non-negative scalars $\{t_\pind,
\; \pind \in \Mset_\jind\}$ such that $\sum_{\pind \in \Mset_\jind} t_\pind =
1$.
\item[(ii)] If $\bwit{\pind}_\jind = 0$ for all $\pind = 1, \ldots,
\numreg$, then we require $\sum_{\pind=1}^\numreg |\dwit{\pind}_\jind|
\leq 1$.
\end{enumerate}
\end{lemma}

\subsection{Primal-dual construction}

We now describe our method for constructing the matrix pair $(\Bwit,
\Dwit)$.  Recalling that \mbox{$\Joint = \bigcup_{\pind=1}^\numreg
\Sset(\bstar{\pind})$} denotes the union of supports of the true
regression vectors, let $\Jointcom$ denote the complement of
$\{1,\ldots, \pdim\} \backslash \Joint$.  With this notation,
Figure~\ref{FigWitness} provides the four steps of the primal-dual
witness construction.

\begin{figure}[h]
\begin{center}
\framebox[.95\textwidth]{\parbox{0.9\textwidth}{
\noindent {\bf{Primal-dual witness construction:}}
\begin{enumerate}
\item[(A)] First, we solve the restricted program
\begin{equation}
\label{EqnRestricted}
\Bwit = \arg \min_{ B \in \real^{\pdim \times \numreg}, B_{\Jointcom}
  = 0 } \big \{ \frac{1}{2 \numobs} \sum_{\pind=1}^\numreg \|
  \yobs{\pind} - \Xmat{\pind} \bvec{\pind} \|_2^2 + \relaxn
  \bnorm{B}{1}{\infty} \big \}.
\end{equation}
Given our assumption that the $|\Joint| \times |\Joint|$ sub-matrices
$\inprod{\Xmatf{\pind}{\Joint}}{\Xmatf{\pind}{\Joint}}$ are
invertible, the solution to this convex program is unique.  Moreover,
note that $\Bwit_\Jointcom = 0$ by construction.
\item[(B)] We choose $\Dwit_\Joint \in \real^{|\Joint| \times
\numreg}$ as an element of the subdifferential $\partial
\bnorminf{\Bwit_\Joint}$.
\item[(C)] Using the optimality conditions associated with the
original convex program~\eqref{EqnBlockReg}, we then solve for the
matrix $\Dwit_\Jointcom$, and verify that its rows satisfy the
\emph{strict dual feasibility} condition
\begin{eqnarray}
\label{EqnStrictDual}
\sum_{\pind=1}^\numreg |\dwit{\pind}_\jind| & < & 1 \qquad \mbox{for
all $\jind \in \Jointcom$.}
\end{eqnarray}
\item[(D)] A final (optional) step is to verify that $\Bwit_\Joint$
satisfies the \emph{sign consistency} conditions $\sign(\Bwit_\Joint)
= \sign(\Bstar_\Joint)$.

\end{enumerate}
}}
\caption{Steps in the primal-dual witness construction.  Steps (A) and
(B) are straightforward; the main difficulties lie in verifying the
strict dual feasibility and sign consistency conditions stated in step
(C) and (D).}
\label{FigWitness}
\end{center}
\end{figure}

\noindent The following lemma summarizes the utility of the
primal-dual witness method:
\begin{lemma}
\label{LemKey}
Suppose that for each $\pind = 1, \ldots, \numreg$, the $|\Joint|
\times |\Joint|$ sub-matrix
$\inprod{\Xmatf{\pind}{\Joint}}{\Xmatf{\pind}{\Joint}}$ is invertible.
Then for any $\relaxn > 0$, we have the following correspondences:
\begin{enumerate}
\item[(i)] If steps (A) through (C) of the primal-dual construction
succeed, then $(\Bwitf{\Joint}, 0) \in \real^{\pdim \times \numreg}$ is
the unique optimal solution of the original convex
program~\eqref{EqnBlockReg}.
\item[(ii)] Conversely, suppose that there is a solution $\Best \in
\real^{\pdim \times \numreg} $ to the convex
program~\eqref{EqnBlockReg} with support contained within $\Joint$.
Then steps (A) through (C) of the primal-dual witness construction
succeed.
\end{enumerate}
\end{lemma}
\noindent We provide the proof of Lemma~\ref{LemKey} in
Appendix~\ref{AppLemKey}.  It is convex-analytic in nature, based on
exploiting the subgradient optimality conditions associated with both
the restricted convex program~\eqref{EqnRestricted} and the original
program~\eqref{EqnBlockReg}, and performing some algebra to
characterize when the convex program recovers the correct signed
support.  Lemma~\ref{LemKey} lies at the heart of all three of our
theorems.  In particular, the positive results of
Theorem~\ref{ThmDetDesign}, Theorem~\ref{ThmGauss} and
Theorem~\ref{ThmPhase}(a) are based on claims (i) and (iii), which
show that it is sufficient to verify that the primal-dual witness
construction succeeds with high probability.  The negative result of
Theorem~\ref{ThmPhase}(b), in contrast, is based on part (ii), which
can be restated as asserting that if the primal-dual witness
construction fails, then no solution has support contained with
$\Joint$.

Before proceeding to the proofs themselves, we introduce some
additional notation and develop some auxiliary results concerning the
primal-dual witness procedure, to be used in subsequent development.
With reference to steps (A) and (B), we show in
Appendix~\ref{AppLemKey} that unique solution $\Bwit_\Joint$ has the
form
\begin{equation}
\label{EqnDefnVecWit}
\Bwit_\Joint \, = \, \Bstar_\Joint + \Uvar_\Joint,
\end{equation}
where the matrix $\Uvar_\Joint \in \real^{|\Joint|
\times \numreg}$ has columns
\begin{eqnarray}
\label{EqnDefnUvar}
\Nuvar{\pind} & \define & \big ( \frac{1}{\numobs}
\inprod{\Xmatf{\pind}{\Joint}}{\Xmatf{\pind}{\Joint}}\big )^{-1}
\biggr [ \frac{1}{\numobs}
\inprod{\Xmatf{\pind}{\Joint}}{\wnoise{\pind}} - \relaxn
\dwitf{\pind}{\Joint} \biggr], \qquad \mbox{for $\pind = 1, \ldots,
\numreg$,}
\end{eqnarray}
and $\dwitf{\pind}{\Joint}$ is the $\pind^{th}$ column of the
sub-gradient matrix $\Dwitf{\Joint}$.

With reference to step (C), we obtain the candidate dual solution
$\Dwitf{\Jointcom} \in \real^{|\Jointcom| \times \numreg}$ as follows.
For each $\pind = 1, \ldots, \numreg$, let
$\proj_{\Xmatf{\pind}{\Joint}}$ denote the orthogonal projection onto
the range of $\Xmatf{\pind}{\Joint}$.  Using the sub-matrix
$\Dwitf{\Joint} \in \real^{|\Joint| \times \numreg}$ obtained from
step (B), we define column $\pind$ of the matrix $\Dwitf{\Jointcom}$
as follows:
\begin{eqnarray}
\label{EqnDefnDualwitOff}
\dwitf{\pind}{\Jointcom} & = & \frac{1}{\relaxn \numobs}
\binprod{\Xmatf{\pind}{\Jointcom} } {
(I-\proj_{\Xmatf{\pind}{\Joint}}) \wnoise{\pind} } + \frac{1}{\numobs}
\binprod{\Xmatf{\pind}{\Jointcom}} {\Xmatf{\pind}{\Joint}
(\frac{1}{\numobs}
\inprod{\Xmatf{\pind}{\Joint}}{\Xmatf{\pind}{\Joint}})^{-1}
\dwitf{\pind}{\Joint}} \quad \mbox{for $\pind = 1, \ldots,
\numreg$. $\qquad$}
\end{eqnarray}
See the end of Appendix~\ref{AppLemKey} for derivation of this
condition.

Finally, in order to further simplify notation in our proofs, for each
$\jind \in \Jointcom$, we define the random variable
\begin{eqnarray}
  \label{EqnDefnV}
\Vvar_\jind & \define &  \sum_{\pind=1}^\numreg |\dwitf{\pind}{\jind}|
\end{eqnarray}
With this notation, the strict dual feasibility
condition~\eqref{EqnStrictDual} is equivalent to the event \mbox{$\{
\max \limits_{\jind \in \Jointcom} \Vvar_\jind < 1 \}$.}

\section{Proof of Theorem~\ref{ThmDetDesign}}

We begin by establishing a set of sufficient conditions for
deterministic design matrices, as stated in
Theorem~\ref{ThmDetDesign}.

\subsection{Establishing strict dual feasibility}

We begin by obtaining control on the probability of the event
$\Event(\Vvar)$, so as to show that step (C) of the primal-dual
witness construction succeeds.  Recall that
$\proj_{\Xmatf{\pind}{\Joint}}$ denotes the orthogonal projection onto
the range space of $\Xmatf{\pind}{\Joint}$, and the
definition~\eqref{EqnMutIncoDet} of the incoherence parameter
$\incopar \in (0,1]$.  By the mutual incoherence
condition~\eqref{EqnMutIncoDet}, we have
\begin{eqnarray}
\label{EqnKeyInco}
\max_{\jind \in \Jointcom} \biggr \{ \sum_{\pind=1}^\numreg \big
|\frac{1}{\numobs} \binprod{\Xmatf{\pind}{\jind}}{
  \Xmatf{\pind}{\Joint} \; \big(\inprod{\frac{1}{\numobs}
    \Xmatf{\pind}{\Joint}} {\Xmatf{\pind}{\Joint}} \big)^{-1}
  \dwitf{\pind}{\Joint} } \biggr \} & \leq & 1-\incopar,
\end{eqnarray}
where we have used the fact that
$\sum_{\pind=1}^\numreg |\dwit{\pind}_\kind| = 1$ for each $\kind \in \Joint$.
Recalling that
\mbox{$\Vvar_\jind = \sum_{\pind=1}^\numreg |\dwitf{\pind}{\jind}|$}
and using the definition~\eqref{EqnDefnDualwitOff}, we have by
triangle inequality
\begin{equation*}
\Prob[\max_{\jind \in \Jointcom} \Vvar_\jind > 1] \; \leq \;
\mprob[\Aevent(\incopar)],
\end{equation*}
where we have defined the event
\begin{eqnarray}
\Aevent(\incopar) & \defn & \big \{ \max_{\jind \in \Jointcom}
    \sum_{\pind=1}^\numreg \big | \frac{1}{\relaxn \numobs}
    \inprod{\Xmatf{\pind}{\jind}}{ (I - \proj_{\Xmatf{\pind}{\Joint}})
    \wnoise{\pind}} \big | \geq \incopar\big \}.
\end{eqnarray}

To analyze this remaining probability, for each index $\pind = 1,
\ldots, \numreg$ and $\jind \in \Jointcom$, define the random variable
\begin{eqnarray}
\label{EqnZeroMean}
\Wnew{\pind}{\jind} & \defn & \frac{1}{\relaxn \numobs}
\binprod{\Xmatf{\pind}{\jind}}{(I - \proj_{\Xmatf{\pind}{\Joint}})}
\wnoise{\pind}.
\end{eqnarray}
Since the elements of the $\numobs$-vector $\wnoise{\pind}$ follow a
$N(0, \sigma^2)$ distribution, the variable $\Wnew{\pind}{\jind}$ is
zero-mean Gaussian with variance $\frac{\sigma^2}{\relaxn^2 \,
\numobs^2} \binprod{\Xmatf{\pind}{\jind}}{(I -
\proj_{\Xmatf{\pind}{\Joint}}) \Xmatf{\pind}{\jind}}$.  Since
$\|\Xmatf{\pind}{\jind}\|_2^2 \leq 2 \numobs$ by assumption and $(I -
\proj_{\Xmatf{\pind}{\Joint}})$ is an orthogonal projection matrix,
the variance of each $\Wnew{\pind}{\jind}$ is upper bounded by
$\frac{2 \sigma^2}{\relaxn^2 \numobs}$.  Consequently, for any choice
of sign vector $\mysign \in \{-1,+1\}^\numreg$, the variance of the
zero-mean Gaussian $\sum_{\pind=1}^\numreg \mysign_\pind
\Wnew{\pind}{\jind}$ is upper bounded by $\frac{2 \numreg
\sigma^2}{\relaxn^2 \numobs}$.

Consequently, by taking the union bound over all sign vectors
and over indices $\jind \in \Jointcom$, we have
\begin{eqnarray*}
\Prob[\Aevent(\incopar)] \; = \; \Prob \big[\max_{\jind \in \Jointcom}
     \max_{\mysign \in \{-1,+1\}^\numreg} \sum_{\pind=1}^\numreg
     \mysign_\pind \Wnew{\pind}{\jind} > \incopar\big] & \leq & 2 \,
     \exp \big( -\frac{\relaxn^2 \numobs \incopar^2}{4 \numreg
     \sigma^2} + \numreg + \log \pdim \big).
\end{eqnarray*}
With the choice $\relaxn^2 \geq \frac{4 \keypar \sigma^2}{\incopar^2}
\frac{\numreg^2 + \numreg \log (\pdim)}{\numobs}$ for some $\keypar >
1$, we conclude that
\begin{eqnarray*}
\mprob[\Event(\Vvar)] & \geq & 1 - 2 \exp(-(\keypar-1) [\numreg + \log
\pdim]) \; \rightarrow \; 1.
\end{eqnarray*}
By Lemma~\ref{LemKey}(i), this event implies the uniqueness of the
solution $\Best$, and moreover the inclusion of the supports
$\Sset(\Best) \subseteq \Sset(\Bstar)$, as claimed.

\subsection{Establishing $\ell_\infty$ bounds}

We now turn to establishing the claimed
$\ell_\infty$-bound~\eqref{EqnEllinfDet} on the difference $\Best -
\Bstar$.  We have already shown that this difference is exactly zero
for rows in $\Jointcom$; it remains to analyze the difference
$\Uvar_\Joint = \Best_\Joint - \Bstar_\Joint$.  It suffices to prove
the $\ell_\infty$ bound for the columns $\Uvarf{\pind}{\Joint}$
separately, for each $\pind =1, \ldots, \numreg$.

We split the analysis of the random variable $\max_{\jind \in \Joint}
|\Uvarf{\pind}{\jind}|$ into two terms, based on the form of $\Uvar$
from equation~\eqref{EqnDefnUvar}, one involving the dual variables
$\dwitf{\pind}{\Joint}$, and the other involving the observation noise
$\wnoise{\pind}$, as follows:
\begin{eqnarray*}
\max_{\jind \in \Joint} |\Uvarf{\pind}{\jind}| & \leq &
    \underbrace{\big \| \big(\frac{1}{\numobs}
    \inprod{\Xmatf{\pind}{\Joint}} {\Xmatf{\pind}{\Joint}} \big)^{-1}
    \frac{1}{\numobs} \inprod{\Xmatf{\pind}{\Joint}}{\wnoise{\pind}}
    \big \|_\infty} + \underbrace{ \big \| \big(
    \inprod{\frac{1}{\numobs}
    \Xmatf{\pind}{\Joint}}{\Xmatf{\pind}{\Joint}} \big )^{-1} \relaxn
    \dwitf{\pind}{\Joint} \big \|_\infty.}  \\
& & \quad \qquad \qquad \Uterma^\pind \qquad \qquad \qquad \qquad
\qquad \qquad \Utermb^\pind
\end{eqnarray*}
The second term is easy to control: from the characterization of the
subdifferential (Lemma~\ref{LemSubdiff}), we have
$\norm{\dwitf{\pind}{\Joint}}_{\infty} \leq 1$, so that $\Utermb^\pind
\leq \relaxn \matsnorm{( \inprod{\frac{1}{\numobs}
\Xmatf{\pind}{\Joint}}{\Xmatf{\pind}{\Joint}} )^{-1}}{\infty} \; \leq
\; \Dmax \relaxn$.

Turning to the first term $\Uterma^\pind$, we note that since
$\Xmatf{\pind}{\Joint}$ is fixed, the $|\Joint|$-dimensional random
vector $Y \defn \big(\binprod{\frac{1}{\numobs}
\Xmatf{\pind}{\Joint}}{ \Xmatf{\pind}{\Joint}} \big)^{-1}
\frac{1}{\numobs} \inprod{\Xmatf{\pind}{\Joint}}{\wnoise{\pind}}$ is
zero-mean Gaussian, with covariance $\frac{1}{\numobs}
\big(\binprod{\frac{1}{\numobs}
\Xmatf{\pind}{\Joint}}{\Xmatf{\pind}{\Joint}} \big)^{-1}$.  Therefore,
we have $\var(Y_\jind) \leq \frac{1}{\Cmin \numobs}$, and can use this
in standard Gaussian tail bounds.  By applying the union bound twice,
first over $\jind \in \Joint$, and then over $\pind \in \{1, 2,
\ldots, \numreg \}$, we obtain 

\begin{eqnarray*}
\mprob[ \max_{\pind = 1, \ldots, \numreg} \Uterma^\pind \geq t] & \leq
& 2 \exp(-t^2 \numobs \Cmin/(2) + \log (\numreg \spindex) + \log
\numreg),
\end{eqnarray*}
where we have used the fact that $|\Joint| \leq \numreg \spindex$.
Setting $t = \keypar\; \sqrt{\frac{4 \log (\numreg \spindex)}{\Cmin
\numobs}}$ yields that
\begin{eqnarray*}
\max_{\pind =1, \ldots, \numreg} \quad \max_{\jind \in \Joint}
|\Uvarf{\pind}{\jind}| & \leq & \keypar \sqrt{\frac{4}{\Cmin} \;
  \frac{ \log \numreg \spindex}{\numobs}} + \Dmax \relaxn \; = \,:
\Bou,
\end{eqnarray*}
with probability greater than $1- 2\exp(-(\keypar^2-1)\log (\numreg
\spindex))$, as claimed.

Finally, to establish support recovery, recall that we proved above
that $\Nuvar{\pind}$ is bounded by $\Bou$. Hence, as long as $\betamin
> \Bou$, then we are guaranteed that if $\Bstar^\pind_\jind \neq 0$,
then $\Best^\pind_\jind \neq 0$.

%%%%%%%%%%%%%%%%%%%%%%%%%%%%%%%%%%%%%%%%%%%%%%%%%%%%%%%%%%%%%%%%%%%%%%%%%%%%%

\section{Proof of Theorem~\ref{ThmGauss}}
\label{SecProofThmGauss}

We now turn to the proof of Theorem~\ref{ThmGauss}, providing
sufficient conditions for general Gaussian ensembles.  Recall that for
$\pind = 1, 2, \ldots, \numreg$, each $\Xmat{\pind} \in \real^{\numobs
\times \pdim}$ is a random design matrix, with rows drawn
i.i.d. from a zero-mean Gaussian with $\pdim \times \pdim$ covariance
matrix $\Covmati{\pind}$.

\subsection{Establishing strict dual feasibility}
\label{SubSecGaussRand}
Recalling that \mbox{$\Vvar_\jind = \sum_{\pind=1}^\numreg
  |\dwitf{\pind}{\jind}|$} and using the
definition~\eqref{EqnDefnDualwitOff}, we have the decomposition
\begin{eqnarray*}
\max_{\jind \in \Jointcom} |\Vvar_\jind| & \leq &
    \underbrace{\max_{\jind \in \Jointcom} \sum_{\pind=1}^\numreg
    |\frac{1}{\relaxn \numobs} \binprod{ \Xmatf{\pind}{\jind} } {
    (I-\proj_{\Xmatf{\pind}{\Joint}}) \wnoise{\pind} }| } \quad +
    \quad
\underbrace{\max_{\jind \in \Jointcom} \sum_{\pind=1}^\numreg \big
  |\frac{1}{\numobs} \binprod{\Xmatf{\pind}{\jind}}
  {\Xmatf{\pind}{\Joint} (\frac{1}{\numobs}
    \inprod{\Xmatf{\pind}{\Joint}}{\Xmatf{\pind}{\Joint}})^{-1}
    \dwitf{\pind}{\Joint}} \big|.}  \\
& & \qquad \qquad \qquad \quad \Mterm_1 \qquad \qquad \qquad \qquad \qquad \qquad
\qquad \qquad \qquad \quad \Mterm_2
\end{eqnarray*}
In order to show that $\max_{\jind \in \Jointcom} |\Vvar_\jind| < 1$
with high probability, we deal with each of these two terms in turn,
showing that $\Mterm_1 < \incopar/2$, and $\Mterm_2 < 1 - \incopar/2$,
both with high probability.

In order to bound $\Mterm_1$, we require the following condition on
the columns of the design matrices:
\begin{lemma}
\label{LemColControl}
Let $\sigma_{\max} = \max_i \Sigma$. For $\numobs > 2 \log (\numreg \pdim)$, each column of the design
matrices $\Xmat{\pind}, \pind = 1, \ldots, \numreg$ has controlled
$\ell_2$-norm:
\begin{eqnarray}
\mprob \big[ \max_{\pind = 1, \ldots, \numreg} \;\max_{\jind=1,
\ldots, \pdim} \|\Xmatf{\pind}{\jind}\|_2^2 \leq 2 \sigma_{\max} \numobs \big] &
\leq & 2 \exp \big( -\frac{\numobs}{2} + \log (\pdim \numreg) \big) \;
\rightarrow 0.
\end{eqnarray}
\end{lemma}
\noindent This claim follows immediately by union bound and
concentration results for $\chi^2$-variates; in particular, the
bound~\eqref{EqnMyChiUpper} in Appendix~\ref{AppLargeDev}.

Under the condition of Lemma~\ref{LemColControl}, each variable
$\Wnew{\pind}{\jind} \defn \frac{1}{\relaxn \numobs} \binprod{
\Xmatf{\pind}{\jind} }{(I-\proj_{\Xmatf{\pind}{\Joint}})
\wnoise{\pind}}$ is zero-Gaussian, with variance at most $\frac{2
\sigma^2}{\relaxn^2 \numobs}$.  Consequently, for any choice of signs
$\mysign \in \{-1,+1\}^\numreg$, the vector $\sum_{\pind=1}^\numreg
\mysign_\pind \Wnew{\pind}{\jind}$ is zero-mean Gaussian, with
variance at most $\frac{2 \sigma^2 \numreg}{\relaxn^2 \numobs}$.
Therefore, for any $t > 0$, we have
\begin{eqnarray*}
\mprob[\max_{\jind \in \Jointcom} \sum_{\pind=1}^\numreg
  |\Wnew{\pind}{\jind}| \geq t] & = & \mprob[\max_{\jind \in
    \Jointcom} \max_{\mysign \in \{-1,+1\}^\numreg}
  \sum_{\pind=1}^\numreg \mysign_\pind \Wnew{\pind}{\jind} \geq t] \\
& \leq & 2 \; \exp \big( -\frac{\relaxn^2 \numobs}{4 \sigma^2 \numreg}
t^2 + \numreg  + \log \pdim \big)
\end{eqnarray*}
Setting $t = \incopar/2$ yields that 
\begin{eqnarray*}
\mprob[\Mterm_1 \geq \incopar/2] & \leq & 2 \; \exp \big(
-\frac{\relaxn^2 \numobs}{16 \sigma^2 \numreg} \incopar^2 + \numreg + \log
\pdim \big).
\end{eqnarray*}

\begin{lemma}
\label{LemPopMutInco}
Suppose that the design covariance matrices $\Covmati{\pind}, \pind =
1, \ldots, \numreg$ satisfy the mutual incoherence
condition~\eqref{EqnMutIncoDet}.  Then we have
\begin{eqnarray}
\label{EqnMtermPrime}
\Mterm_2 & \leq & (1-\incopar) + \underbrace{\max_{\jind \in
\Jointcom} \sum_{\pind=1}^\numreg \big| \frac{1}{\numobs}
\binprod{\Yif{\pind}{\jind}} {\Xmatf{\pind}{\Joint}
(\inprod{\frac{1}{\numobs} \Xmatf{\pind}{\Joint}}
{\Xmatf{\pind}{\Joint}})^{-1} \dwitf{\pind}{\Joint}} \big |}, \\
& & \qquad \qquad \qquad \qquad \qquad \Mterm'_2 \nonumber
\end{eqnarray}
where each random vector $\Yif{\pind}{\jind} \in \real^{\numobs \times
1}$ has i.i.d. $N(0,1)$ entries, and is independent of
$\wnoise{\pind}$ and $\Xmatf{\pind}{\Joint}$.
\end{lemma}
\noindent See Appendix~\ref{AppLemPopMutInco} for the proof of this
claim.   \\

It remains to show that the random variable $\Mterm'_2$ defined in
equation~\eqref{EqnMtermPrime} is upper bounded by $\incopar/2$ with
high probability.  Conditioning on $\Xmatf{\pind}{\Joint}$ and
$\wnoise{\pind}$, the scalar random variable $\frac{1}{\numobs}
\binprod{\Yif{\pind}{\jind}} {\Xmatf{\pind}{\Joint}
(\inprod{\frac{1}{\numobs} \Xmatf{\pind}{\Joint}}
{\Xmatf{\pind}{\Joint}})^{-1} \dwitf{\pind}{\Joint}}$ is zero-mean
Gaussian, with variance upper bounded as
\begin{eqnarray*}
\frac{1}{\numobs} \binprod{\dwitf{\pind}{\Joint}}
{(\inprod{\frac{1}{\numobs} \Xmatf{\pind}{\Joint}}
{\Xmatf{\pind}{\Joint}})^{-1} \dwitf{\pind}{\Joint}} & \leq &
\frac{\|\dwitf{\pind}{\Joint}\|^2_2 }{\Cmin \; \numobs}.
\end{eqnarray*}
Recalling that $\sum_{\pind=1}^\numreg |\dwit{\pind}_\kind| = 1$, for
any choice of signs $\mysign \in \{-1,+1\}^\numreg$, the variable
\begin{equation*}
\sum_{\pind=1}^\numreg \mysign_\pind \frac{1}{\numobs}
\binprod{\Yif{\pind}{\jind}} {\Xmatf{\pind}{\Joint}
(\inprod{\frac{1}{\numobs} \Xmatf{\pind}{\Joint}}
{\Xmatf{\pind}{\Joint}})^{-1} \dwitf{\pind}{\Joint}}
\end{equation*}
is zero-mean Gaussian, with variance at most $\frac{\numreg \spindex
}{\Cmin \; \numobs}$.  Therefore, we have
\begin{eqnarray*}
\mprob[\Mterm_2' \geq \incopar/2] & \leq & \mprob \big[\max_{\jind \in
\Jointcom} \max_{\mysign \in \{-1,+1 \}^\numreg}
|\sum_{\pind=1}^\numreg \mysign_\pind \frac{1}{\numobs}
\binprod{\Yif{\pind}{\jind}} {\Xmatf{\pind}{\Joint}
(\inprod{\frac{1}{\numobs} \Xmatf{\pind}{\Joint}}
{\Xmatf{\pind}{\Joint}})^{-1} \dwitf{\pind}{\Joint}}| \geq \incopar/2
\big] \\
& \leq & 2 \exp( -\frac{\Cmin \numobs}{8 \numreg \spindex } 
\incopar^2 + \numreg + \log \pdim).
\end{eqnarray*}

This probability vanishes faster than $2 \exp \big \{ -\tautwo (
\numreg + \log \pdim) \big \} \rightarrow 0$, as long as
\begin{eqnarray*}
\numobs & > & \frac{8 \, \tautwo \; \numreg }{\Cmin \incopar^2} \,
\spindex \big ( \numreg + \log \pdim \big)
\end{eqnarray*}

%%%%%%%%%%%%%%%%%%%%%%%%%%%%%%%%%%%%%%%%%%%%%%%%%%%%%%%%%%%%%%%%%%%%%

\subsection{Establishing $\ell_\infty$ bounds}

We now turn to establishing the claimed
$\ell_\infty$-bound~\eqref{EqnEllinfDet} on the difference $\Best -
\Bstar$. As in the analogous portion of the proof of
Theorem~\ref{ThmDetDesign}, we use the decomposition
\begin{eqnarray*}
\max_{\jind \in \Joint} |\Uvarf{\pind}{\jind}| & \leq &
    \underbrace{\big \| \big(\frac{1}{\numobs}
    \inprod{\Xmatf{\pind}{\Joint}} {\Xmatf{\pind}{\Joint}} \big)^{-1}
    \frac{1}{\numobs} \inprod{\Xmatf{\pind}{\Joint}}{\wnoise{\pind}}
    \big \|_\infty} + \underbrace{ 
\big \| \big(
    \inprod{\frac{1}{\numobs}
    \Xmatf{\pind}{\Joint}}{\Xmatf{\pind}{\Joint}} \big )^{-1} \relaxn
    \dwitf{\pind}{\Joint} \big \|_\infty.}  \\
& & \quad \qquad \qquad \Uterma^\pind \qquad \qquad \qquad \qquad
\qquad \qquad \Utermb^\pind
\end{eqnarray*}
In the setting of random design matrices, a bit more work is required
to control these terms.

Beginning with the second term, by triangle inequality, we have
\begin{eqnarray*}
\Utermb^\pind & \leq & \big \| \big[\big( \inprod{\frac{1}{\numobs}
    \Xmatf{\pind}{\Joint}}{\Xmatf{\pind}{\Joint}} \big )^{-1} -
    (\Covmatif{\pind}{\Joint \Joint})^{-1} \big)\relaxn
    \dwitf{\pind}{\Joint} \big] \|_\infty + \big \|
    \big(\Covmatif{\pind}{\Joint \Joint})^{-1} \relaxn
    \dwitf{\pind}{\Joint} \big \|_\infty \\
& \leq & \matsnorm{\big[\big( \inprod{\frac{1}{\numobs}
    \Xmatf{\pind}{\Joint}}{\Xmatf{\pind}{\Joint}} \big )^{-1} -
    (\Covmatif{\pind}{\Joint \Joint})^{-1} \big)}{2} \relaxn
    \sqrt{\spindex} + \Dmax \relaxn
\end{eqnarray*}
where we have used the facts that $\|\dwitf{\pind}{\Joint}\|_2 \leq
\sqrt{\spindex}$, since $\dwitf{\pind}{\Joint}$ belongs to the
sub-differential of the block $\ell_1/\ell_\infty$ norm (see
Lemma~\ref{LemSubdiff}) so that $\sum_{\pind=1}^r |
\dwitf{\pind}{\kind} | \leq 1$ for all $\kind \in \Joint$.  By,
concentration bounds for eigenvalues of Gaussian random matrices (see
equation~\eqref{genmatnorm:2} in Appendix~\ref{AppLargeDev}), we
conclude that
\begin{eqnarray*}
\Utermb^\pind & \leq & 4 \relaxn \sqrt{\spindex} \;
\sqrt{\frac{\spindex}{\numobs}} + \Dmax \relaxn \; = \;
\relaxn \, \big[\frac{4 \spindex}{\sqrt{\numobs}} + \Dmax
\big].
\end{eqnarray*}

Now consider the first term $\Uterma^\pind$: if we condition on
$\Xmatf{\pind}{\Joint}$, then the $|\Joint|$-dimensional random vector
$Y \defn \big(\binprod{\frac{1}{\numobs} \Xmatf{\pind}{\Joint}}{
\Xmatf{\pind}{\Joint}} \big)^{-1} \frac{1}{\numobs}
\inprod{\Xmatf{\pind}{\Joint}}{\wnoise{\pind}}$ is zero-mean Gaussian,
with covariance $\frac{1}{\numobs} \big(\binprod{\frac{1}{\numobs}
\Xmatf{\pind}{\Joint}}{\Xmatf{\pind}{\Joint}} \big)^{-1}$.  By
concentration bounds for eigenvalues of Gaussian random matrices (see
equation~\eqref{genmatnorm:2} in Appendix~\ref{AppLargeDev}), we have

\begin{eqnarray*}
\frac{1}{\numobs} \matsnorm{\big(\binprod{\frac{1}{\numobs}
\Xmatf{\pind}{\Joint}}{\Xmatf{\pind}{\Joint}} \big)^{-1}}{2} & \leq &
\frac{1}{\numobs} \Big \{ \matsnorm{
\big(\binprod{\frac{1}{\numobs}
\Xmatf{\pind}{\Joint}}{\Xmatf{\pind}{\Joint}} \big)^{-1} -
(\Covmatif{\pind}{\Joint \Joint})^{-1}}{2} +
\matsnorm{(\Covmatif{\pind}{\Joint \Joint})^{-1}}{2} \Big \}\\
& \leq & \frac{4}{\Cmin \numobs} \, \sqrt{\frac{\numreg
\spindex}{\numobs}} + \frac{1}{\Cmin \numobs} \\
& \leq & \frac{5}{\Cmin \numobs},
\end{eqnarray*}
since $\numreg \spindex/\numobs \leq 1$.  Therefore, we have shown
that the variance of each element of $Y$ is upper bounded by $5/(\Cmin
\numobs)$, so that we can apply standard Gaussian tail bounds.  By
applying the union bound twice, first over $\jind \in \Joint$, and
then over $\pind \in \{1, 2, \ldots, \numreg \}$, we obtain
\begin{eqnarray*}
\mprob[ \max_{\pind = 1, \ldots, \numreg} \Uterma^\pind \geq t] & \leq
& 2 \exp(-t^2 \numobs \Cmin/(50) + \log |\Joint| + \log \numreg).
\end{eqnarray*}
Setting $t = \keypar\; \sqrt{\frac{100 \log (\numreg \spindex)}{\Cmin
\numobs}}$ yields that
\begin{eqnarray*}
\mprob[ \max_{\pind = 1, \ldots, \numreg} \Uterma^\pind \geq t] & \leq
& 2 \exp \big \{ -2 \keypar^2 \log (\numreg \spindex) + \log (\numreg
\spindex) + \log \numreg \big \} \\
& \leq & 2 \exp \big \{ -2 (\keypar^2 -1)\log (\numreg \spindex) \big \},
\end{eqnarray*}
where we have used the fact that $|\Joint| \leq \numreg \spindex$.
Combining the pieces, we conclude that
\begin{eqnarray*}
\max_{\jind \in \Joint} |\Uvarf{\pind}{\jind}| & \leq & \keypar\;
\sqrt{\frac{100 \log (\numreg \spindex)}{\Cmin \numobs}} + \relaxn \,
\big[\frac{4 \spindex}{\sqrt{\numobs}} + \Dmax \big],
\end{eqnarray*}
with probability greater than 
\begin{equation*}
1- 2 \exp \big \{ -2 (\keypar^2 -1)\log (\numreg \spindex) \big \} -
c_1 \exp(-c_2 \numobs), 
\end{equation*}
as claimed.

%%%%%%%%%%%%%%%%%%%%%%%%%%%%%%%%%%%%%%%%%%%%%%%%%%%%%%%%%%%%%%%%%%%%%%%%%
%%%%%%%%%%%%%%%%%%%%%%%%%%%%%%%%%%%%%%%%%%%%%%%%%%%%%%%%%%%%%%%%%%%%%%%%%

\section{Proof of Theorem~\ref{ThmPhase}}

We now turn to the proof of the phase transition predicted by
Theorem~\ref{ThmPhase}, which applies to random design matrices
$\Xmat{1}$ and $\Xmat{2}$ drawn from the standard Gaussian ensemble.
This proof requires significantly more technical work than the
preceding two proofs, since we need to control all the constants
exactly, and to establish both necessary and sufficient conditions on
the sample size. 

%%%%%%%%%%%%%%%%%%%%%%%%%%%%%%%%%%%%%%%%%%%%%%%%%%%%%%%%%%%%%%%%%
\subsection{Proof of Theorem~\ref{ThmPhase}(a)}

We begin with the achievability result.  Our proof parallels that of
Theorems~\ref{ThmDetDesign} and~\ref{ThmGauss}, in that we first
establish strict dual feasibility, and then turn to proving
$\ell_\infty$ bounds and exact support recovery.

\subsubsection{Establishing strict dual feasibility}

Recalling that $\Vvar_\jind = \sum_{\pind=1}^2 |\dwitf{\pind}{\jind}|$,
we have
\begin{eqnarray*}
  \max_{\jind \in \Jointcom} |\Vvar_\jind| \, \leq \, \Mterm_1 \, + \,
  \Mterm_2,
\end{eqnarray*}
where the random variables $\Mterm_1$ and $\Mterm_2$ were defined at
the start of Section~\ref{SubSecGaussRand}.  In order to prove that
$\max_{\jind \in \Jointcom} |\Vvar_\jind| < 1$ with high probability
for the values of $\numobs$, $\spindex$, and $\pdim$, we will first
establish that $\Mterm_1 < \epsilon/2$ and $\Mterm_2 < 1 - \epsilon$
for an appropriately chosen value of $\epsilon$.

By the results from the previous section, we have $M_1 < \epsilon/2$
with probability
\begin{eqnarray*}
\mprob[\max_{\jind \in \Jointcom} \sum_{\pind=1}^2
  |\Wnew{\pind}{\jind}| \geq \epsilon/2]
& \leq & 2 \; \exp \big( -\frac{\relaxn^2 \numobs \epsilon^2}{32 \sigma^2}
 + 2  + \log \pdim \big)
\end{eqnarray*}

Recall that
\begin{equation*}
  \Mterm_2 = \max_{\jind \in \Jointcom} \sum_{\pind=1}^2 \big
  |\frac{1}{\numobs} \binprod{\Xmatf{\pind}{\jind}}
  {\Xmatf{\pind}{\Joint} (\frac{1}{\numobs}
    \inprod{\Xmatf{\pind}{\Joint}}{\Xmatf{\pind}{\Joint}})^{-1}
    \dwitf{\pind}{\Joint}} \big|,
\end{equation*}
and that $\Xmatf{\pind}{\Jointcom}$ is independent of
$\Xmatf{\pind}{\Joint}$ and $\wnoise{\pind}$. We will show that
$\Mterm_2 < 1 - \epsilon$ with high probability by using results on
Gaussian extrema.  Conditioning on $(\mata_\Joint, \noisea,
\duala_\Joint)$, the random variable $\Yif{\pind}{\jind} =
\frac{1}{\numobs} \binprod{\Xmatf{\pind}{\jind}}
{\Xmatf{\pind}{\Joint} (\frac{1}{\numobs}
\inprod{\Xmatf{\pind}{\Joint}}{\Xmatf{\pind}{\Joint}})^{-1}
\dwitf{\pind}{\Joint}}$ is zero-mean with variance upper-bounded as
  \begin{eqnarray*}
    \frac{1}{\numobs} \binprod{\dwitf{\pind}{\Joint}}
	 {(\inprod{\frac{1}{\numobs} \Xmatf{\pind}{\Joint}}
	 {\Xmatf{\pind}{\Joint}})^{-1} \dwitf{\pind}{\Joint}} & \leq &
	 \matsnorm{(\inprod{\frac{1}{\numobs} \Xmatf{\pind}{\Joint}}
	 {\Xmatf{\pind}{\Joint}})^{-1}}{2}^2
	 \frac{\|\dwitf{\pind}{\Joint}\|^2_2 }{\numobs},
  \end{eqnarray*}
Under the given conditioning, the random variables $\Yif{1}{\jind}$
and $\Yif{2}{\jind}$ are independent and for any sign vector $\mysign
\in \{-1,+1\}^2$, the random variable $\sum_{\pind=1}^2 \mysign_\pind
\Yif{\pind}{\jind}$ is Gaussian, zero-mean with variance upper bounded
as
\begin{eqnarray*}
    \sum_{\pind=1}^2 \matsnorm{(\inprod{\frac{1}{\numobs}
	   \Xmatf{\pind}{\Joint}} {\Xmatf{\pind}{\Joint}})^{-1}}{2}^2
	   \frac{\|\dwitf{\pind}{\Joint}\|^2_2 }{\numobs}
  \end{eqnarray*}
  
By Lemma~\ref{LemRandMat}, $\matsnorm{(\inprod{\frac{1}{\numobs}
\Xmatf{\pind}{\Joint}} {\Xmatf{\pind}{\Joint}})^{-1}}{2}^2 \leq (1 +
\delta)$ with probability at least $1 - c_1 \exp{-c_2 n}$ for
sufficiently large $s$ and $n$ under the given scaling for each
$\pind$. Hence, $\sum_{\pind=1}^2 \mysign_\pind \Yif{\pind}{\jind}$ is
normal, zero-mean, with variance upper bounded as
\begin{eqnarray*}
  \frac{(1 + \delta)}{\numobs} \sum_{\pind=1}^2 \|\dwitf{\pind}{\Joint}\|^2_2
\end{eqnarray*}

Recall that $\Dwitf{\Joint}$ was obtained from Step (B) of the Prima-dual
witness construction. The next lemma provides control
over $\sum_{\pind=1}^2 \|\dwitf{\pind}{\Joint}\|^2_2$.
\begin{lemma}
\label{LemVarBound}
Under the assumptions of Theorem~\ref{ThmPhase} and Corollary~\ref{gapcor}, if $\relaxn^2 \numobs \rightarrow +\infty$ and $\spindex/\numobs
\rightarrow 0$, then $\|\duala_\Joint\|^2_2$ is concentrated: for all
$\delta > 0$, we have that for sufficiently large $\spindex$ and $\numobs$
\begin{subequations}
\begin{eqnarray}
\label{EqnVarLower}
\mprob \big[ \|\duala_\Joint\|^2_2 + \|\dualb_\Joint\|^2_2 \leq
  (1-\delta) \frac{\spindex}{2} \big \{ (4 - 3 \overlap) +
  \frac{1}{\relaxn^2 \spindex} \norm{\bdiff}_2^2 \big \}\big] &
  \rightarrow & 0, \qquad \mbox{and} \\
\label{EqnVarUpper}
\mprob \big[ \|\duala_\Joint\|^2_2 + \|\dualb_\Joint\|^2_2 \geq
  (1+\delta) \frac{s}{2} \big \{(4 - 3 \overlap) + \frac{1}{\relaxn^2
  \spindex} \norm{\bdiff}_2^2 \big \} \big] & \leq & c_1 \exp(
  -c_2 \numobs),
\end{eqnarray}
\end{subequations}
\end{lemma}
\noindent See Appendix~\ref{AppLemVarBound} for the proof of this
claim.

Now, by applying the union bound and using Gaussian tail bounds, we
obtain that the probability $\Prob[\Mterm_2 \geq 1 - \epsilon]$ is
upper bounded by
\begin{equation*}
c_1 \exp(-c_2 n) \, + \, 4 \exp \big ( -(1-\epsilon)^2 \numobs
  /[(1+\delta) \spindex \big \{(4 - 3 \overlap) + \frac{1}{\relaxn^2
  \spindex} \norm{\bdiff}_2^2 \big \}] + \log(\pdim - (2-\overlap)
  \spindex) \big),
\end{equation*}
which goes to $0$ as $\numobs \to \infty$ under the condition
\begin{eqnarray*}
\numobs & > & [(1+\delta) \spindex \big \{(4 - 3 \overlap) +
\frac{1}{\relaxn^2 \spindex} \norm{\bdiff}_2^2 \big
\}]/(1-\epsilon)^2 \log(\pdim - (2-\overlap) \spindex).
\end{eqnarray*}

%%%%%%%%%%%%%%%%%%%%%%%%%%%%%%%%%%%%%%%%%%%%%%%%%%%%%%%%%%%%%%%%%%%%%%

\subsection{Proof of Theorem~\ref{ThmPhase}(b)}

We now turn to the proof of the converse claim in
Theorem~\ref{ThmPhase}.  We establish the claim by contradiction.  We
show that if a solution $\Best$ exists such that $\Best_{\Joint^c} =
0$, then under the stated upper bound on the sample size $\numobs$,
there exists some $\epsilon > 0$ such that $\Prob[\max \limits_{k \in
\Jointcom} ( \abs{\duala_k} + \abs{\dualb_k}) > 1 + \epsilon]$
converges to one.  From the definition~\eqref{EqnDefnDualwitOff}, we
see that conditioned on $(\mata_\Joint, \noisea, \duala_\Joint)$, the
variables $\{\duala_k, k \in \Jointcom \} \}$ are i.i.d. zero-mean
Gaussians, with variance given by
\begin{equation*}
\myvar(\duala_k) \define \norm{\frac{1}{\relaxn \numobs}
\Proj{\Joint^\perp} \noisea - \frac{1}{\numobs} \mata_\Joint
(\frac{1}{\numobs} \mata_\Joint^T \mata_\Joint)^{-1}
\duala_\Joint}_2^2.
\end{equation*}
By orthogonality, we have $\myvar(\duala_k) = \norm{\frac{1}{\relaxn
\numobs} \Proj{\Joint^\perp} \noisea}_2^2 + \norm{\frac{1}{\numobs}
\mata_\Joint (\frac{1}{\numobs} \mata_\Joint^T \mata_\Joint)^{-1}
\duala_\Joint}_2^2$, so that (using the idempotency of projection
operators), we have
\begin{eqnarray}
\label{EqnDefnRandVar}
\myvar(\duala_k) & \geq & \randvar \; \defn \; \max \left \{
\frac{1}{\relaxn^2 \numobs} \; \frac{\|\Proj{\Joint^\perp}
\noisea\|_2^2}{\numobs}, \; \lammin((\frac{1}{\numobs}\mata_\Joint^T
\mata_\Joint)^{-1}) \; \frac{\|\duala_\Joint \|_2^2}{\numobs} \right
\}.
\end{eqnarray}
Note that $\randvar = \randvar(\mata_\Joint, \noisea, \duala_\Joint)$
is a scalar random variable, but fixed under the conditioning.
Turning to the variables $\{\dualb_k, k \in \Jointcom\}$, a similar
argument shows that have $\myvar(\dualb_k) \geq \randvartil$, where
$\randvartil = \randvartil(\matb_\Joint, \noiseb, \dualb_\Joint)$ is
the analogous random variable.

For $\jind \in \Jointcom$, let $\dwitf{1}{\jind} \sim N(0,\randvar)$ and
$\dwitf{2}{\jind} \sim N(0, \randvartil)$.  We then have
\begin{eqnarray*}
  \Prob[\max_{\jind \in \Jointcom} (\abs{\dwitf{1}{\jind}} +
    \abs{\dwitf{2}{\jind}}) > (1 + \epsilon)] & \stackrel{(a)}{\geq} &
    \Prob [\max_{\jind \in \Jointcom} \abs{\dwitf{1}{\jind}} +
    \abs{\dwitf{2}{\jind}} > 1 + \epsilon] \\
  & \stackrel{}{\geq} & \Prob[\max_{\jind \in \Jointcom}
    (\dwitf{1}{\jind} + \dwitf{2}{\jind} ) > 1+\epsilon] \\
  & \stackrel{(b)}{=} & \Prob [\max_{\jind \in \Jointcom} Z_\jind > 1 +
    \epsilon],
\end{eqnarray*}
where $Z_\kind \sim N(0, \randvar + \randvartil)$.  Here inequality
(a) follows because $\randvar$ and $\randvartil$ are lower bounds on
the variances of $\{\dwitf{1}{\jind}, \jind \in \Jointcom\}$ and
$\{\dwitf{2}{\jind}, \jind \in \Jointcom\}$ respectively, and equality
(b) follows since $\dwitf{1}{\kind}$ and $\dwitf{2}{\kind}$ are
independent zero-mean Gaussians with variances $\randvar$ and
$\randvartil$, respectively.

To simplify notation, let $N = |\Jointcom| = \pdim -
(2-\overlap)\spindex$.  By standard results for Gaussian
maxima~\cite{LedTal91}, for any $\delta > 0$, there exists an integer
$N(\delta)$ such that for all $N \geq N(\delta)$,
\begin{eqnarray*}
\Exs[\max_{j \in \Jointcom} Z_j] & \geq & (1-\delta) \sqrt{2 (\randvar
+ \randvartil) \log N}.
\end{eqnarray*}
Moreover, the maximum function is Lipschitz, so that by Gaussian
concentration for Lipschitz functions~\cite{LedTal91,Ledoux01}, for
any $\eta > 0$, we have
\begin{eqnarray*}
\Prob \big[ \max_{j \in \Jointcom} Z_j \leq \Exs[\max_{j \in
\Jointcom} Z_j] -\eta \big] & \leq & \exp \big(-\frac{\eta^2}{2
(\randvar + \randvartil)} \big).
\end{eqnarray*}
Combining these two statements yields that for all $N \geq N(\delta)$,
we have
\begin{eqnarray}
\label{EqnCombine}
  \Prob \big[\max_{j \in \Jointcom} Z_j \leq (1-\delta) \sqrt{ 2
(\randvar + \randvartil) \log N} - \eta \big] & \leq & \exp
\big(-\frac{\eta^2}{2 (\randvar + \randvartil)} \big).
\end{eqnarray}
It remains to show that there exists some $\epsilon > 0$ such that
\mbox{$\mprob[\max_{k \in \Jointcom} Z_k \leq 1 +\epsilon]$} converges
to zero.  \\

\noindent {\bf{Case 1:}} First suppose that $\relaxn^2 \numobs =
\order(1)$.  In this case, we have $\randvar = \Omega \big(
\frac{\|\Proj{\Joint^\perp} \noisea\|_2^2}{\numobs} \big)$. With
probability greater than $1 - c_1 \exp(-c_2 \numobs)$, this quantity
is lower bounded by a constant, using concentration for
$\chi^2$-variates.  In this case, $\sqrt{2 (\randvar + \randvartil)
\log N} - \eta \rightarrow +\infty$ w.h.p., so that the result follows
trivially. \\

\noindent {\bf{Case 2:}} Otherwise, we must have $\relaxn^2 \numobs
\rightarrow +\infty$.  Under this condition, we now establish a lower
bound on $\randvar$ that holds with high probability; it will be seen
that a similar lower bound holds for $\randvartil$.  We begin by
noting the lower bound $\randvar \geq
\frac{\|\duala_\Joint\|_2^2}{\numobs}
\lammin((\frac{1}{\numobs}\mata_\Joint^T \mata_\Joint)^{-1})$.  To
control the minimum eigenvalue, define the event
\begin{equation}
\Tail(\mata_\Joint) \define \big \{ \mata_\Joint \; \mid \;
\lammin((\frac{1}{\numobs}\mata_\Joint^T \mata_\Joint)^{-1}) \geq (1 +
\sqrt{\spindex/\numobs})^{-2} \big \}.
\end{equation}
By standard random matrix concentration arguments (see
Appendix~\ref{AppLargeDev}), for some fixed $c > 0$, we are guaranteed
that \mbox{$\Prob[\Tail^c(\mata_\Joint)] \leq 2 \exp(-c \numobs)$.}
Consequently, conditioned on $\Tail(\mata_\Joint)$, we have
\begin{equation}
\label{EqnFirstLower}
\randvar + \randvartil \; \geq \; \frac{\|\duala_\Joint\|_2^2 + \| \dualb_\Joint \|_2^2}{\numobs} \; (1 +
\sqrt{\spindex/\numobs})^{-2}.
\end{equation}

From Lemma~\ref{LemVarBound}, we note that if $\spindex/\numobs =
o(1)$, then for any $\delta > 0$, we have the lower bound
\begin{equation}
\label{EqnSecondLower}
\randvar + \randvartil \; \geq \; (1-\delta) \frac{\spindex}{2 \numobs} \big
\{ (4 - 3 \overlap) + (\Gaplim(\Bstar,\relaxn))^2 \big \}(1-o(1)).
\end{equation}

\noindent The following result is the final step in the proof of
Theorem~\ref{ThmPhase}(b).
\begin{lemma}
\label{LemCaseOne}
Suppose that $\relaxn^2 \numobs \rightarrow +\infty$.  Under this
condition:
\begin{enumerate}
\item[(a)] If $\frac{\spindex}{\numobs} = \Omega(1)$, then $\Prob
[\max \limits_{k\in \Jointcom} Z_k \leq 2] \rightarrow 0$.
\item[(b)] If $\frac{\spindex}{\numobs} \to 0$, then there exists some
$\epsilon > 0$ such that $\Prob [\max \limits_{k \in \Jointcom} Z_k
\leq 1+\epsilon] \rightarrow 0$.
\end{enumerate}
\end{lemma}
\begin{proof}
(a) If $\frac{\spindex}{\numobs}$ is bounded below by some constant $c
> 0$, then we have 
\begin{equation*}
\randvar \; \geq \; (1-\overlap) \frac{\spindex}{\numobs} \; \geq \;
(1-\overlap) c,
\end{equation*}
which implies that $(\randvar + \randvartil) \log N \rightarrow
+\infty$.  Thus, setting $\delta = 1/4$ and $\eta = \frac{1}{2}
\sqrt{2 (\randvar + \randvartil) \log N}$ in
equation~\eqref{EqnCombine} yields that (for $N$ sufficiently large):
\begin{eqnarray*}
\Prob \big[\max_{k \in \Jointcom} Z_k \leq (\frac{1}{2} -\delta)
  \sqrt{2 (\randvar + \randvartil) \log N} \big] & = & \Prob
  \big[\max_{k \in \Jointcom} Z_k \leq \frac{1}{4} \; \sqrt{2
  (\randvar + \randvartil) \log N} \big] \\
& \leq &  \exp \big(-\frac{\log N}{4} \big) \rightarrow 0.
\end{eqnarray*} 
Since $\frac{1}{4} \sqrt{2 (\randvar + \randvartil) \log N} \geq 2$
for $N$ large enough, the claim follows.

\noindent (b) In this case, we may apply the lower
bound~\eqref{EqnSecondLower}, so that, for any $\delta > 0$, we have
\begin{eqnarray*}
  \randvar + \randvartil \; \geq \; (1-\delta) \frac{\spindex}{2 \numobs} \big
  \{ (4 - 3 \overlap) + (\Gaplim(\Bstar,\relaxn)) \big \}(1-o(1))
\end{eqnarray*}
with high probability. Since $\numobs < (1-\nu) [(4 - 3 \overlap) +
(\Gaplim(\Bstar,\relaxn))] \spindex \log N$ by assumption, we have
\begin{eqnarray*}
  \sqrt{2 (\randvar +\randvartil) \log N} & \geq & (1-o(1)) \;
  \sqrt{(1-\delta) \frac{\spindex}{\numobs} \big \{ (4 - 3 \overlap)
    + (\Gaplim(\Bstar,\relaxn)) \big \}\log N} \\ % & \geq & (1-o(1)) \;
  & \geq & (1-o(1)) \; \sqrt{\frac{1-\delta}{1-\nu}}.
\end{eqnarray*}
Consequently, from equation~\eqref{EqnCombine}, for any $\eta > 0$ and
$\delta > 0$, we have for all $N \geq N(\delta)$,
\begin{eqnarray}
\label{EqnCombineTwo}
  \Prob \big[\max_{k \in \Jointcom} Z_k \leq (1-\delta) \, (1 - o(1))
\; \frac{1}{\sqrt{1-\nu}} - \eta \big] & \leq & \exp
\big(-\frac{\eta^2}{2 (\randvar + \randvartil)} \big).
\end{eqnarray}
Since $\nu > 0$, we may choose $\eta, \delta > 0$ sufficiently small
so that for sufficiently large choices of $(\spindex, \numobs)$, we
have
\begin{equation*}
(1-\delta) \, (1 - o(1)) \frac{1}{\sqrt{1-\nu}} - \eta \; \geq \; 1 +
\epsilon
\end{equation*}
for some $\epsilon > 0$.  Since from Lemma~\ref{LemVarBound}, the
condition $\spindex/\numobs = o(1)$ implies that $\randvar +
\randvartil = o(1)$ w.h.p, we thus conclude that, using these choices
of $\eta$ and $\delta$, we have
\begin{eqnarray*}
\mprob[\max_{k \in \Jointcom} Z_k \leq 1 + \epsilon] & \leq & o(1) +
\exp \big(-\frac{\eta^2}{2 (\randvar + \randvartil)} \big) \;
\rightarrow \; 0,
\end{eqnarray*}
as claimed.
\end{proof}

%%%%%%%%%%%%%%%%%%%%%%%%%%%%%%%%%%%%%%%%%%%%%%%%%%%%%%%%%%%%%%%%%%%%%%%%%%%

\section{Discussion}

In this paper, we provided a number of theoretical results that
provide a sharp characterization of when, and if so by how much the
use of block $\ell_1/\ell_\infty$ regularization actually leads
improvements in statistical efficiency in the problem of multivariate
regression.  As suggested in a body of past work, the use of block
$\ell_1/\ell_\infty$ regularization is well-motivated in many
application contexts.  However, since it involves greater
computational cost than more naive approaches, the question of whether
this greater computational price yields statistical gains is an
important one.

This paper assessed statistical efficiency in terms of the number of
samples required to recover the support exactly; however, one could
imagine studying the same issue for related loss functions (e.g.,
$\ell_2$-loss or prediction loss), and it would be interesting to see
if the results were qualitatively similar or not.  Our results
demonstrate that some care needs to be exercised in the application of
$\ell_1/\ell_\infty$ regularization.  Indeed, it can yield improved
statistical efficiency when the regression matrix exhibits structured
sparsity, with high overlaps among the sets of active coefficients
within each column.  However, our analysis shows that these
improvements are quite sensitive to the exact structure of the
regression matrix, and how well it aligns with the regularizing norm.
When this alignment is not high enough, then the use of
$\ell_1/\ell_\infty$ can actually impair performance relative to more
naive (and less computationally intensive) schemes based on
$\ell_1$-regularization, such as the Lasso.  Moreover, whether or not
the $\ell_1/\ell_\infty$ yields statistical improvements is very
sensitive to the actual magnitudes of the different regression
problems.  In comparison to related results obtained by Obozinski et
al.~\cite{OboWaiJor08} on block $\ell_1/\ell_2$ regularization, the
block $\ell_1/\ell_\infty$ exhibits some fragility, in that the
conditions under which it actually improves statistical efficiency are
delicate and easily violated.  An interesting open direction is study
whether or not it is possible to develop computationally efficient
methods that are \emph{fully adaptive} to the sparsity
overlap--namely, methods that behave like ordinary
$\ell_1$-regularization when there is no or little shared sparsity,
and behave like block regularization schemes in the presence of shared
sparsity.

%%%%%%%%%%%%%%%%%%%%%%%%%%%%%%%%%%%%%%%%%%%%%%%%%%%%%%%%%%%%%%%%%%%%%%%%%%%%

\appendix

\section{Recovering individual signed supports}
\label{SecIndividualSupports}

In this appendix, we discuss some issues associated with recovering
individual signed supports.  We begin by observing that once the
support union $\Joint$ has been recovered, one can restrict the
regression problem to this subset $\Joint$, and then apply Lasso to
each problem separately (with substantially lower cost, since each
problem is now low-dimensional) in order to recover the individual
signed supports.  If one is not willing to perform some extra
computation in this way, then the the interpretation of
Theorems~\ref{ThmDetDesign} and~\ref{ThmGauss}---in terms of
recovering the individual signed supports---requires a more delicate
treatment, which we discuss in this appendix.

Interestingly, the structure of the block $\ell_1/\ell_\infty$ norm
permits two ways in which to recover the individual signed supports.
\noindent \paragraph{$\ell_1/\ell_\infty$ primal recovery:} Solve the
block-regularized program~\eqref{EqnBlockReg}, thereby obtaining a
(primal) optimal solution $\Best \in \real^{\pdim \times \numreg}$.
Estimate the support union via $\estim{\Joint} \defn \bigcup
\limits_{\pind = 1, \ldots, \numreg} S(\best{\pind})$, and and
estimate the signed support vectors via
\begin{eqnarray}
\label{EqnSignPri}
[\SignPri(\best{\pind})]_\jind & \defn & \sign(\best{\pind}_\jind).
\end{eqnarray}

\noindent \paragraph{$\ell_1/\ell_\infty$ dual recovery:} Solve the
block-regularized program~\eqref{EqnBlockReg}, thereby obtaining an
primal solution $\Best \in \real^{\pdim \times \numreg}$.  For each
row $\jind = 1, \ldots, \pdim$, compute the set $\Mset_\jind \defn
\arg \max \limits_{\pind = 1, \ldots, \numreg} |\best{\pind}_\jind|$.
Estimate the support union via $\estim{\Joint} = \bigcup
\limits_{\pind = 1, \ldots, \numreg} S(\best{\pind})$, and estimate
the signed support vectors
\begin{eqnarray}
\label{EqnSignDuaTwo}
[\SignDua(\best{\pind}_\jind)] & = & \begin{cases}
\sign(\best{\pind}_\jind) & \mbox{if $\pind \in \Mset_\jind$} \\ 0 &
\mbox{otherwise.}
  \end{cases}
\end{eqnarray}
The procedure~\eqref{EqnSignDuaTwo} corresponds to estimating the
signed support on the basis of a dual optimal solution associated with
the optimal primal solution.

The dual signed support recovery method~\eqref{EqnSignDuaTwo} is more
conservative in estimating the individual support sets.  In
particular, for any given $\pind \in \{1, \ldots, \numreg\}$, it only
allows an index $\jind$ to enter the signed support estimate
$\SignDua(\best{\pind})$ when $|\best{\pind}_\jind|$ achieves the
maximum magnitude (possibly non-unique) across all indices $\pind = 1,
\ldots, \numreg$.  Consequently, unlike the primal
estimator~\eqref{EqnSignDuaTwo}, a corollary of
Theorem~\ref{ThmDetDesign} guarantees that the dual signed support
method~\eqref{EqnSignDuaTwo} never suffers from false inclusions in
the signed support set.  On the other hand, unlike the primal
estimator, it may incorrectly exclude indices of some supports---that
is, it may exhibit false exclusions. 

To provide a concrete illustration of this distinction, suppose that
$\pdim = 4$ and $\numreg = 3$, and that the true matrix $\Bstar$ and
estimate take the following form:
\begin{equation*}
\Bstar \, =\, \begin{bmatrix}  2 & 0 & -3  \\
                               2 & 4 &  0  \\
                               0 & 0 & 0   \\
                               0 & 0 & 0   
	      \end{bmatrix}, \quad \mbox{and} \quad
\Best \, =\, \begin{bmatrix}  1.9 & 0.1 & -2.9  \\
                              1.7 & 3.9 &  -0.1  \\
                               0 & 0 & 0   \\
                               0 & 0 & 0   
	      \end{bmatrix}.
\end{equation*}
Consistent with the claims of Theorem~\ref{ThmDetDesign}, the estimate
$\Best$ correctly recovers the support union---viz.  $S(\Best) =
\estim{\Joint} = \{1, 2\} = S(\Bstar)$.  The primal~\eqref{EqnSignPri}
and dual~\eqref{EqnSignDuaTwo} methods return the following estimates
of the individual signed supports:
\begin{equation*}
\SignPri(\Best) \; = \;  \begin{bmatrix}  1 & 1 & -1  \\
                               1 & 1 &  -1  \\
                               0 & 0 & 0   \\
                               0 & 0 & 0   
	      \end{bmatrix}, \quad \mbox{and} \quad
\SignDua(\Best) \; = \;  \begin{bmatrix}  0 & 0 & -1  \\
                               0 & 1 &  0  \\
                               0 & 0 & 0   \\
                               0 & 0 & 0  
	      \end{bmatrix}.
\end{equation*}
Consequently, the primal estimate includes false non-zeros in
positions $(1,2)$ and $(2,3)$, whereas the dual estimate includes
false zeros in positions $(1,1)$ and $(2,1)$.

We note that it is possible to ensure that under some conditions that
the dual support method~\eqref{EqnSignDuaTwo} will correctly recover
each of the individual signed supports, without any incorrect
exclusions.  However, as illustrated by Theorem~\ref{ThmPhase} and
Corollary~\ref{gapcor}, doing so requires additional assumptions on
the size of the gap $|\bstar{\pind}_\jind| - |\bstar{\pindtwo}_\jind|$
for indices $\jind \in \Both \defn \Sset(\bstar{\pind}) \cap
\Sset(\bstar{\pindtwo})$.

\section{Proof of Lemma~\ref{LemPopMutInco}}
\label{AppLemPopMutInco}
Note that conditioned $\mata_\Joint$, the rows of the random matrix
$\Xmatf{\pind}{\Jointcom}$ are i.i.d. Gaussian random vectors with
mean $\binprod{\Covmatif{\pind}{\Jointcom \Joint}
(\Covmatif{\pind}{\Joint \Joint})^{-1}} {\Xmatf{\pind}{\Joint}}$ and
covariance
\begin{eqnarray*}
\Covmatif{\pind}{\Jointcom | \Joint} & = & \Covmatif{\pind}{\Jointcom
\Jointcom} - \Covmatif{\pind}{\Jointcom \Joint}
(\Covmatif{\pind}{\Joint \Joint})^{-1} \Covmatif{\pind}{\Joint
\Jointcom}.
\end{eqnarray*}
\begin{eqnarray*}
\frac{1}{\numobs} \binprod{\Xmatf{\pind}{\Jointcom}}{
\Xmatf{\pind}{\Joint} ( \inprod {\frac{1}{\numobs}
\Xmatf{\pind}{\Joint}} {\Xmatf{\pind}{\Joint}})^{-1}} &
\stackrel{d}{=} & \Covmatif{\pind}{\Jointcom \Joint}
(\Covmatif{\pind}{\Joint \Joint})^{-1} \dwitf{\pind}{\Joint} +
\frac{1}{\numobs} \binprod{\Yif{\pind}{\Jointcom}} {\mata_\Joint
(\inprod{\frac{1}{\numobs}
\Xmatf{\pind}{\Joint}}{\Xmatf{\pind}{\Joint}})^{-1} }
\end{eqnarray*}
where $\Yif{\pind}{\Jointcom} \sim N(0, \Covmatif{\pind}{\Jointcom
\mid \Joint})$.

Using these expressions and triangle inequality, we obtain that
$\Mterm$ is upper bounded by
\begin{equation*}
\max_{\jind \in \Jointcom} \big \{ \sum_{\pind=1}^\numreg \| e_\jind^T
\Covmatif{\pind}{\Jointcom \Joint} (\Covmatif{\pind}{\Joint
\Joint})^{-1} \|_1 \big \} + \max_{\jind \in \Jointcom}
\sum_{\pind=1}^\numreg \big | \frac{1}{\numobs}
\binprod{\Yif{\pind}{\jind}} {\Xmatf{\pind}{\Joint}
(\inprod{\frac{1}{\numobs} \Xmatf{\pind}{\Joint}}
{\Xmatf{\pind}{\Joint}})^{-1} \dwitf{\pind}{\Joint}} \big |.
\end{equation*}
Applying the mutual incoherence assumption~\eqref{EqnMutIncoRand}, we
obtain
\begin{eqnarray*}
\Mterm & \leq & (1-\incopar) + \max_{\jind \in \Jointcom}
\sum_{\pind=1}^\numreg \big| \frac{1}{\numobs}
\binprod{\Yif{\pind}{\jind}} {\Xmatf{\pind}{\Joint}
(\inprod{\frac{1}{\numobs} \Xmatf{\pind}{\Joint}}
{\Xmatf{\pind}{\Joint}})^{-1} \dwitf{\pind}{\Joint}} \big |,
\end{eqnarray*}
as claimed.

\section{Proof of Lemma~\ref{LemVarBound}}
\label{AppLemVarBound}

Recall that $\dwitf{1}{\Joint} = (\dwitf{1}{\Sing},
\dwitf{1}{\Both})$, $\|\dwitf{1}{\Sing}\|_2^2 = (1-\overlap)
\spindex$, and that $\Both$ is the set where $|\best{1}_\Both| =
|\best{2}_\Both|$.  Thus, the claim is equivalent to showing that
$\|\duala_\Both\|_2^2$ is concentrated.  If $\overlap = 0$, then the claim is trivial, so that we
may assume that $\overlap > 0$.

Recall that
\begin{multline}
  \label{EqnRefDualB}
  \Signmat^1 \dwitf{1}{\Both} = \frac{1}{\relaxn} \big \{ \Smat^2
    \big[\Smat^1 + \Smat^2 \big]^{-1} \sahand^1 - \Smat^1 \big[\Smat^2
    + \Smat^1 \big]^{-1} \sahand^2 \big \} + \Smat^1 \big[ \Smat^1 +
    \Smat^2 \big]^{-1} \vec{1} - \\ \frac{1}{\relaxn} [(\Smat^1)^{-1}
    + (\Smat^2)^{-1}]^{-1}\bdiff.
\end{multline}

Using $\matsnorm{\cdot}{2}$ to denote the spectral norm, we first
claim that as long as $\spindex/\numobs \rightarrow 0$, then the
following events hold with probability greater than $1- c_1 \exp(-c_2
\numobs)$:
\begin{subequations}
  \label{AllMatControl}
  \begin{eqnarray}
    \label{EqnSmatControl}
    \matsnorm{\Smat^1 - I}{2} & = & \myorder, \\
    \label{EqnAdjControl}
    \matsnorm{[(\Smat^1)^{-1} + (\Smat^2)^{-1}]^{-1} - I/2}{2} & = & \myorder, \qquad \mbox{and} \\
    \label{EqnRatControl}
    \matsnorm{\Smat^1 \big[\Smat^1 + \Smat^2 \big]^{-1} - I/2}{2} & = &
    \myorder,
  \end{eqnarray}
\end{subequations}
as well as the analogous events with $\Smat^1$ and $\Smat^2$
interchanged.

To verify the bound~\eqref{EqnSmatControl}, we first diagonalize the
projection matrix.  All of its eigenvalues are $0$ or $1$, and it has
rank $(\numobs - \spindex)$ w.p. one, so that we may write
$\Proj{\Sing^\perp} = U^T D U$ for some orthogonal matrix $U$, and the
diagonal matrix $D = \diag \{1_{\numobs - \spindex}, 0_\spindex \}$,
\begin{eqnarray*}
\Smat & = & \numobs^{-1} \mata_\Both^T U^T D U \mata_\Both.
\end{eqnarray*}
But the projection $\Proj{\Sing^\perp}$ is independent of
$\mata_\Both$, which implies that the random rotation matrix $U$ is
independent of $\mata_\Both$, and hence $\mata_\Both \edist U
\mata_\Both$.  Since $D$ is diagonal with $(\numobs - \spindex)$ ones
and $\spindex$ zeros, $\Smat \edist \numobs^{-1} W^T W$, where $W \in
\real^{(\numobs - \spindex) \times |\Both|}$ is a standard Gaussian
random matrix.  Consequently, we have
\begin{eqnarray*}
\matsnorm{\Smat - I}{2} & \edist & \matsnorm{\numobs^{-1} W^T W -
I}{2} \\
& \leq & \big| \frac{\numobs-\spindex}{\numobs} -1
\big|\matsnorm{\frac{1}{\numobs-\spindex} W^T W}{2} +
\matsnorm{\frac{1}{\numobs-\spindex} W^T W - I}{2} \\
& = & \myorder,
\end{eqnarray*}
since $\matsnorm{W^T W/(\numobs-\spindex)}{2} = \order(1)$, and
\begin{equation*}
\matsnorm{\frac{1}{\numobs-\spindex} W^T W - I}{2} \; = \;
\order(\sqrt{\frac{\spindex}{\numobs- \spindex}}) \; = \; \myorder,
\end{equation*}
using concentration arguments for random matrices (see
Lemma~\ref{LemRandMat} in Appendix~\ref{AppLargeDev}).

For \eqref{EqnAdjControl} we may use the triangle inequality and the
submultiplicativity of the norm so that
\begin{eqnarray*}
  \matsnorm{[\Smat^{-1} + \Smattil^{-1}]^{-1} - I/2}{2} & = &
  \matsnorm{[\Smat^{-1} + \Smattil^{-1}]^{-1}(I - [\Smat^{-1} +
  \Smattil^{-1}]/2)}{2} \\
& \leq & \matsnorm{[\Smat^{-1} + \Smattil^{-1}]^{-1}}{2} \;
  \matsnorm{I - [\Smat^{-1} + \Smattil^{-1}]/2}{2} \\ 
& \leq & \frac{1}{2} \big \{ \matsnorm{I/2 - \Smat^{-1}/2}{2} +
  \matsnorm{I/2 - \Smattil^{-1}/2}{2} \big \} \matsnorm{[\Smat^{-1} +
  \Smattil^{-1}]^{-1}}{2} \\
& = & \matsnorm{[\Smat^{-1} + \Smattil^{-1}]^{-1}}{2} \; \myorder,
\end{eqnarray*}
Finally, since $\matsnorm{[\Smat^{-1} + \Smattil^{-1}]^{-1}}{2} =
\order(1)$, equation~\eqref{EqnAdjControl} is valid.

In order to establish the bound~\eqref{EqnRatControl}, we have
\begin{eqnarray*}
\matsnorm{\Smat [\Smat + \Smattil ]^{-1} - I/2}{2} & = &
\matsnorm{(\Smat/2 - \Smattil/2) [\Smat + \Smattil ]^{-1}}{2} \\
& \leq & \frac{1}{2}\big \{ \matsnorm{\Smat - I}{2} \; +
\matsnorm{\Smattil-I}{2} \big \} \; \matsnorm{[\Smat +
\Smattil]^{-1}}{2} \\
& = & \matsnorm{[\Smat + \Smattil]^{-1}}{2} \; \myorder.
\end{eqnarray*}
Since $\matsnorm{[\Smat + \Smattil] - 2I}{2} = \myorder \rightarrow
0$, we have $\matsnorm{[\Smat + \Smattil]^{-1}}{2} = \order(1)$, which
establishes the claim~\eqref{EqnRatControl}.

We are now ready to establish the claims of the lemma.  From the
representation~\eqref{EqnRefDualB}, we apply triangle inequality and
our bounds on spectral norms, thereby obtaining
\begin{eqnarray*}
  \sqrt{\norm{\duala_\Both}_2^2 + \norm{\dualb_\Both}_2^2} & \leq &
  \sqrt{\norm{\frac{\vecone}{2} - \frac{1}{2 \relaxn}
      (\abs{\bvecstar{2}_{\Both}} - \abs{\bvecstar{1}_{\Both}})}_2^2 +
    \norm{\frac{\vecone}{2} + \frac{1}{2 \relaxn} (\abs{\bvecstar{2}_\Both} -
      \abs{\bvecstar{1}_\Both})}_2^2} + 2 \norm{\are} \\
  & \leq &
  \sqrt{\spindex} \big \{ \sqrt{\frac{\overlap}{2} + \frac{1}{2
      \spindex \relaxn^2} \norm{\abs{\bvecstar{1}_\Both} -
      \abs{\bvecstar{2}_{\Both}}}_2^2} + \frac{2}{\sqrt{\spindex}} \norm{r}_2
  \big \}
\end{eqnarray*}
with probability greater than $1-c_1 \exp(-c_2 \numobs)$, where $\are
= \dwitf{1}{\Both} - \frac{1}{2} (\vecone - \frac{1}{\relaxn}
(\abs{\bvecstar{2}_\Both} - \abs{\bvecstar{1}_\Both}))$.  By the
decomposition of $\dwitf{1}{\Both}$ in equation~\eqref{EqnRefDualB}
and applying bounds~\eqref{AllMatControl}
\begin{equation*}
  \norm{\are}_2 \leq \sqrt{s} \big \{ \myorder + \frac{1}{\relaxn
    \sqrt{s}} \; \myorder \norm{\abs{\bvecstar{1}_\Both} -
    \abs{\bvecstar{2}_\Both}}_2 + \frac{1}{2 \sqrt{\spindex} \relaxn} (1 +
  \myorder) \; \big[ \| \sahand \|_2 + \| \sahandtil \|_2 \big] \big
  \}
\end{equation*}
Since $\spindex/\numobs = o(1)$, in order to establish the upper
bound~\eqref{EqnVarUpper} it suffices to show that $\|\sahand\|_2 +
\|\sahandtil\|_2 = o(\sqrt{\spindex} \relaxn)$ w.h.p.  Similarly, in
the other direction, we have
\begin{eqnarray*}
  \sqrt{\norm{\dwitf{1}{\Both}}_2^2 + \norm{\dwitf{2}{\Both}}_2^2} &
  \geq & \sqrt{\spindex} \big \{ \sqrt{\frac{\overlap}{2} + \frac{1}{2
  \spindex \relaxn^2} \norm{\abs{\bvecstar{1}_\Both} -
  \abs{\bvecstar{2}_\Both}}_2^2} - \frac{2}{\sqrt{\spindex}}
  \norm{r}_2 \big \}
\end{eqnarray*}
Following the same line of reasoning, in order to prove the lower
 bound~\eqref{EqnVarLower}, it suffices to show that $\|\sahand\|_2 +
 \|\sahandtil\|_2 = o(\sqrt{\spindex} \relaxn)$ w.h.p.

Since $\|\sahand\|_2$ and $\|\sahandtil\|_2$ behave similarly, it
suffices to show that $\|\sahand\|_2 = o(\relaxn \sqrt{\spindex})$.
From the definition~\eqref{EqnDefnSahand}, we see that conditioned on
$(\mata_\Sing, \noisea, \duala_\Sing)$, the random vector $\sahand$ is
zero-mean Gaussian, with i.i.d. elements with variance
\begin{eqnarray*}
\sigma^2 & \defn & \frac{\relaxn^2}{\numobs} (\duala_\Sing)^T
(\mata_\Sing^T \mata_\Sing/\numobs)^{-1} \duala_\Sing +
\frac{1}{\numobs} \noisea^T \proj_{\Sing^\perp} \noisea.
\end{eqnarray*}
Recalling that $\|\duala_\Sing\|_2^2 = (1-\overlap) \spindex$, we have
\begin{eqnarray*}
\sigma^2 & \leq & \frac{\relaxn^2 (1-\overlap) \spindex }{\numobs} \;
\lammax((\mata_\Sing^T \mata_\Sing/\numobs)^{-1}) + \frac{1}{\numobs}
\frac{\| \proj_{\Sing^\perp}(\noisea)\|_2^2}{\numobs}
\end{eqnarray*}
By random matrix concentration (see the discussion following
Lemma~\ref{LemRandMat} in Appendix~\ref{AppLargeDev}), we have
$\lammax((\mata_\Sing^T \mata_\Sing/\numobs)^{-1}) \leq 1 + \myorder$
w.h.p., and by $\chi^2$ tail bounds (see Lemma~\ref{LemChiTail} in
Appendix~\ref{AppLargeDev}), we have $\frac{\|
\proj_{\Sing^\perp}(\noisea)\|_2^2}{\numobs} = \order(1)$ w.h.p.
Consequently, with high probability, we have $\sigma^2 = \order(
\frac{\relaxn^2 \spindex}{\numobs} + \frac{1}{\numobs})$.  Since the
Gaussian random vector $\sahand$ has length $|\Both| =
\Theta(\spindex)$, again by concentration for $\chi^2$ random
variables, we have (with probability greater than $1-c_1 \exp(-c_2
\spindex)$), $\|\sahand\|_2^2 = \order(\sigma^2 \spindex)$.  Combining
the pieces, we conclude that w.h.p.
\begin{eqnarray*}
\|\sahand\|_2^2 & = & \order \big(\relaxn^2 \spindex
\frac{\spindex}{\numobs} + \frac{\spindex}{\numobs} \big) \\
& = & \order\big( \relaxn^2 \spindex \; \big[ \frac{\spindex}{\numobs}
+ \frac{1}{\relaxn^2 \numobs} \big] \big) \; = \; o(\relaxn^2
\spindex),
\end{eqnarray*}
where the final equality follows since $\spindex/\numobs = o(1)$ and
$1/(\relaxn^2 \numobs) = o(1)$.

%%%%%%%%%%%%%%%%%%%%%%%%%%%%%%%%%%%%%%%%%%%%%%%%%%%%%%%%%%%%%%%%%%%%%%%%%

\section{Convex-analytic characterization of optimal solutions}

This section is devoted to the development of various properties of
the optimal solution(s) of the block
$\ell_1/\ell_\infty$-regularized problem~\eqref{EqnBlockReg}.

\subsection{Basic optimality conditions}

By standard conditions for optimality in convex
programs~\cite{Rockafellar}, the zero-vector must belong to the
subdifferential of the objective function in the convex
program~\eqref{EqnBlockReg}, or equivalently, we must have
for each $\pdim =1,2,\ldots,\numreg$
\begin{eqnarray}
\label{maindual1}
\frac{1}{\numobs} \binprod{\Xmat{\pind}}{\Xmat{\pind}} \best{\pind} -
\frac{1}{\numobs} \Xmatt{\pind} \yobs{\pind} + \relaxn \dwit{\pind} &
= & 0,
\end{eqnarray}
where $\Dwit \in \real^{\pdim \times \numreg}$ must be an element of
the subdifferential $\partial \| \Best\|_{\infty,1}$.  Substituting
the relation $\yobs{\pind} = \Xmat{\pind} \bvecstar{\pind} +
\wnoise{\pind}$, we obtain
\begin{equation}
  \label{dualdiff}
  \frac{1}{\numobs} \binprod{\Xmat{\pind}}{\Xmat{\pind}} (\best{\pind}
  - \bvecstar{\pind}) - \frac{1}{\numobs} \Xmatt{\pind} \wnoise{\pind}
  + \relaxn \dwit{\pind} = 0.
\end{equation}

\subsection{Proof of Lemma~\ref{LemKey}}
\label{AppLemKey}

We begin with the proof of part (i): suppose that steps (A) through
(C) of the primal-witness construction succeed.  By definition, it
outputs a primal pair, of the form $(\Bwit_\Joint \; , \; 0)$, along
with a candidate dual optimal solution $ \{ (\Dwit_\Joint \; , \;
\Dwit_\Jointcom) \}$.  Note that the conditions defining the
$\ell_1/\ell_\infty$ subdifferential apply in an elementwise manner,
to each index $i = 1, \ldots, \pdim$.  Since the sub-vector
$\Dwit_\Joint$ was chosen from the subdifferential of the restricted
optimal solution, it is dual feasible.  Moreover, since the strict
dual feasibility condition~\eqref{EqnStrictDual} holds, the matrix
$\Dwit_\Jointcom$ constructed in step (C) is dual feasible for the
zero-solution in the sub-block $\Jointcom$.  Therefore, we conclude
that $(\Bwit_\Joint \; , \; 0)$ is a primal optimal solution for the
full block-regularized program~\eqref{EqnBlockReg}.

It remains to establish uniqueness of this solution.  Define the ball
\begin{eqnarray*}
\Kball & = & \{ \Dwit \in \real^{\pdim \times \numreg} \mid
\sum_{\pind = 1}^\numreg |\dwit{\pind}_\jind| \leq 1 \quad \forall
\jind=1, \ldots, \pdim \},
\end{eqnarray*}
and observe that we have the variational representation
\begin{eqnarray*}
  \|B\|_{1, \infty} & = & \sup_{ \Dwit \in \Kball} \inprod{\Dwit}{B}
\end{eqnarray*}
where $\inprod{\cdot}{\cdot}$ denotes the Euclidean inner product.
With this notation, the block-regularized program~\eqref{EqnBlockReg}
is equivalent to the saddle-point problem
\begin{equation*}
  \inf_{B \in \real^{\pdim \times \numreg}} \sup_{\Dwit \in \Kball}
\biggr \{ \frac{1}{2 \numobs} \sum_{\pind = 1}^\numreg \|\yobs{\pind}
- \Xmat{\pind} \beta^\pind\|_2^2 + \relaxn \inprod{\Dwit}{B} \biggr
\}.
\end{equation*}
Since this saddle-point problem is strictly feasible and
convex-concave, it has a value.  Moreover, given any dual optimal
solution---in particular, $\Dwit$ from the primal-dual
construction---any optimal primal solution $\Best$ must satisfy the
saddle point condition
\begin{eqnarray*}
  \|\Best\|_{1, \infty} & = & \sup_{ \Dwit \in
    \Kball} \inprod{\Dwit}{\Best}
\end{eqnarray*}
But this condition can only hold if $\forall \pind \in \{1 , 2 ,
\ldots, \numreg \}$, $\beta^\pind_\jind = 0$ for any index $\jind \in
\{1, \ldots, \pdim \}$ such that $\sum_{\pind = 1}^\numreg
|\dwit{\pind}_\jind| < 1$.  Therefore, any optimal primal solution
must satisfy $\Best_\Jointcom = 0$, so that solving the original
program~\eqref{EqnBlockReg} is equivalent to solving the restricted
program~\eqref{EqnRestricted}. Lastly, if the matrices
$\binprod{\Xmat{\pind}_\Joint}{\Xmat{\pind}_\Joint}$ are invertible
for each $\pind \in \{1 , 2, \ldots, \numreg \}$, then the restricted
problem~\eqref{EqnRestricted} is strictly convex, and so has a unique
solution, thereby completing the proof of Lemma~\ref{LemKey}(i).

%%%%%%%%%%%%%%%%%%%%%%%%%%%%%%%%%%%%%%%%%%%%%%%%%%%%%%%%%%%%%%%%%%%%

We now prove part (ii) of Lemma~\ref{LemKey}.  Suppose that we are
given an estimate $\Best$ of the true parameters $\Bstar$ by solving
the convex program~\eqref{EqnBlockReg} such that $\Best_\Jointcom =
0$.

Since $\Best$ is an optimal solution to the convex
program~\eqref{EqnBlockReg}, the the optimality conditions of
equation~\eqref{dualdiff}, must be satified. We may rewrite those
conditions as
\begin{eqnarray*}
\frac{1}{\numobs} \binprod{\Xmatf{\pind}{\Joint}}{\Xmat{\pind}}
(\Nuvar{\pind}) - \frac{1}{\numobs} \Xmatft{\pind}{\Joint}
\wnoise{\pind} + \relaxn \dwitf{\pind}{\Joint} = 0 \\
\frac{1}{\numobs} \binprod{\Xmatf{\pind}{\Jointcom}}{\Xmat{\pind}}
(\Nuvar{\pind}) - \frac{1}{\numobs} \Xmatft{\pind}{\Jointcom}
\wnoise{\pind} + \relaxn \dwitf{\pind}{\Jointcom} = 0,
\end{eqnarray*}
where $\Nuvar{\pind} = \best{\pind} - \bvecstar{\pind}$. Recalling
that $\Best_\Jointcom = \Bstar_\Jointcom = 0$, we obtain
\begin{subequations}
\begin{eqnarray}
\label{Lem2:1}
\frac{1}{\numobs}
\binprod{\Xmatf{\pind}{\Joint}}{\Xmatf{\pind}{\Joint}}
(\Nuvar{\pind}_\Joint) - \frac{1}{\numobs} \Xmatft{\pind}{\Joint}
\wnoise{\pind} + \relaxn \dwitf{\pind}{\Joint} = 0, \qquad \mbox{and}
\\
\label{Lem2:2}
\frac{1}{\numobs}
\binprod{\Xmatf{\pind}{\Jointcom}}{\Xmatf{\pind}{\Joint}}
(\Nuvar{\pind}_\Joint) - \frac{1}{\numobs}
\Xmatft{\pind}{\Jointcom} \wnoise{\pind} + \relaxn
\dwitf{\pind}{\Jointcom} = 0.
\end{eqnarray}
\end{subequations}
Again, by standard conditions for optimality in convex
programs~\cite{Bertsekas_nonlin,Hiriart1}, the first of these two
equations is exactly the condition that must be satisfied by an
optimal solution of the restricted
program~\eqref{EqnRestricted}. However, we have already shown that the
candidate solution $\Best_\Joint$ satisfies this condition, so that it
must also be an optimal solution of the convex
program~\eqref{EqnRestricted}. Additionally, the value of
$\Dwitf{\Joint}$ that satisfies equation~\eqref{Lem2:1} for each $i
\in \{1,2,\ldots,\numreg\}$ is an element of $\partial \|
\Best\|_{\infty,1}$.  We have thus shown that steps (B) and (C) of the
primal-witness construction succeed.  It remains to establish
uniqueness in part (A).  However, we note that
$\binprod{\Xmatf{\pind}{\Joint}}{\Xmatf{\pind}{\Joint}}$ is invertible
for each $\pind$.  Hence, for any solution $\Best$ such that
$\Best_\Jointcom = 0$,
\begin{equation*}
   \Nuvar{\pind}_\Joint = (\frac{1}{\numobs}
   \binprod{\Xmatf{\pind}{\Joint}}{\Xmatf{\pind}{\Joint}})^{-1} \left
   [ \frac{1}{\numobs} \Xmatft{\pind}{\Joint} \wnoise{\pind} - \relaxn
   \dwitf{\pind}{\Joint} \right ]
\end{equation*}
is well-defined and unique, noting that $\Nuvar{\pind}_\Jointcom =
0$. Thus, we have established the equality~\eqref{EqnDefnUvar} and
that $\Best_\Joint$ is unique. Therefore, $\Best$ gives solutions to
steps (A) and (B) when solving the restricted convex program over the
set $\Joint$.

%%% Off Support
%\label{AppDualwitOff}
Finally, we derive the form of the dual solution
$\dwitf{\pind}{\Jointcom}$, as a function of
$\binprod{\Xmatf{\pind}{\Joint}}{\Xmatf{\pind}{\Joint}}$,
$\Dwit_\Joint$, and $\Best-\Bstar$.  Recall that
$\binprod{\Xmatf{\pind}{\Joint}}{\Xmatf{\pind}{\Joint}}$ is
invertible, $\Dwit_\Joint$ is an element of the subdifferential of
$\partial \bnorminf{\Bwit_\Joint}$, and $\Best_\Jointcom =
\Bstar_\Jointcom = 0$.  From equation~\eqref{EqnDefnUvar}, we have
\begin{eqnarray}
\label{EqnAppDualwitOff}
\dwitf{\pind}{\Jointcom} & = & \frac{1}{\relaxn \numobs}
\binprod{\Xmatf{\pind}{\Jointcom} } {
(I-\myproj{\Xmatf{\pind}{\Joint}}) \wnoise{\pind} } +
\frac{1}{\numobs} \binprod{\Xmatf{\pind}{\Jointcom}}
{\Xmatf{\pind}{\Joint} (\frac{1}{\numobs}
\inprod{\Xmatf{\pind}{\Joint}}{\Xmatf{\pind}{\Joint}})^{-1}
\dwitf{\pind}{\Joint}} \quad \mbox{for $\pind = 1, \ldots,
\numreg$. $\qquad$}
\end{eqnarray}
The claimed form of the dual solution follows by substituting
equation~\eqref{EqnDefnUvar} into equation~\eqref{Lem2:2}.
%

%%%%% On Support Subgradients
\subsection{Subgradients on the support}
\label{subgradsupportdef}
In this section, we focus on the specific form of the dual variables
$\dwit{\pind}_\Joint$.  Our approach is to construct a candidate set
of dual variables, and then show that they are valid.  We begin by
defining the sets $\Both = \Sset(\bstar{\pind}) \cap
\Sset(\bstar{\pindtwo})$, corresponding to the intersection of the
supports, and the set $\Sing = \Joint \setminus \Both$ corresponding
to elements in one (but not both) of the supports.  For $\pind = 1,2$,
we let $\Signmat^\pind \in \real^{\overlap \spindex \times \overlap
\spindex}$ is a diagonal matrix whose diagonal entries correspond to
$\sign(\bvecstar{\pind}_{\Both})$.  In addition, we define the vectors
$\sahand^\pind \in \real^{\overlap \spindex}$ and matrices
$\Smat^\pind \in \real^{\overlap \spindex \times \overlap \spindex}$
via
\begin{subequations}
\begin{eqnarray}
  \label{EqnDefnSahand}
    \sahand^\pind & \define & \Signmat^\pind \left [ \frac{1}{\numobs}
	\binprod{\Xmatf{\pind}{\Both}}{\Xmatf{\pind}{\Sing}
	(\frac{1}{\numobs}
	\binprod{\Xmatf{\pind}{\Sing}}{\Xmatf{\pind}{\Sing}})^{-1}}
	\relaxn \dwitf{\pind}{\Sing} - \frac{1}{\numobs}
	\binprod{\Xmatf{\pind}{\Both}}{I -
	\myproj{\Xmatf{\pind}{\Sing}}} \wnoise{\pind} \right ]\\
    \Smat^\pind & \define & \frac{1}{\numobs} \Signmat^\pind
	 \binprod{\Xmatf{\pind}{\Both}} {(I -
	 \proj_{(\Xmatf{\pind}{\Sing})})\Xmatf{\pind}{\Both}}
	 \Signmat^\pind.
  \end{eqnarray}
\end{subequations}
Given these definitions, we have the following lemma:
\begin{lemma}
\label{lemmaboth}
Assume that $\numreg = 2$, and that $|\best{1}_\Both| =
|\best{2}_\Both|$.  If $\Best_\Jointcom = \Bstar_\Jointcom = 0$, then
the dual variable $\dwit{1}$ satisfies the relation
\begin{multline}
\label{lemma:ref}
\Signmat^1 \dwitf{1}{\Both} = \frac{1}{\relaxn} \big \{ \Smat^2
\big[\Smat^1 + \Smat^2 \big]^{-1} \sahand^1 - \Smat^1 \big[\Smat^2 +
\Smat^1 \big]^{-1} \sahand^2 \big \} + \Smat^1 \big[ \Smat^1 + \Smat^2
\big]^{-1} \vec{1} - \\
    \frac{1}{\relaxn} [(\Smat^1)^{-1} +
      (\Smat^2)^{-1}]^{-1} \bdiff
  \end{multline}
  and $\dwitf{2}{\Sing} = \SignSup(\bvecstar{2}_\Sing)$, with
  analogous results holding for $\dwit{2}$.
\end{lemma}

Given these forms for $\Signmat^1 \dwitf{1}{\Both}$ and $\Signmat^2
\dwitf{2}{\Both}$, it remains to show that the relation $\Signmat^1
\dwitf{1}{\Both} + \Signmat^2 \dwitf{2}{\Both} = 1$ holds under the
conditions of Theorem~\ref{ThmPhase}(a).  Intuitively, this condition
should hold since under the conditions of theorem~\ref{ThmPhase}(a),
the matrix $\Smat^\pind$ is approximately the identity, and the vector
$\sahand^\pind$ is approaching $0$. Finally, we expect that $\bdiff
\defn |\bstar{2}_\Both| - |\bstar{1}_\Both|$ is very small, hence the
final term is also very small. Therefore, on the set $\Both$, both
$\Signmat^1 \dwitf{1}{\Both}$ and $\Signmat^2 \dwitf{2}{\Both}$ are
approximately equal to $\frac{1}{2}$. We formalize this rough
intuition in the following lemma:
\begin{lemma}
  \label{LemBound}
  Under the assumptions of Theorem~\ref{ThmPhase}(a) each of the
  following conditions hold for sufficiently large $\numobs$,
  $\spindex$, and $\pdim$ with probability greater than $1 - c_1
  \exp(-c_2 \numobs)$:
  \begin{subequations}
    \begin{eqnarray}
      \label{lembound1}
      \norm{\frac{1}{\relaxn} [(\Smat^1)^{-1} + (\Smat^2)^{-1}]^{-1}(\bdiff)}_{\infty} & \leq & \epsilon \\
      \label{lembound2}
      \| \frac{1}{\relaxn} \big \{ \Smat^2 \big[\Smat^1 +
	\Smat^2 \big]^{-1} \sahand^1 - \Smat^1 \big[\Smat^2 + \Smat^1
	\big]^{-1} \sahand^2 \big \} \|_{\infty} & \leq & \epsilon \\
      \label{lembound3}
      \| \Smat^1 \big[ \Smat^1 + \Smat^2 \big]^{-1} \vec{1} - \frac{1}{2} \|_{\infty} & \leq & \frac{1}{2} - 3 \epsilon.
    \end{eqnarray}
  \end{subequations}
\end{lemma}

Given Lemmas~\ref{lemmaboth} and~\ref{LemBound}, we can conclude that
the definition for the dual variables on the support is valid. The
remaining subsections in this appendix are dedicated to verifying the
above results: in particular, we prove Lemma~\ref{lemmaboth} in
Appendix~\ref{AppLemBoth} and Lemma~\ref{LemBound} in
Appendix~\ref{AppLemBound}.

%. Skipping ahead to Appendix~\ref{AppLemPopMutInco}, we
%prove the results stated in the proof of Theorem~\ref{ThmGauss}. An
%important step in the proof was establishing a bound on the $\ell_2$
%norm of the dual vectors, which we prove in
%Section~\ref{AppLemVarBound}.

%%%%%%%%%%%%%%%%%%%%%%%%%%%%%%%%%%%%%%%%%%%%%%%%%%%%%%%%%%%%%%%%%%%

\subsection{Proof of Lemma~\ref{lemmaboth}}
\label{AppLemBoth}

We now proceed to establish the validity of the closed form
expressions for $\dwit{1}_\Joint$ and $\dwit{2}_\Joint$. From
equation~\eqref{Lem2:1} we have that
\begin{equation*}
  \Nuvar{1}_\Sing = -(\frac{1}{\numobs}
  \binprod{\Xmatf{\pind}{\Sing}}{\Xmatf{\pind}{\Sing}})^{-1} \left [
  \frac{1}{\numobs} \binprod{\Xmatf{1}{\Sing}}{\Xmatf{1}{\Both}}
  \Nuvar{1}_\Both + \relaxn \dwitf{1}{\Sing} \right ] +
  (\frac{1}{\numobs}
  \binprod{\Xmatf{\pind}{\Sing}}{\Xmatf{\pind}{\Sing}})^{-1}
  \Xmatft{1}{\Sing} \wnoise{1}
\end{equation*}
substituting back into \eqref{Lem2:1}
\begin{equation*}
\frac{1}{\numobs} \binprod{\Xmatf{1}{\Both}}{\Xmatf{1}{\Both}}
\Nuvar{1}_\Both + \frac{1}{\numobs}
\binprod{\Xmatf{1}{\Both}}{\Xmatf{1}{\Sing}} \Nuvar{1}_\Sing -
\frac{1}{\numobs} \Xmatft{1}{\Both} \wnoise{1} + \relaxn
\dwitf{1}{\Both} \, = \, 0,
\end{equation*}
so that we obtain
\begin{subequations}
  \begin{eqnarray}
    \label{inter:1}
    \Smat^1 \Nuvar{1}_\Both & = & \sahand^1 - \relaxn
    \dwitf{1}{\Both}\qquad \mbox{and similarly,} \\
    \label{inter:2}
    \Smat^2 \Nuvar{2}_\Both & = & \sahand^2 - \relaxn \dwitf{2}{\Both}
  \end{eqnarray}
\end{subequations}
Recall that by assumption that $\Signmat^1 \best{1}_{\Both} =
\abs{\best{1}_{\Both}} = \abs{\best{2}_{\Both}} = \Signmat^2
\best{2}_{\Both}$, and $\Signmat \duala_\Both + \Signmattil
\dualb_\Both = 1$.

Subtracting $\Smat^1 \Signmat^1 \adjusta$ and $\Smat^2 \Signmat^2
\adjustb$ from equations~\eqref{inter:1} and \eqref{inter:2}
\begin{subequations}
  \begin{gather}
    \label{inter:3}
    \Smat^1 \Signmat^1 (\Nuvar{1}_{\Both} - \adjusta) = \sahand^1 - \relaxn
    \Signmat^1 \dwitf{1}{\Both} - \Smat^1 \Signmat^1 \adjusta \\
    \label{inter:4} \Smat^2 \Signmat^2 (\Nuvar{2}_\Both - \adjustb) = \sahand^2 - \relaxn \Signmat^2 \dwitf{1}{\Both} - \Smat^2 \Signmat^2 \adjustb
  \end{gather}
\end{subequations}
Applying the fact that $\Signmat^1 (\Nuvar{1}_\Both - \adjusta) =
\Signmat^2 (\Nuvar{2}_\Both - \adjustb)$.
\begin{equation*}
  (\Smat^1 + \Smat^2) \Signmat^1 (\Nuvar{1}_\Both-\adjusta) = (\sahand^1
  + \sahand^2) - \relaxn \vec{1} - \Smat^1 \Signmat^1 \adjusta -
  \Smat^2 \Signmat^2 \adjustb,
\end{equation*}
where $\vec{1} \in \real^{\overlap s}$. Then solving for $\Signmat^1
(\Nuvar{1}_\Both - \adjusta)$ letting $\Signmat^1 \adjusta -
\Signmat^2 \adjustb = \bdiff$ and substituting back into
equation~\eqref{inter:3}
\begin{multline}
  \relaxn \Signmat^1 \dwitf{1}{\Both} = \Smat^1 \big [ \Smat^1 + \Smat^2
    \big ] ^{-1} \relaxn \vecone - [(\Smat^1)^{-1} +
    (\Smat^2)^{-1}]^{-1}(\bdiff) \\
  + \Smat^2 \big [ \Smat^1 + \Smat^2 \big ]^{-1} \sahand^1 - \Smat^1
  \big [ \Smat^1 + \Smat^2 \big ]^{-1} \sahand^2.
\end{multline}

%%%%%%%%%%%%%%%%%%%%%%%%%%%%%%%%%%%%%%%%%%%%%%%%%%%%%%%%%%%%%%%%%%%%%%%%%%%%%

\subsection{Proof of Lemma~\ref{LemBound}}
\label{AppLemBound}

The first term $\frac{1}{\relaxn} [(\Smat^1)^{-1} +
(\Smat^2)^{-1}]^{-1} \, \bdiff$ can be decomposed as
\begin{eqnarray*}
  \frac{1}{\relaxn} [(\Smat^1)^{-1} + (\Smat^2)^{-1}]^{-1}\bdiff & = &
  \underbrace{\frac{1}{\relaxn} ([(\Smat^1)^{-1} +
  (\Smat^2)^{-1}]^{-1} - I/2) \bdiff}_{\mbox{\large $\mytermone_1$}} +
  \underbrace{\frac{\bdiff}{2 \relaxn}}_{\mbox{\large $\mytermone_2$}}
\end{eqnarray*}
Under the assumptions of Theorem~\ref{ThmPhase}(a), we have $\big |
\frac{\bdiff}{2 \, \relaxn} \big | \to 0$, hence, for $\spindex$ large
enough, $\mytermone_2 \leq \epsilon/4$.

In order to bound $\mytermone_1$, we note that with probability
greater than $1 - c_1 \exp(-c_2 \numobs)$, the spectral norm of
$([(\Smat^1)^{-1} + (\Smat^2)^{-1}]^{-1} - I/2)$ is $\myorder$ (see
the bound~\eqref{EqnAdjControl} from Appendix~\ref{AppLemVarBound}).
Consequently, we may decompose $([(\Smat^1)^{-1} +
(\Smat^2)^{-1}]^{-1} - I/2)$ as $Q D Q^T$ where $Q$ and $D$ are
independent and $Q$ is distributed uniformly over all orthogonal
matrices, and $\matsnorm{D}{2} = \order(\sqrt{\spindex/\numobs})$.
Using this decomposition, the following lemma, proved in
Appendix~\ref{AppBigLemmaOne}, allows us to obtain the necessary
control on the quantity $\norm{\mytermone_1}_\infty$:
\begin{lemma} 
\label{biglemma1}  
Let $\Quni \in \real^{\spindex \times \spindex}$ be a matrix chosen
uniformly at random from the space of all orthogonal matrices.
Consider a second random matrix $\Amat$, independent of $\Quni$.
If $\spindex/\numobs = o(1)$, then
for any fixed vector $x \in \real^\spindex$ and fixed $\epsilon > 0$,
we have:
\begin{enumerate}
\item[(a)] \label{lem9:parta} If  $\matsnorm{\Amat}{2} \leq
  \sqrt{\frac{\spindex}{\numobs}}$, then
  \begin{eqnarray*}
    \mprob[\norm{\Quni^T \Amat \Quni x}_{\infty} \geq \frac{\epsilon}{2}]
    & \leq & c_1 \exp \big ( -c_2 \epsilon^2 \frac{\numobs}{\spindex
      \|x\|_\infty^2} + \log(\spindex) \big ).
  \end{eqnarray*}
\item[(b)] \label{lem9:partb} If $\matsnorm{\Amat}{2} \leq \frac{\spindex}{\numobs}$, then
  \begin{eqnarray*}
    \mprob[\norm{\Quni^T \Amat \Quni x}_{\infty} \geq \frac{\epsilon}{2}]
    & \leq & c_1 \exp \big ( -c_2 \epsilon^2 \frac{\numobs^2}{\spindex^2
      \|x\|_\infty^2} + \log(\spindex) \big ).
  \end{eqnarray*}
\end{enumerate}
\end{lemma} 
With reference to the problem of bounding $\| \mytermone_1\|_\infty$, we may
apply part (a) of this lemma with $\Amat = D$ and $x = \frac{\bdiff}{2 \relaxn}$
to conclude that $\|\mytermone_1\|_\infty \leq \epsilon/2$ with high
probability, thereby establishing the bound~\eqref{lembound1}.

\vspace*{.2in}

We now turn the proving the bound~\eqref{lembound2}.  We begin by
decomposing the terms involved in this equation as
\begin{eqnarray*}
\frac{1}{\relaxn} \Smat^2 \big[\Smat^1 + \Smat^2 \big]^{-1} \sahand^1
& = & \frac{1}{\relaxn} \biggr [\Smat^2 \big [ \Smat^1 + \Smat^2
\big]^{-1} - \frac{I}{2} \biggr ] \sahand^1 + \frac{\sahand^1}{2
\relaxn} \\
\frac{1}{\relaxn} \Smat^1 \big[\Smat^1 + \Smat^2 \big]^{-1} \sahand^2
& = & \frac{1}{\relaxn} \biggr [\Smat^1 \big [ \Smat^1 + \Smat^2
\big]^{-1} - \frac{I}{2} \biggr ] \sahand^2 + \frac{\sahand^2}{2
\relaxn}
\end{eqnarray*}
Recalling the form of $\sahand^\pind \in \real^{\overlap \spindex}$,
conditioned on $\Xmatf{\pind}{\Sing}$ and $\wnoise{\pind}$, we have
\begin{equation*}
  \sahand^\pind/(2 \relaxn) \sim N \left (0,\frac{1}{4} \langle
  \dwitf{\pind}{\Sing} , \frac{1}{\numobs} (\frac{1}{\numobs}
  \binprod{\Xmatf{\pind}{\Sing}}{\Xmatf{\pind}{\Sing}})^{-1}
  \dwitf{\pind}{\Sing} \rangle I_{\overlap \spindex} +
  \norm{\wnoise{\pind}}_2^2/(n^2 \relaxn^2) I_{\overlap \spindex}
  \right ).
\end{equation*}
\noindent However, by Lemmas~\ref{LemChiTail} and~\ref{LemRandMat}
(see Appendix~\ref{AppLargeDev}), as well as the fact that
$\norm{\dwitf{\pind}{\spindex}}_2^2 = (1-\overlap)\spindex$, for
$\numobs$ and $\spindex$ large enough, the variance term is bounded by
\begin{equation}
  \label{sahand:varbound}
  \frac{1}{4} \langle \dwitf{\pind}{\Sing} , \frac{1}{\numobs}
  (\frac{1}{\numobs}
  \binprod{\Xmatf{\pind}{\Sing}}{\Xmatf{\pind}{\Sing}})^{-1}
  \dwitf{\pind}{\Sing} \rangle + \norm{\wnoise{\pind}}_2^2/(n^2
  \relaxn^2) \leq \frac{1}{4} (1-\overlap) \frac{\spindex}{\numobs} (1
  + \delta) + \frac{1}{2} \frac{1}{\numobs \relaxn^2}
\end{equation}
\noindent with probability greater than $1 - c_1 \exp(-c_2
\numobs)$. Hence, by standard Gaussian tail bounds, the inequalities
$\norm{\sahand^1/(2\relaxn)}_{\infty} < \epsilon/4$ and
$\norm{\sahand^2/(2 \relaxn)}_{\infty} < \epsilon/4$ both hold with
probability greater than $1 - c_1 \exp(-\delta'
\log(\pdim-2\spindex))$.

Now to bound the first term in the decomposition we begin by
diagonalizing $\Smat^2 = \Quni^T \Smatd \Quni$. Note that $Q$ is
independent of $\Xmat{1}$ and $\Smatd$ and by symmetry
$\Xmatf{1}{\Both} \stackrel{d}{=} \Quni \Xmatf{1}{\Both}$. 
Following some algebra, we find that
\begin{equation*}
\frac{1}{\relaxn} \biggr [\Smat^2 \big [ \Smat^1 + \Smat^2 \big]^{-1}
  - \frac{I}{2} \biggr ] \sahand^1 = \frac{1}{\relaxn} \Quni^T \biggr
  [D \big [ \Quni \Smat^1 \Quni^T + D \big]^{-1} - \frac{I}{2} \biggr
  ] \Quni \sahand^1
\end{equation*}
The random vector $\sahand^1$ is independent of $\Quni$ and $\Quni
f^1$ is independent of $\Quni$ by symmetry. Hence, the vector $v \defn
\frac{1}{2} [ 2 \Smatd (\Smatd + \Quni \Smat^1 \Quni^T)^{-1} - I]
\Quni \frac{1}{\relaxn} \sahand^1$ is independent of $\Quni$. 
For a given constant $c_3$, let us define the event
\begin{equation*}
  \mathcal{\Sail} \define \big \{ \norm{v}_2^2 \, \leq \, c_3^2 \;
  \frac{\spindex^2}{\numobs} \big [ \frac{\spindex}{\numobs} +
    \frac{1}{\relaxn^2 \numobs} \big ] \big \}.
\end{equation*}
We can then write
\begin{eqnarray*}
  \Prob[\norm{\Quni^T v}_{\infty} \geq \epsilon] & \leq & \Prob
  \big[\norm{\Quni^T v}_{\infty} \geq \epsilon \, \mid \, \Sail \big]
  + \Prob[\Sail^c].
\end{eqnarray*}
Note that we may consider the event that $\matsnorm{D}{2} = \order(1)$ and
$[2D(D + Q \Smat^1 Q^T)^{-1} - I] = \myorder$. We claim that each of these events happens with high probability. Note that the former event occurs with high
probability by Lemma~\ref{LemRandMat}. The latter event holds with
high probability since,
\begin{equation*}
  [2D(D + Q \Smat^1 Q^T)^{-1} - I] =
  [2D((D + Q \Smat^1 Q^T)^{-1}-I/2) + D - I].
\end{equation*}
and, both $\matsnorm{D-I}{2} = \myorder$ and
$((D + Q \Smat^1 Q^T)^{-1}-I/2) = \myorder$
by equation~\eqref{matnorm:1}. Thus, the sum of the two
random matrices is also $\myorder$.

Recall the bound on the variance of each component of $\sahand^1$
from equation~\eqref{sahand:varbound} and note that each
component is independent. Applying
the concentration results from Lemma~\ref{LemChiTail}
for $\chi$-squared random variables yields that
$\| \sahand^1 \|_2^2 \leq \frac{1}{4}(1+\delta)\frac{s^2}{n} + \frac{1}{2} \frac{s}{n \relaxn^2}$ with high probability. Hence, under the above conditions
\begin{eqnarray*}
  \| \frac{1}{2} [ 2 \Smatd (\Smatd + \Quni \Smat^1 \Quni^T)^{-1} - I]
  \Quni \frac{1}{\relaxn} \sahand^1\|_2^2 & \leq & \matsnorm{\| \frac{1}{2} [ 2 \Smatd (\Smatd + \Quni \Smat^1 \Quni^T)^{-1} - I]}{2}^2 \| \Quni \frac{1}{\relaxn} \sahand^1\|_2^2 \\
  & \leq & c_3^2 \;
\frac{\spindex^2}{\numobs} \big [ \frac{\spindex}{\numobs} +
  \frac{1}{\relaxn^2 \numobs} \big ] ,
\end{eqnarray*}
with high probabilty, which implies that $\Sail$ holds with high probability
as well. Therefore, it immediately follows then that
$\Prob[\Sail^c] \leq c_1 \exp(-c_2 \spindex)$.

It remains to control the first term.  We do
so using the following lemma, which is proved in
Appendix~\ref{AppMiniLemmaOne}:
%%%%%%%%%%%%%%%%%%%%%%%%%%%%%%%%%%%%%%%%%%%%%%%%%%
\begin{lemma}
\label{minilemma1}
Let $\Quni \in \real^{\mindex \times \mindex}$ be a matrix chosen
uniformly at random from the space of orthogonal matrices. Let
$\martin \in \real^\mindex$ be a random vector independent of $\Quni$,
such that $\|v\|_2 \leq \myvmax$ with probability one.  Then we have
\begin{equation*}
\Prob \big [ \norm{\Quni^T \martin}_{\infty} \geq 2 \, \myvmax
\sqrt{\frac{\log \mindex}{\mindex}} \big] \; = \; o(1).
\end{equation*}
\end{lemma}

We now apply this lemma to the random vector $v$ with $\mindex =
\spindex$, and $\myvmax = c_3 \; \frac{\spindex}{\sqrt{\numobs}} \;
\sqrt{ \frac{\spindex}{\numobs} + \frac{1}{\relaxn^2 \numobs}}$.
Note that
\begin{eqnarray*}
2 \myvmax \, \sqrt{\frac{\log \spindex}{\spindex}} & = & 2 c_3 \;
 \sqrt{\frac{\log \spindex}{\numobs}} \; \sqrt{
 \frac{\spindex}{\numobs} + \frac{1}{\relaxn^2 \numobs}} \; = \; o(1),
\end{eqnarray*}
from which the second claim~\eqref{lembound2} in Lemma~\ref{LemBound}
follows.

Finally, we turn to proving the third claim~\eqref{lembound3} in
Lemma~\ref{LemBound}.  Following some algebra, we obtain
\begin{equation}
\label{EqnDecomp}
  \| \Smat^1 \big[ \Smat^1 + \Smat^2 \big]^{-1} \vec{1} - \frac{1}{2}
  \|_{\infty} \, = \, \frac{1}{4} (\Smat^1 - \Smat^2) \vecone +
  \frac{1}{2} (\Smat^1 - \Smat^2)(I/2 - (\Smat^1 + \Smat^2)^{-1})
  \vecone.
\end{equation}
We diagonalize the matrix $\Smat^1 = \Quni^T \Smatd \Quni$, where
$\Smatd$ is diagonal. Since the random matrix $\Smat^1$ has a
spherically symmetric distribution, the matrix $\Quni$ has a uniform
distribution over the space of orthogonal matrices and is independent
of $\Smatd$. Using this decomposition, we can rewrite the second term
in equation~\eqref{EqnDecomp} as
\begin{equation}
\label{EqnIntermediate}
\frac{1}{2} \Quni^T (\Smatd - \Quni \Smat^2 \Quni^T) (\frac{I}{2} -
(\Smatd + \Quni \Smat^2 \Quni^T)^{-1}) \Quni \vecone \; = \; \Quni^T
\Specmat \Quni \vecone
\end{equation}
where $\Specmat \define \frac{1}{4} (\Smatd - \Quni \Smat^2 \Quni^T)
(I - 2 (\Smatd + \Quni \Smat^2 \Quni^T)^{-1})$.  We note that
$\Specmat$ is independent of $\Quni$, because $\Smatd$ and $\Smat^2$
are independent of $\Quni$.  This independence follows from the
spherical symmetry of $\Smat^2$ and the fact that $\Smat^2
\stackrel{d}{=} \Quni \Smat^2 \Quni^T$.

Defining the event \mbox{$\Tail \define \big \{ \matsnorm{\Specmat}{2}
\leq 4 \spindex/\numobs \big \}$}, we claim that
\begin{equation}
\label{EqnFirstClaim}
\mprob[\Tail^c] \; \leq \; c_1 \exp(-c_2 \numobs) \; \rightarrow \; 0.
  \end{equation}
In order to establish this claim, we note that sub-multiplicativity
and triangle inequality imply that
\begin{eqnarray*}
\matsnorm{\Specmat}{2} & \leq & \frac{1}{4} \matsnorm{\Smatd - \Quni
  \Smat^2 \Quni^T}{2} \; \matsnorm{(\Smatd + \Quni \Smat^2 \Quni^T)/2
  - I}{2}\; \matsnorm{2 (\Smatd + \Quni \Smat^2 \Quni^T)^{-1}}{2} \\ %
& \leq & 2 (\matsnorm{\Smatd-I}{2} + \matsnorm{I - \Quni \Smat^2
  \Quni^T}{2}) \matsnorm{(\Smatd + \Quni \Smat^2 \Quni^T)/2 - I}{2},
\end{eqnarray*}
since $\matsnorm{2 (\Quni^T \Smatd \Quni + \Quni \Smat^2
\Quni^T)^{-1}}{2} \leq 2$ with probability greater than $1-c_1
\exp(-c_2 \numobs)$, from the discussion following
Lemma~\ref{LemRandMat}.  Similarly, from this same result, we have
$\order(\matsnorm{\Smatd-I}{2}) = \order(\matsnorm{I - \Quni \Smat^2
\Quni^T}{2}) = \order(\matsnorm{(\Smatd + \Quni \Smat^2 \Quni^T)/2 -
I}{2})\; \leq \; 2 \sqrt{\frac{\spindex}{\numobs}}$, so that the
claim~\eqref{EqnFirstClaim} follows.
  
Using the decomposition~\eqref{EqnIntermediate} and the tail
bound~\eqref{EqnFirstClaim}, we have
\begin{eqnarray*}
\mprob[\norm{\Quni^T \Specmat \Quni 1}_{\infty} \geq \epsilon] & = &
\mprob[ \norm{\Quni^T \Specmat \Quni 1}_{\infty} \geq \epsilon \; \mid
\; \Tail] + \mprob[\Tail^c] \\
& \leq & \order \big ( \, \frac{1}{s} \, \big ) +
\order(\exp(-c(\epsilon) \numobs)),
\end{eqnarray*}
where Lemma~\ref{biglemma1} (proved in Appendix~\ref{AppBigLemmaOne})
provides control on the first term in the inequality.

\subsection{Proof of Lemma~\ref{biglemma1}}
\label{AppBigLemmaOne}

We provide the proof for part (a) of the Lemma and note that
part (b) is analogous.

\noindent By union bound, we have
\begin{eqnarray*}
\mprob[\norm{\Quni^T \Amat \Quni x}_{\infty} \geq \epsilon] & \leq &
\spindex \; \max_{i=1, \ldots, \spindex} \mprob[ | e_i^T \Quni^T \Amat
\Quni x| \geq \epsilon].
\end{eqnarray*}
We will derive a bound on the probability $\mprob[ | e_1^T \Quni^T
\Amat \Quni x| \geq \epsilon]$ that holds for all $e_i, i = 1, \ldots,
\spindex$.  We write $e_1^T \Quni^T \Amat \Quni x = x_1 v_1^T \Amat
v_1 + v_1^T \Amat v_2$, where $v_1$ denotes the first column of
$\Quni$, and $v_2 = \sum_{k=2}^\spindex x_k \Quni_k$ denotes the
weighted sum of the remaining $(k-1)$ columns of $\Quni$.  Since
$\Quni$ is orthogonal, the vector $v_1$ has unit norm $\|v_1\|_2 = 1$,
the vector $v_2$ is orthogonal to $v_1$, and moreover $\norm{v_2}_2^2
\leq \|x\|_\infty^2 \spindex-1$.  Owing to the bound on the spectral
norm of $\Amat$, we have
\begin{eqnarray*}
|x_1 v_1^T \Amat v_1 | & \leq \|x\|_\infty
 \sqrt{\frac{\spindex}{\numobs}}
\end{eqnarray*}
which is less than $\epsilon/2$ for $(\spindex, \numobs)$ sufficiently
large, since $\spindex/\numobs = o(1)$.

We now turn to the second term.  Note that conditioned on $v_2$, the
vector $v_1$ is uniformly distributed over an
$(\spindex-1)$-dimensional unit sphere, contained within the subspace
orthogonal to $v_2$.  Still conditioning on $v_2$, consider the
function $f(v_1) = v_1^T \Amat v_2$.  For any pair of vectors $v_1,
v_1'$ on the unit sphere, we have
\begin{eqnarray*}
|f(v_1) - f(v_1')|^2 & = & |(v_1 -v_1')^T \Amat v_2|^2 \\
& \leq & \matsnorm{\Amat}{2}^2 \; \|x\|_\infty^2 \, (\spindex-1) \|v_1 -
v_1'\|^2_2 \\
& = & \matsnorm{\Amat}{2}^2 \; \|x\|_\infty^2 \, (\spindex-1) \;
\big[2 \, (1 - \cos(d(v_1, v_1'))) \big],
\end{eqnarray*}
where $d = \arccos (v_1^T v_1')$ is the geodesic distance.  Using the
inequality $\cos(d) \geq 1- d^2/2$, valid for $d \in [0, \pi]$, and
the assumption $\matsnorm{\Amat}{2} \leq \sqrt{\spindex/\numobs}$, and
taking square roots, we obtain
\begin{eqnarray*}
  |f(v_1) - f(v_1')| & \leq & \sqrt{\frac{\spindex}{\numobs}} \; \|x
\|_\infty \; \sqrt{(\spindex-1)} \; d(v_1, v'_v),
\end{eqnarray*}
so that $f$ is a Lipschitz constant on the unit sphere (with dimension
$\spindex -1$) with constant $L = \|x\|_\infty
\sqrt{\frac{\spindex}{\numobs} \; (\spindex-1)}$.  Consequently, by
Levy's theorem~\cite{Ledoux01}, for any $\epsilon > 0$, we have
\begin{eqnarray*}
\mprob[|f(v_1)| \geq \epsilon] & \leq & 2 \exp(- (\spindex-2) \;
{\frac{\numobs}{\|x\|_\infty^2 \; \spindex (\spindex-1)}} \;
\epsilon^2) \; \leq \; 2 \, \exp \big( -c_1 \,
\frac{\numobs}{\|x\|_\infty^2 \; \spindex} \epsilon^2 \big).
\end{eqnarray*}

As a final side remark, we note that under the scaling of
Theorem~\ref{ThmPhase}(b), we have \mbox{$\frac{\numobs}{\spindex}
\epsilon^2 - \log(\spindex) \to \infty$} as $\numobs \to \infty$, so
that the probability in question vanishes.

%%%%%%%%%%%%%%%%%%%%%%%%%%%%%%%%%%%%%%%%%%%%%%%%%%%%%%%%%%%%%%%%%%%%

\subsection{Proof of Lemma~\ref{minilemma1}}
\label{AppMiniLemmaOne}
By union bound and symmetry of the distribution $\Quni$, for any $t >
0$, we have
\begin{eqnarray*}
\Prob \big [ \norm{\Quni^T \martin}_{\infty} \, \geq \, t \big ] &
\leq & \mindex \; \Prob \big [ \abs{e_1^T \Quni^T \martin} \, \geq \, t
\big] \\
& = &  \mindex  \; \Prob \big [ \abs{\qvec_1^T \martin} \, \geq \, t
\big],
\end{eqnarray*}
where $\qvec_1$ is the first column of $\Quni$.  Note that $\qvec_1$
is a random vector distributed uniformly over the unit sphere
$S^{\mindex-1}$ in $\mindex$ dimensions.  Viewing the vector $\martin
\in \real^m$ as fixed, consider the function $f(\qvec) = \qvec^T
\martin$ defined over $S^{\mindex-1}$. As in Lemma~\ref{biglemma1},
some calculation shows the Lipschitz constant of $g$ over
$S^{\mindex-1}$ is at most \mbox{$L = \|v\|_2$.}  Applying Levy's
theorem~\cite{Ledoux01}, we conclude that for any $\epsilon > 0$,
\begin{eqnarray*}
\mindex \; \Prob[\abs{f(\qvec_1)} \geq t] & \leq & 2 \exp \big (
-(\mindex - 1) \, \frac{t^2} {2 \|v\|_2^2} + \log \mindex \big).
\end{eqnarray*}
Since $\|v\|_2 \leq \myvmax$ by assumption, it suffices to set $t = 2
\myvmax \sqrt{\frac{\log \mindex}{\mindex}}$.

%\end{proof}
%%%%%%%%%%%%%%%%%%%%%%%%%%%%%%%%%%%%%%%%%%%%%%%%%%%%%%%%%%%%%%%%%%%%%%%%%%%%

%%%%%%%%%%%%%%%%%%%%%%%%%%%%%%%%%%%%%%%%%%%%%%%%%%%%%%%%%%%%%%%%%%%%%%%

\section{Some large deviation bounds}
\label{AppLargeDev}

In this appendix, we state some known large deviation bounds for the
Gausssian variates, $\chi^2$-variates, as well as the eigenvalues of
random matrices.  The following Gaussian tail bound is standard:
\begin{lemma}
\label{LemGaussTail}
For a Gaussian variable $Z \sim N(0, \sigma^2)$, for all $t > 0$,
\begin{eqnarray}
\mprob[|Z| \geq t] & \leq & 2 \exp \big(- \frac{t^2}{2 \sigma^2}
\big).
\end{eqnarray}
\end{lemma}
\noindent The following tail bounds on chi-squared variates are also
useful:
\begin{lemma}
\label{LemChiTail}
Let $X$ be a $\chi$-squared random variable with $d$ degrees of
freedom.  Then for all $t > 0$, we have
\begin{subequations}
\label{EqnMyChi}
\begin{eqnarray}
\label{EqnMyChiUpper}
\Prob[ \frac{X}{d} \geq (1 + t)^2] & \leq & \exp(- \frac{d t^2}{2}),
\qquad \mbox{and} \\
\label{EqnMyChiLower}
  \Prob[\frac{X}{d} \leq (1 - 2 t)] & \leq & \exp(- d t^2).
\end{eqnarray}
\end{subequations}
\end{lemma}
\begin{proof}
These tail bounds are immediate consequences of results due to Laurent
and Massart~\cite{LauMas98}, who prove that for all $x > 0$, we have
\begin{subequations}
\begin{eqnarray}
\label{EqnMassOne}
\Prob[X \geq x + (\sqrt{x} + \sqrt{d})^2] & \leq & \exp(-x), \quad
\mbox{and} \\
\label{EqnMassTwo}
\Prob(X - d \leq -2 \sqrt{d x}) & \leq & \exp(-x).
\end{eqnarray}
\end{subequations}
Letting $x = d t^2/2$ in equation~\eqref{EqnMassOne}, we have
\begin{eqnarray*}
\exp(- \frac{d t^2}{2}) & \geq & \Prob[ \frac{X}{d} \geq \sqrt{2} t +
  1 + t^2] \; \geq \; \Prob[ \frac{X}{d} \geq (1 + t)^2],
\end{eqnarray*}
thereby establishing~\eqref{EqnMyChiUpper}.  With the same choice of
$x$, equation~\eqref{EqnMassTwo} implies the
bound~\eqref{EqnMyChiLower} immediately.
\end{proof}

Finally, the following type of large deviations bound on the
eigenvalues of Gaussian random matrices is standard
(e.g.,~\cite{DavSza01}):
\begin{lemma}
\label{LemRandMat}
Let $X \in \real^{\numobs \times \spindex}$ be a random matrix from
the standard Gaussian ensemble (i.e., $X_{ij} \sim N(0,1)$, i.i.d).
Then with probability greater than $1 - c_1 \exp(-c_2 \numobs)$, for
any $\delta > 0$, its eigenspectrum satisfies the bounds
  \begin{equation*}
(1-\delta) \; \Big[1- \sqrt{\frac{\spindex}{\numobs}}\Big]^2 \; \leq
    \; \Lambda_{\min} \big ( \frac{X^T X}{\numobs} \big ) \leq
    \Lambda_{\max} \big ( \frac{X^T X}{\numobs} \big ) \; \leq \;
    (1+\delta) \; \Big[1 + \sqrt{\frac{\spindex}{\numobs}}\Big]^2.
  \end{equation*}
\end{lemma}
Note that this lemma implies similar bounds for eigenvalues of the
inverse:
\begin{equation*}
\frac{1}{(1+\delta) \; \big[1 +
\sqrt{\frac{\spindex}{\numobs}}\big]^2}\; \leq \; \Lambda_{\min}
\big ( ( \frac{X^T X}{\numobs})^{-1} \big ) \leq \Lambda_{\max}
\big ( (\frac{X^T X}{\numobs})^{-1} \big ) \; \leq \;
\frac{1}{(1-\delta) \; \big[1-
\sqrt{\frac{\spindex}{\numobs}}\big]^2}.
\end{equation*}

From the above two sets of inequalities, we conclude for
$\spindex/\numobs \leq 1$, we have with probability greater than $1 -
c_1 \exp( - c_2 \numobs)$
\begin{subequations}
\label{standardmat}
\begin{eqnarray}
\label{matnorm:1}
\matsnorm{ \frac{1}{\numobs} \mata^T \mata - I}{2} & \leq &
4 \sqrt{\frac{\spindex}{\numobs}}, \quad \mbox{and} \\
\label{matnorm:2}
\matsnorm{ (\frac{1}{\numobs} \mata^T \mata)^{-1} - I}{2} & \leq & 4
\sqrt{\frac{\spindex}{\numobs}}.
\end{eqnarray}
\end{subequations}

\newcommand{\eigmin}{\ensuremath{\lambda_{\operatorname{min}}}}
\newcommand{\eigmax}{\ensuremath{\lambda_{\operatorname{max}}}}

For random matrices where each row is distributed $N(0,\Covmat)$ and
$\Lambda_{\min}(\Covmat) > \mineig_{\min}$ and
$\Lambda_{\max}(\Covmat) \leq \mineig_{\max}$, we have
\begin{subequations}
\label{generalmat}
\begin{eqnarray}
\label{genmatnorm:1}
\matsnorm{ \frac{1}{\numobs} \mata^T \mata - \Covmat}{2} & \leq &
\eigmax(\Covmat) 4 \sqrt{\frac{\spindex}{\numobs}}, \quad \mbox{and}\\
\label{genmatnorm:2}
\matsnorm{ (\frac{1}{\numobs} \mata^T \mata)^{-1} - \Covmat^{-1}}{2} &
\leq & \frac{4}{\eigmin(\Covmat)} \; \sqrt{\frac{\spindex}{\numobs}}.
\end{eqnarray}
\end{subequations}

%%%%%%%%%%%%%%%%%%%%%%%%%%%%%%%%%%%%%%%%%%%%%%%%%%%%%%%%%%%%%%%%%%%%%%%%

%\bibliographystyle{unsrt}
%\bibliographystyle{plain} 
%\bibliography{project,sahand_blocksparse}
%\bibliography{updated_sahand}
%%%%%%%%%%%%%%%%%%%%%%%%%%%%%%%%%%%%%%%%%%%%%%%%%%%%%%%%%%%%%%%%%%%%%%%%%

%%%%%%%%%%%%%%%%%
\end{document}